\documentclass[a4paper,10pt]{amsart}
\usepackage{amsmath,amsthm,amssymb,latexsym,enumerate,xcolor,hyperref}
\usepackage{graphicx}
\usepackage[margin=2.5cm]{geometry}

\numberwithin{equation}{section}

\begin{document}

\newtheorem{thm}{Theorem}[section]
\newtheorem{prop}[thm]{Proposition}
\newtheorem{lem}[thm]{Lemma}
\newtheorem{cor}[thm]{Corollary}
\newtheorem{rem}[thm]{Remark}
\newtheorem*{defn}{Definition}

\newtheorem{definit}[thm]{Definition}
\newtheorem{setting}{Setting}
\renewcommand{\thesetting}{\Alph{setting}}

\newcommand{\DD}{\mathbb{D}}
\newcommand{\NN}{\mathbb{N}}
\newcommand{\ZZ}{\mathbb{Z}}
\newcommand{\QQ}{\mathbb{Q}}
\newcommand{\RR}{\mathbb{R}}
\newcommand{\CC}{\mathbb{C}}
\renewcommand{\SS}{\mathbb{S}}

\renewcommand{\theequation}{\arabic{section}.\arabic{equation}}

\newcommand{\supp}{\mathop{\mathrm{supp}}}    

\newcommand{\re}{\mathop{\mathrm{Re}}}   
\newcommand{\im}{\mathop{\mathrm{Im}}}   
\newcommand{\dist}{\mathop{\mathrm{dist}}}  
\newcommand{\link}{\mathop{\circ\kern-.35em -}}
\newcommand{\spn}{\mathop{\mathrm{span}}}   
\newcommand{\ind}{\mathop{\mathrm{ind}}}   
\newcommand{\rank}{\mathop{\mathrm{rank}}}   
\newcommand{\Fix}{\mathop{\mathrm{Fix}}}   
\newcommand{\codim}{\mathop{\mathrm{codim}}}   
\newcommand{\conv}{\mathop{\mathrm{conv}}}   
\newcommand{\epsi}{\mbox{$\varepsilon$}}
\newcommand{\eps}{\mathchoice{\epsi}{\epsi}
{\mbox{\scriptsize\epsi}}{\mbox{\tiny\epsi}}}
\newcommand{\cl}{\overline}
\newcommand{\pa}{\partial}
\newcommand{\ve}{\varepsilon}
\newcommand{\zi}{\zeta}
\newcommand{\Si}{\Sigma}
\newcommand{\cA}{{\mathcal A}}
\newcommand{\cG}{{\mathcal G}}
\newcommand{\cH}{{\mathcal H}}
\newcommand{\cI}{{\mathcal I}}
\newcommand{\cJ}{{\mathcal J}}
\newcommand{\cK}{{\mathcal K}}
\newcommand{\cL}{{\mathcal L}}
\newcommand{\cN}{{\mathcal N}}
\newcommand{\cR}{{\mathcal R}}
\newcommand{\cS}{{\mathcal S}}
\newcommand{\cT}{{\mathcal T}}
\newcommand{\cU}{{\mathcal U}}
\newcommand{\OM}{\Omega}
\newcommand{\B}{\bullet}
\newcommand{\ol}{\overline}
\newcommand{\ul}{\underline}
\newcommand{\vp}{\varphi}
\newcommand{\AC}{\mathop{\mathrm{AC}}}   
\newcommand{\Lip}{\mathop{\mathrm{Lip}}}   
\newcommand{\es}{\mathop{\mathrm{esssup}}}   
\newcommand{\les}{\mathop{\mathrm{les}}}   
\newcommand{\nid}{\noindent}
\newcommand{\pzr}{\phi^0_R}
\newcommand{\pir}{\phi^\infty_R}
\newcommand{\psr}{\phi^*_R}
\newcommand{\pow}{\frac{N}{N-1}}
\newcommand{\ncl}{\mathop{\mathrm{nc-lim}}}   
\newcommand{\nvl}{\mathop{\mathrm{nv-lim}}}  
\newcommand{\la}{\lambda}
\newcommand{\La}{\Lambda}    
\newcommand{\de}{\delta}    
\newcommand{\fhi}{\varphi} 
\newcommand{\ga}{\gamma}    
\newcommand{\ka}{\kappa}   

\newcommand{\core}{\heartsuit}
\newcommand{\diam}{\mathrm{diam}}

\newcommand{\lan}{\langle}
\newcommand{\ran}{\rangle}
\newcommand{\tr}{\mathop{\mathrm{tr}}}
\newcommand{\diag}{\mathop{\mathrm{diag}}}
\newcommand{\dv}{\mathop{\mathrm{div}}}

\newcommand{\al}{\alpha}
\newcommand{\be}{\beta}
\newcommand{\Om}{\Omega}
\newcommand{\na}{\nabla}

\newcommand{\cC}{\mathcal{C}}
\newcommand{\cM}{\mathcal{M}}
\newcommand{\nr}{\Vert}
\newcommand{\De}{\Delta}
\newcommand{\cX}{\mathcal{X}}
\newcommand{\cP}{\mathcal{P}}
\newcommand{\om}{\omega}
\newcommand{\si}{\sigma}
\newcommand{\te}{\theta}
\newcommand{\Ga}{\Gamma}

\newcommand{\vV}{\mathbf{v}}
\newcommand{\lbunu}{\ul{m}}
\newcommand{\ca}{\tilde{a}}
\newcommand{\Vve}{\ul{\varepsilon}}

\title[Soap Bubbles and convex cones: optimal quantitative rigidity]{Soap Bubbles and convex cones: optimal quantitative rigidity}

\author{Giorgio Poggesi}
\address{Department of Mathematics and Statistics, The University of Western Australia, 35 Stirling Highway, Crawley, Perth, WA 6009, Australia}
%
%
\email{giorgio.poggesi@uwa.edu.au}


\begin{abstract}
We consider a class of rigidity results in a convex cone $\Si \subseteq \RR^N$. These include overdetermined Serrin-type problems for a mixed boundary value problem relative to $\Si$, Alexandrov's soap bubble-type results relative to $\Si$, and a Heintze-Karcher's inequality relative to $\Si$. Each rigidity result is obtained by means of a single integral identity and holds true under weak integral conditions.
Optimal quantitative stability estimates are obtained in terms of an $L^2$-pseudodistance. 
In particular, the optimal stability estimate for Heintze-Karcher's inequality is new even in the classical case $\Si = \RR^N$.

Stability bounds in terms of the Hausdorff distance are also provided.


Several new results are established and exploited, including a new Poincar\'e-type inequality for vector fields whose normal components vanish on a portion of the boundary and an explicit (possibly weighted) trace theory -- relative to the cone $\Si$ -- for harmonic functions satisfying a homogeneous Neumann condition on the portion of the boundary contained in $\pa\Si$.

We also introduce new notions of uniform interior and exterior sphere conditions relative to the cone $\Si\subseteq\RR^N$, which allow to obtain (via barrier arguments) uniform lower and upper bounds for the gradient in the mixed boundary value-setting. In the particular case $\Si=\RR^N$, these conditions return the classical uniform interior and exterior sphere conditions (together with the associated classical gradient bounds of the Dirichlet setting).
\end{abstract}

\keywords{Heintze-Karcher's inequality, Alexandrov's Soap Bubble Theorem, Serrin's overdetermined problem, convex cones, symmetry, rigidity, integral identities, stability, quantitative estimates}
\subjclass{Primary 35N25, 53A10, 35B35; Secondary 35A23}

\maketitle

\raggedbottom

\section{Introduction}

Given a convex cone $\Si$ in $\RR^N$ with vertex at the origin, we consider a bounded domain (i.e., a bounded connected open set) $\Si \cap \Om$  -- where $\Om$ is a smooth bounded domain in $\RR^N$ --
such that its boundary relative to the cone $\Ga_0:= \Si \cap \pa\Om$ is smooth,  while $\pa \Ga_0$ is a $N-2$-dimensional manifold and $\pa(\Si\cap\Om)\setminus \Ga_0$ is smooth enough outside a singular set $\mathcal{S} \subset \pa \Si$ of finite $\ell$-dimensional upper Minkowski content for some $0 \le \ell \le N-2$.
%
%
Set $\Ga_1 :=  \pa(\Si\cap\Om) \setminus \left( \ol{\Ga}_0 \cup \ol{\cS} \right)$ and denote with $\nu$ the (exterior) unit normal vector field to $\Ga_0 \cup \Ga_1$.
We consider the following mixed boundary value problem:
\begin{equation}\label{eq:problem torsion}
	\begin{cases}
		\De u = N \quad & \text{ in } \Si \cap \Om 
		\\
		u= 0 \quad & \text{ on } \Ga_0
		\\
		u_\nu=0 \quad & \text{ on } \Ga_1 .
	\end{cases}
\end{equation}
As in \cite{PT}, we assume that the solution $u$ of \eqref{eq:problem torsion} is of class $W^{1,\infty}(\Si\cap\Om) \cup W^{ 2,2 }(\Si\cap\Om)$; such an assumption can be viewed as a gluing condition, and, as proved in \cite{PT} for cones smooth outside of the vertex, it is surely satisfied if $\ol{\Ga}_0$ and $\pa\Si$ intersect orthogonally. 
%
%
For details on the setting, see Setting \ref{Setting} and Remark \ref{rem:alternative settings}.

In the spirit of \cite{MP, MP2, Pog2}, we start our analysis providing a {\it fundamental integral identity} for each problem addressed by the present paper. Each fundamental identity encodes all the information for the relevant problem and provides an elegant short proof of the relevant rigidity result.

Setting
$$R=\frac{N\,|  \Si\cap\Om |}{|\Ga_0|} ,$$
we have the following fundamental identity for Serrin's problem in $\Si$:
\begin{equation*}
	\int_{ \Si \cap \Om} (-u) \left\{ |\na ^2 u|^2- \frac{ (\De u)^2}{N} \right\} dx + \int_{\Ga_1} u \,  \langle \na^2 u \na u, \nu \rangle \, dS_x =
	\frac{1}{2}\,\int_{\Ga_0} \left( u_\nu^2- R^2\right) (u_\nu- \langle x-z , \nu \rangle )\,dS_x ,
\end{equation*}
for any fixed $z \in \RR^N$ such that
\begin{equation}\label{eq:INTRO_inner product z in cone}
	\langle x-z ,\nu \rangle = 0 \quad \text{ for any } x \in \Ga_1 .
\end{equation}
Denoting with $H$ the mean curvature of $\Ga_0$, we have the following fundamental identity for Alexandrov's Soap Bubble Theorem in $\Si$:
\begin{equation*}
	\begin{split}
		\frac1{N-1} \left\lbrace
		\int_{ \Si \cap \Om } \left\{ |\na ^2 u|^2-\frac{(\De u)^2}{N}\right\}dx
		+
		\int_{\Ga_1} \left( - \langle \na^2 u \na u , \nu \rangle  \right) \, dS_x 
		\right\rbrace
		& +
		\frac1{R}\,\int_{\Ga_0} (u_\nu-R)^2 dS_x 
		\\
		& = 
		\int_{\Ga_0} \left( \frac{1}{R} -H \right) \, (u_\nu)^2 dS_x .
	\end{split}
\end{equation*}

Finally, if $\Ga_0$ is mean-convex, that is $H\ge 0$, we have the following fundamental identity for Heintze-Karcher's identity in $\Si$:
\begin{equation*}
	\begin{split}
		\frac1{N-1} \left\lbrace \int_{\Si \cap \Om} \left\{ |\na ^2 u|^2-\frac{(\De u)^2}{N}\right\} \, dx +  \int_{\Ga_1} \left( - \langle \na^2 u \na u , \nu \rangle  \right) \, dS_x \right\rbrace
		& + \int_{\Ga_0} \frac{(1-H\,u_\nu)^2}{H}\,dS_x 
		\\
		& =
		\int_{\Ga_0} \frac{dS_x}{H}-N |\Si\cap\Om|.
	\end{split}
\end{equation*}

The common feature of the fundamental identities above is that their left-hand sides consist of non-negative summands and that the {\it Cauchy-Schwarz deficit} of $\na^2 u$, that is $|\na^2 u| - ( \De u )^2 /N$, appears as integrand in one of those summands. Such a Cauchy-Schwarz deficit plays the role of spherical symmetry detector. In fact (see Lemma \ref{lem:sphericaldetector}), it is non-negative and equals zero if and only if $u$ is a quadratic polynomial of the form $\frac{|x-z|^2 - R^2}{2}$ for some $z$ satisfying $ \langle x-z , \nu \rangle =0$ on $\Ga_1$; in such a case the spherical shape of $\Ga_0$ ensues easily taking into account the Dirichlet boundary condition on $\Ga_0$, while the condition $\langle x-z , \nu \rangle =0$ on $\Ga_1$ tells that $\ol{\Ga}_0$ and $\pa\Si$ intersect orthogonally ($\cH^{N-2}$-a.e. in $\pa\Ga_0$).
That all the summands at the left-hand sides of the fundamental identities are non-negative easily follows by noting that $u \le 0$ in $\ol{\Si\cap\Om}$ by comparison (Lemma \ref{lem:comparison in cones}) and that $\langle \na^2u \na u , \nu \rangle \le 0$ on $\Ga_1$ by convexity of the cone $\Si$ 
(see \eqref{eq:disuguaglianza puntuale gratis per convex}).
We also notice that $u<0$ in $\Si\cap\Om$ (see, e.g., \eqref{eq:relation u dist general})  -- more precisely, $u<0$ in  $(\Si\cap\Om) \cup \Ga_1$ (see, e.g., \eqref{eq:non serve ma serve inproof reldist finer}).

Thus, denoting with $B_R(z)$ the ball of radius $R$ centered at $z$, the rigidity statement
$$
\Si\cap\Om = \Si\cap B_R(z) \quad \text{ and } \quad u(x)= \frac{1}{2} \, (|x-z|^2 - R^2)
\quad
\text{for some $z$ satisfying \eqref{eq:INTRO_inner product z in cone}},
$$
is easily obtained whenever the integral at the right-hand side in any of our fundamental identities is assumed to be non-negative. In particular, this is the case if the classical pointwise Serrin's overdetermined condition $u_\nu=const.$ is assumed on $\Ga_0$, or if $H$ is assumed to be constant on $\Ga_0$ (under a gluing condition on $\pa \Ga_0$), or if the equality sign holds true in the Heintze-Karcher's inequality relative to $\Si$, i.e., $	\int_{\Ga_0} \frac{dS_x}{H} \ge N |\Si\cap\Om|$ (which immediately follows from our last fundamental identity); see Theorems \ref{thm:rigidity Serrin}, \ref{thm:SBT cone}, and \ref{thm:heintze-karcher}.

We stress that such a rigidity is achieved here under weak integral type assumptions and in possibly non-smooth cones. 
	Rigidity under the stronger classical pointwise assumptions $u_\nu=const.$ on $\Ga_0$ or $H=const.$ on $\Ga_0$ (under a gluing condition on $\pa\Ga_0$) in cones smooth outside of their vertex was obtained in \cite{PT}.  Heintze-Karcher's inequality in smooth cones was treated in \cite{CP} (see also \cite{JWXZ} for a recent improvement of the inequality for capillary hypersurfaces in the half-space or in a wedge).
 
We also point out that convexity of the cone is not necessary to get rigidity, as rigidity can still be obtained under a weaker integral assumption, namely 
$
\int_{\Ga_1} \langle \na^2u \na u , \nu \rangle \, dS_x \le 0
$
for the Soap Bubble Theorem and Heintze-Karcher's inequality, and 
$
\int_{\Ga_1} u \,  \langle \na^2u \na u , \nu \rangle \, dS_x \ge 0 
$
for Serrin's problem; see Remark \ref{rem:nota su convexity non necessaria}. Similar observations also hold for the stability results presented in what follows (see Remark \ref{rem:nota su convexity non necassaria per Stability}). On the other hand, we point out that conterexamples to rigidity in nonconvex cones are provided in \cite{IPT}.


\bigskip

	In our setting, where non-smooth cones are allowed, characterizing the location of $z$ is a delicate issue.
	Of course, being as $\Si$ a cone with vertex at the origin, the choice $z=0$ always provides a solution, but this may not be the only possibility.
	In general (up to a change of coordinates), every convex cone is of the form $\Si = \RR^{l} \times \Tilde{\Si}$ for an integer $l \in \left[0,N\right]$, where $\Tilde{\Si}$ is a convex cone containing no lines. It is clear that, for cones of this form setting $\Om$ to be any ball centered at any point of the form $z=(z', 0_{\RR^{N-l}})$ with $z' \in \RR^l$ gives that $\Si\cap\Om$ provides a solution to all of our overdetermined-type problems\footnote{The overdetermined-type problems we are referring to consist in asking that the solution $u \in W^{1,\infty} (\Si\cap\Om) \cap W^{2,2}(\Si\cap\Om)$ of \eqref{eq:problem torsion} satisfies one of the following three overdeterminations:
	\begin{itemize}
		\item $u_\nu=const.$ on $\Ga_0$, 
		\item$H=const.$ on $\Ga_0$, 
		\item $\int_{\Ga_0} \frac{dS_x}{H} = N |\Si\cap\Om|$ (which is equivalent to $u_\nu=1/H$ on $\Ga_0$).
	\end{itemize}
}.

 These are also (the only) minimizers of the isoperimetric problem in convex cones pioneered by Lions and Pacella \cite{LP} (see also \cite{FI, BF, PT isoperimetric, RR, CGPRS, DPV2, CirLi, CavMan}). However, in contrast with the isoperimetric problem in convex cones, these may not be the only possible solutions to our overdetermined-type problems. For instance, if the boundary $\pa\Si$ of the cone has flat portions, any half ball lying on a flat portion of $\pa\Si$ provides a solution to our overdetermined-type problems.
 In fact, for a cone $\Si$ smooth outside of its vertex $0$, a complete characterization is given by the sets $\Si\cap B_R (z)$ where either $z=0$ or $z\in\pa\Si\setminus \left\lbrace 0 \right\rbrace$ and $\Ga_0$ is a half sphere lying over a flat portion of $\pa\Si$; this follows from the arguments in \cite{PT, RR}.
 However, since our analysis allows non-smooth cones this characterization does not apply. In fact, in general, if the boundary $\pa\Si$ of the cone possesses points $z$ such that $\pa\Si \cap \ol{B}_R(z)$ is contained in the boundary of a (possibly different) cone with vertex at $z$, then $\Si \cap B_R(z)$ provides another solution. For instance, in a pyramid-shaped cone in $\RR^3$ with flat faces $F_1$,$\dots$,$F_m$ such that $\bigcap_{i=1}^m F_i = \left\lbrace 0 \right\rbrace$, choosing $z$ on the $1$-dimensional intersection of two of those faces leads to additional solutions.

On the other hand, we notice that $z$ can be characterized in terms of the
linear space generated by the normal vector field $\nu (x)$ for $x \in \Ga_1$. In fact, being as $\Si$ a cone with vertex at the origin we have that $\langle x, \nu \rangle = 0$ on $\Ga_1 \subset \pa \Si$, and hence \eqref{eq:INTRO_inner product z in cone} is equivalent to $\langle z, \nu \rangle = 0$ on $\Ga_1$. That is, $z \in \left[ \mathrm{span}_{x\in \Ga_1} \nu(x) \right]^{\bot}$, where $\left[ \mathrm{span}_{x\in \Ga_1} \nu(x) \right]^{\bot}$ is the orthogonal complement in $\RR^N$ of the vector subspace $\mathrm{span}_{x\in \Ga_1} \nu(x) \subseteq \RR^N$.
%
%
In particular, 
\begin{equation}\label{eq:intro k=N}
	\text{dimension of } \, \, \mathrm{span}_{x \in \Ga_1} \nu(x) = N
\end{equation}
is a sufficient condition that guarantees that $z$ must be the origin.
Condition \eqref{eq:intro k=N} is surely verified if $\Ga_1$ contains at least a transversally nondegenerate point (see Proposition \ref{prop:nondegenerate point}). 
%
%
In particular, this is always the case if $\Si$ is a strictly convex cone (and $\Ga_1 \neq \varnothing$).

That a point of strict convexity was sufficient to force $z$ to be the origin
%
%
was noticed in \cite{PT isoperimetric}. The new condition in \eqref{eq:intro k=N} is more general and, as we are going to show, successfully applies to the study of the stability issue.

We start describing our quantitative estimates when \eqref{eq:intro k=N} is in force, in which case we obtain closeness of $\Si\cap\Om$ to $\Si\cap B_\rho(0)$, where $B_\rho(0)$ denotes the ball of radius $\rho>0$ centered at the origin.
To avoid excessive technicalities, the stability results are presented under the additional assumption that $\ol{\Ga}_0$ and $\pa\Si$ intersect in a Lipschitz way so that $\Si\cap\Om$ is a Lipschitz domain.

The crucial ingredient
	that allow to fix the center $z$ of the approximating ball in $\left[ \mathrm{span}_{x\in \Ga_1} \nu(x) \right]^{\bot}$ (that is the origin whenever \eqref{eq:intro k=N} is in force) 
	is the new Poincar\'e-type inequality (see Theorems \ref{thm:Poincare new RN}, \ref{thm:Poincare new in general}, \ref{thm:Strengthened Poincare new RN})
	\begin{equation}\label{eq:INTRO new Poicare caso particolare p=2}
		\nr \vV \nr_{L^p (\Si \cap \Om)} \le \eta_{p}( \Ga_1 , \Si\cap\Om)^{-1} \, \nr  \na \vV \nr_{L^p (\Si \cap \Om)} , 
	\end{equation}
	which holds true for any $\vV : \Si\cap\Om \to \mathrm{span}_{x \in \Ga_1 }\, \nu(x) \subseteq \RR^N$ such that $\vV \in W^{1,p}(\Si\cap\Om)$ and $\langle \vV , \nu \rangle = 0$ a.e. in $\Ga_1$.
	Such a Poincar\'e-type inequality seems to be new.
	
	Notice that, if we consider the function
	\begin{equation}\label{eq:INTRO def h}
		h := q - u , \quad \text{ where } q \text{ is the quadratic function defined as } q(x)=\frac12\, |x-z|^2 ,
	\end{equation}
	the choice $z = 0$ always guarantees that $\langle \na h , \nu \rangle = 0$ on $\Ga_1$, by the homogeneous Neumann condition $u_\nu=0$ on $\Ga_1$ and $\langle x,\nu \rangle =0$ on $\Ga_1\subset\pa\Si$; therefore, if $\mathrm{span}_{x \in \Ga_1 } \, \nu(x) = \RR^N$ the new Poincar\'e-type inequality can be applied with $\vV:=\na h$.
	%
	%

The quantitative estimates provided in the present paper include, as particular cases\footnote{Precise general statements will be provided later on in this Introduction.} and when \eqref{eq:intro k=N} is in force, the following sharp bounds for the $L^2$-pseudodistance of $\Si\cap\Om$ to $\Si \cap B_\rho(0)$, for some suitable choice of the radius $\rho>0$:
\begin{equation}\label{eq:intro k=N SBT}
	\nr |x| - \rho \nr_{L^2(\Ga_0)} \le C \, \nr H_0 - H \nr_{L^2(\Ga_0)} ,
\end{equation}
\begin{equation}\label{eq:intro k=N HK}
	\nr |x|^2 - \rho^2 \nr_{L^2(\Ga_0)} \le C \, \left( \int_{\Ga_0} \frac{dS_x}{H} - N | \Si\cap\Om | \right)^{1/2}.
\end{equation}

In \eqref{eq:intro k=N SBT} $\rho$ can be chosen to be equal to the reference constant $R=N|\Si\cap\Om|/|\Ga_0|$, whereas in \eqref{eq:intro k=N HK} we can set either
$\rho= \sqrt{ \int_{\Si\cap\Om}( |x|^2 - 2u ) \, dx /|\Si\cap\Om| }$ (as in the proof of Theorem \ref{thm:HK stability pseudodistance})
or $\rho= \sqrt{ \int_{\Ga_0} |x|^2 \, dS_x /|\Ga_0| }$ (see Remark \ref{rem:HK stability scelta rho}). 
In \eqref{eq:intro k=N SBT}, $H_0$ is a suitable reference constant for the mean curvature (see Theorem \ref{thm:INTRO_SBT stab Lipschitz costanti non geometriche} for details).
To avoid an excessively long treatment, we postpone the corresponding stability analysis for Serrin's overdetermined problem including the
sharp estimate $\nr |x| - R \nr_{L^2(\Ga_0)} \le C \, \nr u_\nu^2 -R^2 \nr_{L^2(\Ga_0)}$ to a forthcoming paper. The reason is that to treat Serrin's problem we will need to work in a weighted setting (due to the presence of the term $(-u)$ multiplying the Cauchy-Schwarz deficit of $\na^2u$ in the relevant fundamental identity displayed above) and additional machinery will be required.

\medskip

We stress that the closeness in $L^2$-pseudodistance obtained here is stronger than the closeness in terms of the so called asymmetry in measure. In fact, clearly the $L^2$-pseudodistance is stronger than the $L^1$- pseudodistance, being as
$$
\nr |x| - \rho \nr_{L^1(\Ga_0)} \le |\Ga_0|^{1/2} \nr |x| - \rho \nr_{L^2(\Ga_0)} ,
$$
by H\"older's inequality\footnote{In turn, the term $|\Ga_0|^{1/2} $ may be estimated in terms of the parameter $\lbunu$ (or $\ul{r}_i$) introduced next by means of $|\Ga_0|\le N|\Si\cap\Om|/ \lbunu$, being as
$N|\Si\cap\Om|=\int_{\Ga_0} u_\nu \, dS_x \ge \lbunu |\Ga_0|$.}. 
In turn, \cite[Proposition 6.1]{CGPRS} informs us that the $L^1$-pseudodistance is stronger than the asymmetry in measure, that is
%
%
$$
| (\Si\cap\Om) \De (\Si\cap B_\rho(0))| \lesssim \nr |x| - \rho \nr_{L^1(\Ga_0)} .
$$
Similar deductions are recurrent in the study of quantitative stability for the isoperimetric problem (see, e.g., \cite{FMP, FI}).

\bigskip


The constants in our quantitative estimates can be explicitly computed and estimated in terms of a few chosen geometrical parameters. At first, we obtain \eqref{eq:intro k=N HK} for an explicit constant $C$ only depending on $\eta_{2}(\Ga_1, \Si\cap\Om)$ and $\la_2 (\Ga_0)$, and \eqref{eq:intro k=N SBT} for an explicit constant $C$ only depending on $\eta_{2}(\Ga_1, \Si\cap\Om)$, $\la_2 (\Ga_0)$, and an upper bound for $\nr u_\nu \nr_{L^\infty (\Ga_0)}$. 
With $\la_2(\Ga_0)$ we denote the constant of the trace embedding $W^{1,2}(\Si \cap \Om) \hookrightarrow L^2(\Ga_0)$, that is, the best constant in the trace inequality
\begin{equation}\label{eq:trace inequality intro}
	\nr v \nr_{L^2(\Ga_0)} \le \la_2(\Ga_0) \, \nr  v \nr_{W^{1,2}(\Si \cap \Om)} ,
\end{equation}
which holds true for any $v \in W^{1,2}(\Si\cap\Om)$.

Next, we also show that \eqref{eq:intro k=N SBT} and \eqref{eq:intro k=N HK} hold for explicit constants depending on more explicit geometrical parameters. 
In fact, we are able to drop the dependence on the general trace embedding constant $\la_2 (\Ga_0)$ (appearing in the constants of \eqref{eq:intro k=N SBT} and \eqref{eq:intro k=N HK}) by proving ad-hoc 
trace-type inequalities for harmonic functions\footnote{Notice that the function $h$ defined in \eqref{eq:INTRO def h} is harmonic.} with explicit constants depending only on $\max_{\ol{\Si\cap\Om}}(-u)$ and a positive lower bound $\ul{m}$ for $u_\nu$ on $\Ga_0$. In turn, $\max_{\ol{\Si\cap\Om}}(-u)$ can be estimated only in terms of the diameter $d_{\Si\cap\Om}$ of $\Si\cap\Om$ (see Lemma \ref{lem:max stima con diametro}).
Moreover, we introduce the new notions of uniform interior and exterior sphere conditions relative to the cone $\Si$. These return the classical known uniform sphere conditions in the case $\Si =\RR^N$; when $\Si \subsetneq \RR^N$ they are related to how $\ol{\Ga}_0$ and $\pa\Si$ intersect (see Subsections \ref{subsec:Hopf} and \ref{subsec:upper bound gradient}). As in the classical case $\Si=\RR^N$, these conditions reveal to be useful tools to perform barrier arguments in the mixed boundary value setting for $\Si \subset \RR^N$.
In fact, we prove that the $\ul{r}_i$-uniform interior sphere condition relative to $\Si$ guarantees the validity of Hopf-type estimates -- in particular, we prove in Lemma \ref{lem:Hopf in cones} that $u_\nu \ge \ul{r}_i$ on $\Ga_0$ so that we can take $\lbunu:=\ul{r}_i$ -- whereas the  $\ul{r}_e$-uniform exterior sphere condition relative to $\Si$ allows us to obtain an explicit uniform upper bound for $|\na u|$ on $\Ga_0$ (Lemma \ref{lem:upper bound gradient on Ga0}), and hence also for $\nr \na u \nr_{L^\infty(\Si\cap\Om)}$ (Lemma \ref{lem:upper bound gradient on Ga0 = whole domain}), only in terms of $N$, $\ul{r}_e$, and the diameter $d_{\Si\cap\Om}$ of $\Si\cap\Om$.

Thanks to these results, whenever $\Si\cap\Om$ satisfies the $\ul{r}_i$-uniform interior sphere condition relative to $\Si$, we obtain 
%
%
\eqref{eq:intro k=N HK} with an explicit $C=C(\eta_{2}(\Ga_1, \Si\cap\Om), \ul{r}_i, d_{\Si\cap\Om})$; if $\Si\cap\Om$ also satisfies the $\ul{r}_e$-uniform exterior sphere condition relative to $\Si$, then we get \eqref{eq:intro k=N SBT} with an explicit $C=C(\eta_{2}(\Ga_1, \Si\cap\Om), \ul{r}_i, \ul{r}_e, d_{\Si\cap\Om})$.

We also point out that uniformity of $\eta_{p}(\Ga_1, \Si\cap\Om)$ in the class of uniformly bounded and uniformly Lipschitz domains $G$ such that $\Ga_1 \subset \ol{G}$ may be obtained following the argument in \cite[Theorem 2]{BouCha}.

\bigskip

Remarkably, our method is robust enough to give a complete characterization of the stability issue even in absence of the assumption \eqref{eq:intro k=N}. In fact, in general we can set
$$
k:= \text{dimension of } \, \, \mathrm{span}_{x \in \Ga_1} \nu(x) ,
$$ 
which may be any integer $0 \le k \le N$,
%
%
and obtain closeness of $\Si\cap \Om$ to $\Si\cap B_\rho(z)$ for some suitable point $z$ whose components in the $k$ directions spanned by $\mathrm{span}_{x \in \Ga_1} \nu(x)$ are set to be $0$.
%
%
General statements containing sharp stability estimates for the $L^2$-pseudodistance are presented in what follows.
%
%
Notice that the case $\Ga_1 = \emptyset $ is included in our treatment (in that case, we have $k=0$).

Up to changing orthogonal coordinates, we can assume that $\mathrm{span}_{x \in \Ga_1} \, \nu(x)$ is the space generated by the first $k$ axes $\mathbf{e}_1, \dots, \mathbf{e}_k$.
Notice that, in this way, if we set $z \in \RR^N$ of the form
\begin{equation}\label{eq:INTRO_choice of z 1of2}
	z=(0, \dots , 0 , z_{k+1}, \dots, z_k) \in \RR^N ,
\end{equation}
it surely satisfies \eqref{eq:INTRO_inner product z in cone}.
We also fix
\begin{equation}\label{eq:INTRO_choice of z 2of2}
	z_i = \frac{1}{|\Si\cap\Om|} \int_{\Si\cap\Om} ( x_i - u_i(x) ) \, dx \quad \text{for } i=k+1, \dots, N ,
\end{equation}
where $u_i$ denotes the $i$-th partial derivative of $u$ and $x_i$ the $i$-th component of the vector $x \in \RR^N$.
With this choice of $z$, if we consider the harmonic function $h$ defined in \eqref{eq:INTRO def h},
we have that
\begin{equation*}
	(h_1, \dots, h_k , 0 , \dots, 0) \in \mathrm{span}_{x \in \Ga_1} \, \nu(x) \subseteq \RR^N ,
	\quad \quad
	\langle (h_1, \dots, h_k , 0 , \dots, 0), \nu \rangle = \langle \na h , \nu \rangle  = 0 \, \text{ on } \Ga_1 
\end{equation*}
and
\begin{equation*}
	\int_{\Si \cap \Om} h_i \, dx = 0 \quad \text{for } i = k+1, \dots, N .
\end{equation*}
The identity $\langle \na h , \nu \rangle = 0$ on $\Ga_1$ easily follows by \eqref{eq:INTRO_inner product z in cone} and the
%
%
Neumann condition $u_\nu = 0$ on $\Ga_1$.

This will allow to use the new Poincar\'e inequality \eqref{eq:INTRO new Poicare caso particolare p=2} with $\vV := (h_1, \dots, h_k , 0 , \dots, 0)$ and the classical Poincar\'e inequality for functions with zero mean
\begin{equation*}
	\nr v \nr_{L^p(\Si\cap\Om)} \le \mu_{p} (\Si\cap\Om)^{-1} \nr  \na v  \nr_{L^p(\Si\cap\Om)}, 
	\quad
	\text{ for } v \in L^p (G) \cap W^{1,p}_{loc} (G)
	\text{ with }
	v_{\Si\cap\Om}=0,
\end{equation*}
with $v:=h_i$ for $i=k+1, \dots, N$.
%
%

Setting
\begin{equation*}
	\La_{p}(k) :=
	\begin{cases}
		\mu_{p}(\Si\cap\Om)^{-1} \quad & \text{if } k=0
		\\
		\eta_{p} (\Ga_1 , \Si \cap \Om )^{-1} \quad & \text{if } k=N
		\\
		\max\left[ \mu_{p}(\Si\cap\Om)^{-1} , \, \eta_{p} ( \Ga_1 ,  \Si \cap \Om )^{-1}  \right] \quad & \text{if } 1 \le k \le N-1 ,
	\end{cases}
\end{equation*}
we are now ready to present the sharp stability results for the $L^2$-pseudodistance.

\begin{thm}[Lipschitz stability for Alexandrov's Soap Bubble Theorem in terms of an $L^2$-psudodistance]
	\label{thm:INTRO_SBT stab Lipschitz costanti non geometriche}
	Given the point $z$ defined in \eqref{eq:INTRO_choice of z 1of2} and \eqref{eq:INTRO_choice of z 2of2}, assume that $\int_{\pa \Ga_0} \langle x -z ,\mathbf{n}_x\rangle d \cH^{N-2}_x \le 0$ (which, in particular, always holds true with the equality sign whenever $\ol{\Ga}_0$ and $\pa\Si$ intersect orthogonally) and set
	\begin{equation*}
		H_0 = \frac{1}{R} - \frac{\int_{\pa \Ga_0} \langle x-z,\mathbf{n}_x\rangle d \cH^{N-2}_x }{(N-1)N|\Si\cap\Om|} .
	\end{equation*}
	Then, we have that
	\begin{equation}\label{eq:INTRO_Stab Lipschitz SBT IN SECTION}
		\nr |x-z| - R \nr_{L^2(\Ga_0)} \le \widehat{C} \, \nr H_0 - H \nr_{L^2(\Ga_0)} ,
	\end{equation}
	where $\widehat{C}$ is some explicit constant only depending on $N$, $\la_2(\Ga_0)$, $\nr u_\nu \nr_{L^\infty(\Ga_0)}$, and $\La_{2}(k)$.

	If $\lbunu >0$ is a lower bound for $u_\nu$ on $\Ga_0$, then the dependence on $\la_2(\Ga_0)$ can be dropped, and \eqref{eq:INTRO_Stab Lipschitz SBT IN SECTION} holds true for some explicit constant $\widehat{C}$ that only depends on $N$, $\lbunu$, $\max\limits_{\ol{\Si\cap\Om}}(-u)$, $\nr u_\nu \nr_{L^\infty(\Ga_0)}$, and $\La_{2}(k)$.
	Here, the dependence on $\max\limits_{\ol{\Si\cap\Om}}(-u)$ can be replaced by the dependence on the diameter $d_{\Si\cap\Om}$. Moreover, whenever $\Si\cap\Om$ satisfies the $\ul{r}_i$-uniform interior sphere condition relative to $\Si$, we can take $\lbunu:= \ul{r}_i$.
	
	\smallskip
	
	If $\Si\cap\Om$ satisfies the $\ul{r}_i$-uniform interior and $\ul{r}_e$-uniform exterior sphere conditions relative to $\Si$, then \eqref{eq:INTRO_Stab Lipschitz SBT IN SECTION} holds true for some explicit constant $\widehat{C}$ that only depends on $N$, $\ul{r}_i$, $\ul{r}_e$, $d_{\Si\cap\Om}$, and $\La_{2}(k)$.
\end{thm}

\begin{thm}[Optimal stability for Heintze-Karcher's inequality in terms of an $L^2$-psudodistance]
	\label{thm:INTRO_HK stability pseudodistance}
	Let $\Ga_0$ be mean convex.
	Given the point $z$ defined in \eqref{eq:INTRO_choice of z 1of2} and \eqref{eq:INTRO_choice of z 2of2}, there exists a radius $\rho >0$ such that
	\begin{equation}\label{eq:INTRO_STIMA finale HK}
		\nr \frac{|x-z|^2 - \rho^2}{2} \nr_{L^2(\Ga_0)} \le \widehat{C} \, \left( \int_{\Ga_0} \frac{dS_x}{H} - N | \Si\cap\Om | \right)^{1/2} ,
	\end{equation}
	with
		$$
		\widehat{C}:= \sqrt{N-1} \,  \la_2(\Ga_0)^2  \left( 1 + \La_{2}(k)^{2} \right) .
		$$
%
		
		If $\lbunu >0$ is a lower bound for $u_\nu$ on $\Ga_0$, then the dependence on $\la_2(\Ga_0)$ can be dropped, and \eqref{eq:INTRO_STIMA finale HK} remains true with
		$$
		\widehat{C} := \frac{\sqrt{N-1}}{\lbunu}   \left( N \, \La_{2}(k)^{2} + 2 \max\limits_{\ol{\Si\cap\Om}} (-u) \right) .
		$$
		In turn, $\max\limits_{\ol{\Si\cap\Om}} (-u)$ can be explicitly estimated by $N$ and the diameter $d_{\Si\cap\Om}$ only. 
		
		Whenever $\Si\cap\Om$ satisfies the $\ul{r}_i$-uniform interior sphere condition relative to $\Si$, we can take $\lbunu:= \ul{r}_i$, and \eqref{eq:INTRO_STIMA finale HK} holds true for some explicit constant only depending on $N$, $\ul{r}_i$, $d_{\Si\cap\Om}$, and $\La_{2}(k)$.
\end{thm}

Let us comment on the constant $\La_{p}(k)$ appearing (with $p:=2$) in the statements above. By definition, this is equal to $\mu_{p} (\Si\cap\Om)^{-1}$ in the case $k=0$, $\eta_{p} (\Ga_1, \Si\cap\Om)^{-1}$ in the case $k=N$, and $\max\left[ \mu_{p}(\Si\cap\Om)^{-1} , \, \eta_{p} ( \Ga_1 ,  \Si \cap \Om )^{-1}  \right]$ if $1 \le k \le N-1$.
As explained in Remark \ref{rem:stime mu HS}, the Poincar\'e constant $\mu_{p} (\Si\cap\Om)^{-1}$ can be explicitly estimated in terms of the so-called John parameter and (either) the diameter (or the volume); in turn, the John parameter may be explicitly estimated in terms of the parameters of Lipschitz regularity associated to $\Si\cap\Om$ (see Remark \ref{rem:stima b_0}).
Uniformity of $\eta_{p}( \Ga_1 ,  \Si \cap \Om )^{-1}$, and hence also $\La_{p}(k)$, may be obtained in the class of uniformly bounded and uniformly Lipschitz domains $G$ such that $\Ga_1 \subset \ol{G}$, as already mentioned (see also Remark \ref{rem:casi k=0 e k=1 evitando eta}).
We also point out that, in the case $k=1$, \eqref{eq:INTRO new Poicare caso particolare p=2} reduces to a Poincar\'e-type inequality for (scalar) functions with zero trace on a subset of the boundary, and hence we are able to provide explicit estimates for $\eta_{p}( \Ga_1 ,  \Si \cap \Om )^{-1}$ (see Remark \ref{rem:casi k=0 e k=1 evitando eta}).

\medskip
  
In Section \ref{sec:Stronger stability rhoe - rhoi}, we also provide additional stability estimates for the stronger Hausdorff distance (in place of the $L^2$-pseudodistance used above). More precisely, we provide bounds for $\rho_e -\rho_i$ where $\rho_e= \max_{\ol{\Ga}_0} |x-z|$  and $\rho_i= \min_{\ol{\Ga}_0} |x-z|$; we stress that, even though a stronger measure of closeness to the symmetric configuration (i.e., $\rho_e -\rho_i$) is obtained in those results, we still allow the deviations on the right-hand side of our estimates to be measured by weak norms ($L^2$ or $L^1$), in the spirit of \cite{MP3,MP2,MP}. See Theorems \ref{thm:SBT stability rhoe rhoi}, \ref{thm:HK stability rhoe rhoi}.

\medskip

We point out that
different choices of the point $z$ can lead to alternative stability results. For instance, we may avoid using \eqref{eq:INTRO new Poicare caso particolare p=2} and hence completely remove the dependence on $\eta_{p}( \Ga_1 ,  \Si \cap \Om )^{-1}$ for any $0 \le k \le N$, at the cost of leaving the point $z$ free to have non-zero components also in the directions
spanned by $\nu$ on $\Ga_1$, as stated in the following.
\begin{thm}\label{thm:intro alternative z}
Setting
\begin{equation}\label{eq:Alternative choice z}
z= \frac{1}{|\Si\cap\Om|} \int_{\Si\cap\Om} (x -\na u) \, dx ,
\end{equation}
Theorems \ref{thm:INTRO_SBT stab Lipschitz costanti non geometriche} and \ref{thm:INTRO_HK stability pseudodistance} remain true with $\La_{2}(k)$ replaced simply by $\mu_{2}(\Si\cap\Om)^{-1}$.

In turn, $\mu_{2}(\Si\cap\Om)^{-1}$ may be explicitly estimated as in Remark \ref{rem:stime mu HS}.
\end{thm}

It is clear that, when $\Si=\RR^N$, Theorems \ref{thm:INTRO_SBT stab Lipschitz costanti non geometriche} and \ref{thm:INTRO_HK stability pseudodistance} return sharp stability results for the classical Alexandrov's Soap Bubble Theorem and Heintze Karcher's inequality in $\RR^N$. We provide an explicit statement of this in the next theorem. 
Notice that, when $\Si=\RR^N$, we have that $\Si\cap\Om=\Om$ is a smooth, say $C^2$, bounded domain in $\RR^N$.
Such a domain always satisfies the classical uniform interior and exterior sphere conditions in $\RR^N$.
Moreover, as already mentioned, when $\Si=\RR^N$ the uniform interior and exterior sphere conditions relative to $\Si$ reduce to the classical uniform interior and exterior sphere conditions in $\RR^N$.

Being as $\Si=\RR^N$ and hence $\Ga_1=\varnothing$, we have that $k=0$ and hence $\La_{p}(0)= \mu_{p}(\Si\cap\Om)^{-1}$. In turn, being as $\Om$ a $C^2$ domain, $\mu_{p}(\Si\cap\Om)^{-1}$ can be explicitly estimated in terms of $\ul{r}_i$ and $d_{\Si\cap\Om}$ only (see \cite[(iii) of Remark 2.4]{MP3}).

Finally, notice that when $\Si=\RR^N$ the choice of $z$ in \eqref{eq:INTRO_choice of z 1of2}-\eqref{eq:INTRO_choice of z 2of2} agrees with that in \eqref{eq:Alternative choice z}, and reduces to the center of mass of $\Om$ (see Subsection \ref{subsec:classical case Si=RN}).

\begin{thm}[Sharp stability for the classical Alexandrov's Soap Bubble Theorem and Heintze-Karcher's inequality in $\RR^N$]
	\label{thm:classical INTRO in R^N}
Let $\Om \subset \RR^N$, $N\ge2$, be a bounded domain of class $C^2$. 
Then, we have that
\begin{enumerate}[(i)]
	\item \eqref{eq:INTRO_Stab Lipschitz SBT IN SECTION} holds true for an explicit constant $\widehat{C}$ only depending on $N$, $\ul{r}_i$, $\ul{r}_e$ and $d_\Om$. If $\Ga_0=\pa \Om$ is mean convex, $\ul{r}_e$ can be removed;
	\item \eqref{eq:INTRO_STIMA finale HK} holds true for an explicit constant $\widehat{C}$ only depending on $N$, $\ul{r}_i$, and $d_\Om$. 
\end{enumerate}
\end{thm}

%
The sharp stability estimate for Heintz-Karcher's inequality is new even in the case $\Si=\RR^N$ (that is, (ii) of the Theorem above), whereas (i) provides a finer version of \cite[Theorem 4.8]{MP2}.

%

We recall that the presence in the constants of a parameter such as $\ul{r}_i$ is necessary to guarantee result of proximity to a single ball. In absence of such an assumption, bubbling phenomena may arise (see e.g., \cite{BNST} for Serrin's problem in $\Si=\RR^N$ and \cite{CirMag} for the Soap Bubble Theorem restricted to mean convex hypersurfaces in $\Si=\RR^N$). Finding optimal stability exponents in presence of bubbling is an open problem.

The paper is organised as follows.

In Section \ref{sec:integral identities} we prove the fundamental integral identities and the corresponding rigidity results.

Section \ref{sec:Poincare inequalities} is devoted to discuss Poincar\'e-type inequalities. In Theorems \ref{thm:Poincare new RN}, \ref{thm:Poincare new in general}, and \ref{thm:Strengthened Poincare new RN} we prove the new Poincar\'e-type inequality \eqref{eq:INTRO new Poicare caso particolare p=2} together with its strengthened Sobolev-Poincar\'e version. 

In Section \ref{sec: pointwise estimates for torsion}  we introduce the new notions of uniform interior and exterior sphere conditions relative to $\Si$ and use them to obtain explicit uniform lower and upper bounds for the gradient of $u$. We also provide general explicit estimates for $\nr u \nr_{L^\infty (\Si\cap\Om)}$ and for the growth of $u$ from the boundary. 

Section \ref{sec: Sharp stability in L2} contains the sharp stability results for the $L^2$-pseudodistance, including Theorems \ref{thm:INTRO_SBT stab Lipschitz costanti non geometriche} and \ref{thm:INTRO_HK stability pseudodistance}.

In Section \ref{sec:Stronger stability rhoe - rhoi} we prove the stability results for $\rho_e - \rho_i$: see Theorems \ref{thm:SBT stability rhoe rhoi}, \ref{thm:HK stability rhoe rhoi}.

In Section \ref{sec: Additional remarks} we discuss the corresponding stability results for alternative choices of the point $z$ and the classical case $\Si=\RR^N$, including Theorems \ref{thm:intro alternative z} and \ref{thm:classical INTRO in R^N}.

\section{Integral identities}\label{sec:integral identities}


\begin{setting}\label{Setting}
{\rm
Let $\Si\subseteq\RR^N$ be a convex cone with vertex at the origin.
We consider a bounded domain (i.e., a bounded connected open set) $\Si \cap \Om$ -- where $\Om$ is a smooth bounded domain in $\RR^N$ -- such that its boundary relative to the cone $\Ga_0:= \Si \cap \pa\Om$ is smooth,  while $\pa \Ga_0$ is a $N-2$-dimensional manifold and $\pa(\Si\cap\Om)\setminus \Ga_0$ is smooth enough outside a singular set $\mathcal{S} \subset \pa \Si$ of finite $\ell$-dimensional upper Minkowski content with $0 \le \ell \le N-2$, that is\footnote{Up to a normalization factor, which is not relevant for our purposes.},
\begin{equation}\label{eq:setting B Minkowski content}
\infty > \cM^{* \ell }( \cS ) := \limsup_{r \to 0^+}  \frac{ | \left\{ x : \dist (x , \cS) \le r \right\} | }{r^{N-\ell}} \quad \text{ for some } \quad 0 \le \ell \le N-2 .
\end{equation}
Here, $|\left\{ x : \dist (x , \cS) \le r \right\} |$ denotes the $N$-dimensional Lebesgue measure of $ \left\{ x : \dist (x , \cS) \le r \right\}$. We refer to \cite{Federer} for details on the Minkowski content.

In particular, we have that
\begin{equation}\label{eq:setting A singular set 0 measure}
\mathcal{H}^{N-1} \left( \ol{\mathcal{S}} \right) = 0 .
\end{equation}

Notice that $\mathcal{S}$ also possibly takes into account the $N-2$-dimensional manifold $\pa \Ga_0$\footnote{In the particular case where $\Si \setminus \left\lbrace 0 \right\rbrace$ is smooth, we have that $\mathcal{S}\subseteq \left\lbrace 0 \right\rbrace \cup \pa\Ga_0$.}, and to simplify matters, we also assume that $\cH^{N-2} (\pa^*\Si \cap \pa\Ga_0) = \cH^{N-2} (\pa\Ga_0)$, where $\pa^*\Si$ denotes the smooth part of $\pa\Si$.

We set $\Ga_1 :=  \pa(\Si\cap\Om) \setminus \left( \ol{\Ga}_0  \cup \ol{\cS} \right)$. By \eqref{eq:setting A singular set 0 measure}, we have that
$\mathcal{H}^{N-1} \left( \Ga_0 \cup \Ga_1 \right) = \mathcal{H}^{N-1} \left( \pa(\Si\cap\Om) \right)$.
%
%
Finally, we assume that $\Si\cap\Om$ admits a solution $u$ to \eqref{eq:problem torsion} such that
\begin{equation}\label{eq:setting B W22}
u \in W^{1,\infty}(\Si \cap \Om) \cap W^{2,2} (\Si \cap \Om),
\end{equation}
that is
\begin{equation}\label{eq:problem torsion+inegrability assumption}
\begin{cases}
\De u = N \quad & \text{ in } \Si \cap \Om
\\
u= 0 \quad & \text{ on } \Ga_0
\\
u_\nu=0 \quad & \text{ on } \Ga_1
\end{cases}
\quad 
u \in W^{1,\infty}(\Si \cap \Om) \cap W^{2,2} (\Si \cap \Om)
\end{equation}
We recall that, by standard regularity theory, we have that $u \in C^\infty (\Si\cap\Om)\cap C^{2}((\Si\cap\Om)\cup\Ga_0\cup\Ga_1)$ provided $\Ga_0$ and $\Ga_1$ are smooth enough.
}
\end{setting}

\bigskip

In this setting, all the applications of the divergence theorem performed in the present article are allowed in light of the following.

\begin{prop}[Divergence theorem in cones]
\label{prop:divergence theorem in cones}
Let $\vV : \Si\cap\Om \to \RR^N$ be a vector field such that $\vV \in L^{2}(\Si\cap\Om) \cap C^1((\Si\cap\Om)\cup\Ga_0\cup\Ga_1)$
and $\dv(\vV)\in L^1(\Si\cap\Om)$.
Then, we have that
%
%
\begin{equation*}
\int_{\Si\cap\Om} \dv ( \vV ) \, dx = \int_{\Ga_0 \cup \Ga_1} \langle \vV ,\nu \rangle \, dS_x ,
\end{equation*}
provided that $\langle \vV ,\nu \rangle = w_1 + w_2 $ with $w_1 \in L^1(\Ga_0 \cup \Ga_1)$ and $w_2 \ge 0$ on $\Ga_0\cup\Ga_1$. 
\end{prop}
\begin{proof}
We approximate $\Si\cap\Om$ with domains obtained by chopping off a tubular neighbourhood of $\cS$. 
That is, for small $\ve >0$ we define
$$
A_{\ve} := (\Si\cap\Om) \setminus \left\lbrace x \in \Si\cap\Om \, : \, \dist (x, \mathcal{S} ) \le \ve \right\rbrace ,
$$
and being as $\vV \in C^1 \left( \ol{A}_{\ve} \right)$, we compute
\begin{equation}\label{eq:div theo approx step}
\int_{A_{\ve}} \dv( \vV ) \, dx = \int_{ G_{\ve} }
\langle \vV , \nu \rangle \, dS_x + \int_{\pa A_{\ve} \setminus G_{\ve} } \langle \vV , \nu \rangle \, dS_x ,
\end{equation}
where
$$G_{\ve} := (\Ga_0 \cup\Ga_1) \setminus  \left\lbrace x \in \pa( \Si\cap\Om ) \, : \, \mathrm{dist}(x, \mathcal{S}) \le \ve \right\rbrace .$$

The left-hand side of \eqref{eq:div theo approx step} converges to $\int_{\Si\cap\Om} \dv( \vV ) \, dx $
as $\ve$ tends to $0$.

We now use \eqref{eq:setting B Minkowski content} and $\vV \in L^{2}(\Si\cap\Om)$ to estimate the second integral on the right-hand side of \eqref{eq:div theo approx step}.
We start by noting that
$$
\left| \int_{\pa A_{\ve} \setminus G_{\ve} } \langle \vV , \nu \rangle \, dS_x \right| 
\le 
\int_{(\Si\cap\Om) \cap U_{\ve} } \left| \vV \right| \, dS_x, 
$$
where
$$
U_{\ve} := \left\lbrace x \, : \, \dist (x, \mathcal{S}) = \ve \right\rbrace. $$
%
%
Let us prove that there exists a sequence $\ve_j$ converging to $0$ (which we can assume to be monotone decreasing) such that
\begin{equation}\label{eq:final step div theo approx}
\int_{(\Si\cap\Om) \cap \left(  U_{\ve_j} \right) } \left| \vV \right| \, dS_x \to 0 .
\end{equation}
By coarea formula, the function $\rho \to \int_{(\Si\cap\Om)\cap U_\rho} |\vV| \, dS_x$ belongs to $L^1 \left( \left( 0,1 \right) \right)$, and by \eqref{eq:setting B Minkowski content} we get that 
$$
\left| \left\lbrace x \, : \, \dist (x, \mathcal{S}) < \ve \right\rbrace \right| \lesssim \eps^{N- \ell } \le \eps^{2} \quad \text{ for } 0<\ve<1 .
$$
Hence, for any $\de_j:=1/j$, $j \in \NN$, there exists $\ve_j \in (0, \de_j)$ such that
\begin{equation*}
\begin{split}
\int_{(\Si\cap\Om) \cap U_{\ve_j}} |\vV| \,dS_x 
& \le 
\frac{1}{\de_j} \int_0^{\de_j} \int_{(\Si\cap\Om) \cap U_{\rho}} |\vV| \,dS_x \, d\rho
\\
& \le  \frac{1}{\de_j} \int_{(\Si\cap\Om) \cap \left\lbrace x \, : \, \dist (x,\cS ) \le \de_j  \right\rbrace } |\vV| \, dx
\\
& \le 
\left( \int_{(\Si\cap\Om) \cap \left\lbrace x \, : \, \dist (x, \cS )  \le \de_j  \right\rbrace } |\vV|^2 \, dx \right)^{1/2} \, \frac{\left| \left\lbrace x \, : \, \dist (x, \cS ) \le \de_j \right\rbrace \right|^{1/2}}{\de_j}
\\
& \lesssim \left( \int_{(\Si\cap\Om) \cap \left\lbrace x \, : \, \dist (x, \cS ) \le \de_j  \right\rbrace } |\vV|^2 \, dx \right)^{1/2} .
\end{split}
\end{equation*}
Since $\vV \in L^2(\Si\cap\Om)$, the right-hand side converges to $0$ as $\de_j=1/j \to 0$,
that is, \eqref{eq:final step div theo approx}. 
We thus deduce that
\begin{equation*}
\int_{ G_{\ve_j} }
\langle \vV , \nu \rangle \, dS_x \to \int_{\Si\cap\Om} \dv ( \vV ) \, dx 
\quad \quad
\text{as } j\to +\infty .
\end{equation*}
On the other hand, being as $G_{\ve_j}$ a monotone sequence of sets which exhausts $\Ga_0\cup\Ga_1$, Beppo-Levi's Theorem (for $w_2$ in the statement) and the dominated convergence theorem (for $w_1$ in the statement) give that
$$
\int_{ G_{\ve_j} }
\langle \vV , \nu \rangle \, dS_x \to \int_{ \Ga_0\cup\Ga_1 }
\langle \vV , \nu \rangle \, dS_x  ,
$$
which completes the proof.
\end{proof}

\begin{rem}\label{rem:alternative settings}
{\rm
All the results of the paper remain valid for suitable modifications of Setting \ref{Setting} that allow to apply the divergence theorem (by means of possibly different approximation arguments): more precisely, the assumptions on the singular set $\mathcal{S}$ may be relaxed at the cost of assuming stronger regularity of $u$ than in \eqref{eq:setting B W22}. 
} 
\end{rem}


%

By $|\Si \cap \Om|$ and $|\Ga_0|$, we will denote
%
%
the $N$-dimensional Lebesgue measure of $\Si \cap \Om$
and the surface measure of $\Ga_0$.
\par
We set $R$ to be the reference constant given by
\begin{equation}
\label{def-R}
R=\frac{N\,|  \Si\cap\Om |}{|\Ga_0|}, 
\end{equation}
%
%
and we use the letter $q$ to denote the quadratic polynomial defined by
\begin{equation}
\label{quadratic}
q(x)=\frac12\, |x-z|^2,
\end{equation}
where $z$ is any point in $\RR^N$.

We introduce the P-function
\begin{equation}
\label{P-function}
P = \frac{1}{2}\,|\nabla u|^2 - u,
\end{equation}
and we easily compute that
\begin{equation}
\label{eq:P-gradient and Laplacian-differential-identity}
\na P= \na^2 u \na u - \na u \quad \text{ and } \quad \De P = |\na^2 u|^2-\frac{(\De u)^2}{N}.
\end{equation}

We are now ready to provide the proof of the fundamental identity for Serrin's problem in cones.

\begin{thm}[Fundamental identity for Serrin's problem in cones] 
\label{thm:serrinidentity}
Let $\Si\cap\Om$ as in Setting \ref{Setting}, and $R$ be the positive constant defined in \eqref{def-R}.
If $u$ satisfies \eqref{eq:problem torsion+inegrability assumption}, then we have that
%
%
\begin{equation}
\label{idwps}
\int_{ \Si \cap \Om} (-u) \left\{ |\na ^2 u|^2- \frac{ (\De u)^2}{N} \right\} dx + \int_{\Ga_1} u \,  \langle \na^2 u \na u, \nu \rangle \, dS_x =
\frac{1}{2}\,\int_{\Ga_0} \left( u_\nu^2- R^2\right) (u_\nu-q_\nu)\,dS_x ,
\end{equation}
where $q$ is that defined in \eqref{quadratic}. The identity holds true for any choice of $z \in \RR^N$ provided that
\begin{equation}\label{eq:inner product z in cone}
\langle x-z ,\nu \rangle = 0 \quad \text{ for any } x \in \Ga_1 .
\end{equation}
\end{thm}
As already noticed in the introduction, being as $\Si$ a cone with vertex at the origin and $\Ga_1 \subset \pa \Si$, \eqref{eq:inner product z in cone} is equivalent to $\langle z ,\nu \rangle = 0$ on $\Ga_1$. It is clear that choosing $z$ as the origin
%
%
always guarantees the validity of \eqref{eq:inner product z in cone} (and hence \eqref{idwps}). 
\begin{proof}[Proof of Theorem \ref{thm:serrinidentity}]
By \eqref{eq:inner product z in cone}, the homogeneous Neumann condition on $\Ga_1$ in \eqref{eq:problem torsion}, and the divergence theorem, we have that
$$
\int_{\Ga_0} (u_\nu - q_\nu ) dS_x = \int_{\Ga_0 \cup \Ga_1} (u_\nu - q_\nu ) dS_x = \int_{\Si\cap\Om} \De (u - q ) dx = 0 ,
$$ 
so that it is sufficient to prove that
\begin{equation}\label{eq:Serrinidentityoldnewproof step 1}
\int_{\Si\cap\Om} (-u) \De P \, dx	 + \int_{\Ga_1} u \,  \langle \na^2 u \na u, \nu \rangle \, dS_x =
\frac{1}{2}\,\int_{\Ga_0} u_\nu^2 \, (u_\nu-q_\nu) \, dS_x .
\end{equation}

Integration by parts gives that
$$
\int_{\Si\cap\Om} ( P \, \De u - u \, \De P ) \, dx = \int_{\Ga_0 \cup \Ga_1} ( P \, u_\nu - u \,  P_\nu ) \, dS_x ,
$$
which, using \eqref{eq:problem torsion}  and the fact that
\begin{equation}\label{eq:gradient u = u_nu nu on level set}
	\na u = u_\nu \nu \quad \text{ on } \Ga_0 ,
\end{equation}
leads to
\begin{equation}\label{eq:Serrinidentityoldnewproof step 2}
\int_{\Si\cap\Om} (-u) \De P \, dx	 + \int_{\Ga_1} u \,  \langle \na^2 u \na u, \nu \rangle \, dS_x =
\frac{1}{2} \int_{\Ga_0} u_\nu^3 \, dS_x - N \int_{\Si\cap\Om} P \, dx . 
\end{equation}

Finally, we compute that
\begin{equation}\label{eq:Serrinidentityoldnewproof step 3} 
N \int_{\Si\cap\Om} P \, dx = 
\frac{N}{2} \int_{\Si\cap\Om} | \na u |^2 \, dx - \int_{\Si\cap\Om} u \De u \, dx  
=
\left( \frac{N}{2} + 1 \right) \int_{\Si\cap\Om} | \na u |^2 \, dx 
=
\frac{1}{2} \int_{\Ga_0}   u_\nu^2 \, q_\nu \, dS_x .
\end{equation}
Here, in the first identity we used that $\De u=N$ in $\Si\cap\Om$, the second identity easily follows by integration by parts, and the third identity follows by integrating over $\Si\cap\Om$ the following differential identity
$$
\left( \frac{N}{2} + 1  \right) | \na u |^2 = \dv \left\lbrace \langle x-z , \na u \rangle \na u - \frac{| \na u|^2}{2} (x-z) + N \left[ u \na u - (x-z) u \right] \right\rbrace ,
$$
which holds true in $\Si\cap\Om$ being as $\De u = N$, and using the divergence theorem, \eqref{eq:problem torsion}, and \eqref{eq:inner product z in cone}.
Putting together \eqref{eq:Serrinidentityoldnewproof step 2} and \eqref{eq:Serrinidentityoldnewproof step 3} gives \eqref{eq:Serrinidentityoldnewproof step 1} and completes the proof.
\end{proof}

We now present a couple of integral identities, involving the solution $u$ to \eqref{eq:problem torsion+inegrability assumption} and a harmonic function $v$, which will allow to establish useful ad-hoc (weighted) trace-type inequalities (see Lemma \ref{lem:weighted trace inequality ad hoc}).

\begin{lem}\label{lem:identityfortraceineq}
Let $\Si\cap\Om$ as in Setting \ref{Setting} and let $u$ satisfy \eqref{eq:problem torsion+inegrability assumption}, and consider a function $v \in W^{1,\infty}(\Si\cap\Om) \cap W^{2,2}(\Si\cap\Om)$ which is harmonic in $\Si \cap \Om$.
Then, we have that
\begin{equation}\label{eq:nuovaidentityforidentityfortraceinequality}
\int_{\Ga_0} v^2 u_{\nu} \, dS_x 
=
2 \int_{\Ga_1} u \, v \, v_\nu \, dS_x  + N \int_{\Si\cap\Om} v^2 dx + 2 \int_{\Si\cap\Om} (-u) |\na v|^2 dx .
\end{equation}
If $\langle  \na^2 v \na v , \nu \rangle \le 0$ on $\Ga_1$, then we also have that
\begin{equation}\label{eq:identityfortraceinequality}
\int_{\Ga_0 } |\na v|^2 u_{\nu} dS_x  = 2 \int_{\Ga_1} u \langle  \na^2 v \na v , \nu \rangle dS_x + N \int_{\Si\cap\Om} |\na v|^2 dx + 2 \int_{\Si\cap\Om} (-u) |\na^2 v|^2 dx.
\end{equation}
\end{lem}
\begin{proof}
We begin with the following differential identity:
\begin{equation}\label{eq:prova differential identity for trace}
\dv\,\{v^2 \na u - u \, \na(v^2)\}= v^2 \De u - u \, \De (v^2)= N \, v^2 - 2 u \, |\na v|^2,
\end{equation}
which holds for any $v$ harmonic function in $\Si \cap \Om$, if $u$ satisfies \eqref{eq:problem torsion} (here, we only used that $\De u=N$ in $\Si\cap\Om$).
Next, we integrate on $\Si\cap\Om$ and, by the divergence theorem (and the boundary conditions in \eqref{eq:problem torsion}),
we get \eqref{eq:nuovaidentityforidentityfortraceinequality}.

If we use \eqref{eq:prova differential identity for trace} with $v=v_{x_i}$ and sum up over $i=1,\dots, N$, we get
\begin{equation*}
\dv\,\{ | \na v|^2 \na u - u \, \na \left( | \na v^2| \right) \} = N \, | \na v|^2 - 2 u \, |\na^2 v|^2 ,
\end{equation*}
which integrated over $\Si\cap\Om$ leads to \eqref{eq:identityfortraceinequality} after an application of the divergence theorem.
\end{proof}

\begin{thm}
\label{th:fundamental-identity}
Let $\Si\cap\Om$ as in Setting \ref{Setting}, 
%
%
and denote by $H$ the mean curvature of $\Ga_0$.
\par
If $u$ satisfies \eqref{eq:problem torsion+inegrability assumption},
%
%
then the following identity holds:
\begin{equation}\label{fundamental}
\frac1{N-1} \left\lbrace \int_{\Si \cap \Om} \left\{ |\na ^2 u|^2-\frac{(\De u)^2}{N}\right\} \, dx +  \int_{\Ga_1} \left( - \langle \na^2 u \na u , \nu \rangle  \right) \, dS_x \right\rbrace = N | \Si \cap \Om| -\int_{\Ga_0} H \, (u_\nu)^2\,dS_x.
\end{equation}
\end{thm}

\begin{proof}
Let $P$ be given by \eqref{P-function}. By the divergence theorem we can write:
\begin{equation}
\label{div-p-function}
\int_{ \Si \cap \Om } \De P \,dx = \int_{\Ga_0 \cup \Ga_1} P_\nu \, dS_x = \int_{\Ga_0} P_\nu \, dS_x + \int_{\Ga_1} \langle\na^2 u \na u, \nu \rangle \, dS_x.
\end{equation}
In the second identity, we used \eqref{eq:P-gradient and Laplacian-differential-identity} and the Neumann boundary condition on $\Ga_1$ in \eqref{eq:problem torsion}.
By recalling \eqref{eq:gradient u = u_nu nu on level set}, we compute that
$$
P_\nu=\lan D^2 u\, \na u, \nu\ran-u_\nu=u_\nu \lan (D^2 u)\,\nu,\nu\ran-u_\nu=u_{\nu\nu}\, u_\nu-u_\nu \quad \text{ on } \Ga_0.
$$
We recall Reilly's identity,
\begin{equation*}
u_{\nu\nu}+(N-1)\,H\,u_\nu=N \quad \text{ on } \Ga_0 ,
\end{equation*}
from which we obtain that
$$
u_{\nu\nu}\, u_\nu+(N-1)\,H\, (u_\nu)^2=N\,u_\nu \quad \text{ on } \Ga_0  ,
$$
and hence
$$
P_\nu=(N-1)\,u_\nu-(N-1)\,H\, (u_\nu)^2 \quad \text{ on } \Ga_0 .
$$

Therefore, \eqref{fundamental} follows from this identity, \eqref{eq:P-gradient and Laplacian-differential-identity}, \eqref{div-p-function} and the formula
\begin{equation}
\label{volume}
\int_{\Ga_0} u_\nu \, dS_x = \int_{\Ga_0 \cup \Ga_1} u_\nu \, dS_x = N \, |\Si \cap \Om|,
\end{equation}
which is an easy consequence of the divergence theorem and the Neumann condition in \eqref{eq:problem torsion}.
\end{proof}

\begin{thm}[Fundamental identities for the Soap Bubble Theorem in cones]
\label{thm:identitySBT}
Let $\Si\cap\Om$ as in Setting \ref{Setting},
and
%
%
denote by $H$ the mean curvature of $\Ga_0$.
Let $R$ be the positive constant defined in \eqref{def-R}.

If $u$ satisfies \eqref{eq:problem torsion+inegrability assumption}, then the following identity holds true:
\begin{equation}
\label{H-fundamental}
\begin{split}
\frac1{N-1} \left\lbrace
\int_{ \Si \cap \Om } \left\{ |\na ^2 u|^2-\frac{(\De u)^2}{N}\right\}dx
+
\int_{\Ga_1} \left( - \langle \na^2 u \na u , \nu \rangle  \right) \, dS_x 
\right\rbrace
& +
\frac1{R}\,\int_{\Ga_0} (u_\nu-R)^2 dS_x 
\\
& = 
\int_{\Ga_0} \left( \frac{1}{R} -H \right) \, (u_\nu)^2 dS_x .
\end{split}
\end{equation}
Moreover, if for any $z \in \RR^N$ satisfying \eqref{eq:inner product z in cone} we set
\begin{equation}\label{def-H0}
H_0(z) = \frac{1}{R} - \frac{\int_{\pa \Ga_0} \langle x-z,\mathbf{n}_x\rangle d \cH^{N-2}_x }{(N-1)N|\Si\cap\Om|} ,
\end{equation}
where $\mathbf{n}_x$ denotes the outward unit conormal to $\pa \Ga_0$ belonging to the tangent space of $\Ga_0$, we have that
\begin{multline}
\label{eq:NEWVERSION identity-SBT2}
\frac1{N-1} \left\lbrace \int_{\Si \cap \Om} \left\{ |\na ^2 u|^2-\frac{(\De u)^2}{N}\right\}dx
-  \int_{\Ga_1} \langle \na^2 u \na u , \nu \rangle   \, dS_x
- \frac{\int_{\Ga_0} u_\nu^2 \, dS_x }{N |\Si\cap\Om|}   \int_{\pa \Ga_0 } \langle x-z, \mathbf{n}_x \rangle d \cH^{N-2}_x
\right\rbrace
\\
+
\frac1{R}\,\int_{\Ga_0} (u_\nu-R)^2 dS_x =
\int_{\Ga_0} (H_0(z) - H) (u_\nu^2-R^2) dS_x + R \int_{\Ga_0} (H_0(z) - H)(R - q_\nu ) dS_x
\end{multline}
\end{thm}
\begin{rem}\label{rem:casi H_0 non dipendente da z}
{\rm
(i) If $\ol{\Ga}_0$ and $\pa\Si$ meet orthogonally $\cH^{N-2}$-a.e. in $\pa \Ga_0$, then $H_0(z)$ does not depend on the choice of $z$, and we have that 
\begin{equation*}
H_0(z) = \frac{1}{R} \quad \text{ for any } z \text{ satisfying \eqref{eq:inner product z in cone}.}
\end{equation*}
In fact, in such a case, 
by \eqref{eq:inner product z in cone} we get that $\int_{\pa \Ga_0}\langle x-z , \mathbf{n}_x  \rangle \, d \cH^{N-2}_x = 0$.
%
%
%

(ii) If $H$ is constant on $\Ga_0$, then by \eqref{eq:Minlowski relativa} we have that $H_0(z) \equiv H$
for any $z$ satisfying \eqref{eq:inner product z in cone}, that is, $H_0(z)$ does not depend on the choice of $z$.
}
\end{rem}
\begin{proof}[Proof of Theorem \ref{thm:identitySBT}]
Since, by \eqref{volume}, we have that
$$
\frac1{R}\,\int_{\Ga_0} u_\nu^2\,dS_x=
\frac1{R}\,\int_{\Ga_0} (u_\nu-R)^2\,dS_x+N | \Si \cap \Om|,
$$
then
\begin{multline*}
\int_{\Ga_0} H\,(u_\nu)^2\,dS_x=H_0\,\int_{\Ga_0} (u_\nu)^2\,dS_x+
\int_{\Ga_0} \left( H-H_0\right)\,(u_\nu)^2\,dS_x= \\
\frac1{R}\,\int_{\Ga_0} (u_\nu-R)^2\,dS_x+N | \Si \cap \Om|+
\int_{\Ga_0} ( H-H_0)\,(u_\nu)^2\,dS_x.
\end{multline*}
Thus, \eqref{H-fundamental} follows from this identity and \eqref{fundamental} at once.

To prove \eqref{eq:NEWVERSION identity-SBT2}, we use \eqref{def-H0} to re-write the right-hand side of \eqref{H-fundamental} as
$$
\int_{\Ga_0} \left( \frac{1}{R} - H \right) u_\nu^2 \, dS_x
= \int_{\Ga_0} \left( H_0(z )-H \right) u_\nu^2 \, dS_x + \frac{\int_{\Ga_0} u_\nu^2 \, dS_x }{(N-1)N |\Si\cap\Om|}   \int_{\pa \Ga_0 } \langle x-z, \mathbf{n}_x \rangle d \cH^{N-2}_x ;
$$
hence, \eqref{H-fundamental} becomes
\begin{equation}\label{H-fundamental-rewritten}
\begin{split}
\frac1{N-1} \left\lbrace \int_{\Si \cap \Om} \left\{ |\na ^2 u|^2-\frac{(\De u)^2}{N}\right\}dx
-  \int_{\Ga_1} \langle \na^2 u \na u , \nu \rangle   \, dS_x
- \frac{\int_{\Ga_0} u_\nu^2 \, dS_x }{N |\Si\cap\Om|}   \int_{\pa \Ga_0 } \langle x-z, \mathbf{n}_x \rangle d \cH^{N-2}_x
\right\rbrace
\\
+
\frac1{R}\,\int_{\Ga_0} (u_\nu-R)^2 dS_x = \int_{\Ga_0} \left( H_0(z )-H \right) u_\nu^2 \, dS_x .
\end{split}
\end{equation}
Identity \eqref{eq:NEWVERSION identity-SBT2} follows by noting that the right-hand side of \eqref{H-fundamental-rewritten} can be re-written as
$$
\int_{\Ga_0} (H_0(z) - H) u_\nu^2 \, dS_x = \int_{\Ga_0} (H_0(z) - H) (u_\nu^2-R^2) dS_x + R \int_{\Ga_0} (H_0(z) - H)(R - q_\nu ) dS_x ,
$$
being as
$$
\int_{\Ga_0} (H_0(z) - H) q_\nu \, dS_x = 0 .
$$
The last identity follows by putting together $\int_{\Ga_0} H_0(z) \langle x-z , \nu \rangle dS_x = H_0(z)\,  N | \Si\cap\Om |$ (which easily follows by \eqref{def-H0}, \eqref{eq:inner product z in cone} and the divergence theorem) and the following Minkowski-type identity:
\begin{equation}\label{eq:Minlowski relativa}
\int_{\Ga_0} H \langle x-z , \nu \rangle \, dS_x = |\Ga_0| - \frac{1}{N-1} \int_{\pa \Ga_0 } \langle x-z, \mathbf{n}_x \rangle d \cH^{N-2}_x ,
\end{equation}
where $\mathbf{n}_x$ is the outward unit conormal to $\pa \Ga_0$ belonging to the tangent space of $\Ga_0$.
When $\Ga_0$ is a closed surface, i.e., $\pa \Ga_0 = \varnothing $, the last integral vanishes and \eqref{eq:Minlowski relativa} becomes the usual Minkowski identity. In general, \eqref{eq:Minlowski relativa} can be proved by computing the tangential gradient $D_{\Ga_0}$ and the Laplace-Beltrami operator $\De_{\Ga_0}$ of the quadratic polynomial $q$ defined in \eqref{quadratic}, that is:
$$
D_{\Ga_0} q = (x-z) - \langle x-z , \nu \rangle \nu
\quad \text{ and } \quad
\De_{\Ga_0} q = (N-1) \left( 1 - H \langle x-z , \nu \rangle \right)
\quad \text{ for } x \in \Ga_0.
$$
In fact, integrating over $\Ga_0$ and using that
$$
\int_{\Ga_0} \De_{\Ga_0} q \, d \cH^{N-1}_x = \int_{\pa \Ga_0} \langle D_{\Ga_0} q , \mathbf{n}_x \rangle \, d \cH^{N-2}_x = \int_{\pa \Ga_0} \langle x-z , \mathbf{n}_x \rangle \, d \cH^{N-2}_x ,
$$
\eqref{eq:Minlowski relativa} easily follows. The proof is completed.
\end{proof}

We finally show that, if $\Ga_0$ is mean-convex (i.e., $H \ge 0$ on $\Ga_0$), \eqref{fundamental-identity2} can also be rearranged into an identity leading to a new Heintze-Karcher-type inequality relative to a convex cone.

\begin{thm}[Fundamental identity for Heintze-Karcher's inequality in cones]
\label{thm:identityheintze-karcher}
Let $\Si\cap\Om$ as in Setting \ref{Setting},
and denote by $H$ the mean curvature of $\Ga_0$. 
Let $u$ be the solution of \eqref{eq:problem torsion+inegrability assumption}. 
\par
If $\Ga_0$ is mean-convex (i.e., $H \ge 0$), then we have the following identity:
\begin{equation}
\label{heintze-karcher-identity}
\begin{split}
\frac1{N-1} \left\lbrace \int_{\Si \cap \Om} \left\{ |\na ^2 u|^2-\frac{(\De u)^2}{N}\right\} \, dx +  \int_{\Ga_1} \left( - \langle \na^2 u \na u , \nu \rangle  \right) \, dS_x \right\rbrace
& + \int_{\Ga_0} \frac{(1-H\,u_\nu)^2}{H}\,dS_x 
\\
& =
\int_{\Ga_0} \frac{dS_x}{H}-N |\Si\cap\Om|.
\end{split}
\end{equation}
\end{thm}

\begin{proof}
Since \eqref{volume} holds, from \eqref{fundamental} we obtain that
\begin{equation}
\label{fundamental-identity2}
\frac1{N-1} \left\lbrace \int_{\Si \cap \Om} \left\{ |\na ^2 u|^2-\frac{(\De u)^2}{N}\right\} \, dx +  \int_{\Ga_1} \left( - \langle \na^2 u \na u , \nu \rangle  \right) \, dS_x \right\rbrace
= \int_{\Ga}(1-H u_\nu)\, u_\nu\,dS_x.
\end{equation}
Notice that \eqref{fundamental-identity2} still holds without the assumption $H\ge0$.

Since $u_\nu$ and $H$ are continuous, the identity
\begin{equation}\label{eq:HKsenzastrictly}
\frac{(1-H\,u_\nu)^2}{H}=-(1-H\,u_\nu) u_\nu+\frac1{H}-u_\nu,
\end{equation}
holds pointwise for $H\ge0$.

If the integral $\int_{\Ga_0} \frac{dS_x}{H}$ is infinite, then \eqref{heintze-karcher-identity} trivially holds. Otherwise $1/H$ is finite (and hence $H > 0$) almost everywhere
on $\Ga_0$ and hence, \eqref{eq:HKsenzastrictly} holds almost everywhere.
Thus, by integrating \eqref{eq:HKsenzastrictly} on $\Ga_0$, summing the result up to \eqref{fundamental-identity2} and taking into account \eqref{volume},
we get \eqref{heintze-karcher-identity}.
\end{proof}

\bigskip

The quantity
$$
|\na^2 u|^2 - \frac{(\De u)^2}{ N} ,
$$
that we call Cauchy-Schwarz deficit for the hessian matrix $\na^2 u$, plays the role of detector of spherical symmetry, as is clear from the following lemma. In what follows, $I$ denotes the identity matrix.

\begin{lem}[Spherical detector]
\label{lem:sphericaldetector}
Let $\Si\cap\Om \subseteq \RR^N$ be a domain and $u \in C^2 (\Si\cap\Om)$.
Then it holds that
\begin{equation}
\label{newton}
|\na^2 u|^2 - \frac{(\De u)^2}{ N} \ge 0 \quad \mbox{in } \Si\cap\Om .
\end{equation}

Moreover, if $u$ is the solution of \eqref{eq:problem torsion}, the equality sign holds in \eqref{newton} if and only if $\Ga_0$ is a piece of sphere 
%
%
of radius $R$ given by \eqref{def-R} which meets $\pa \Si$ orthogonally, and
$$
u(x)= \frac{1}{2} \, (|x-z|^2 - R^2),
$$
for some $z \in \RR^N$ satisfying \eqref{eq:inner product z in cone}.
\end{lem}
\begin{proof}
Regarding the matrices $\na^2 u$ and $I$ as vectors in $\RR^{N^2}$, \eqref{newton} is the classical Cauchy-Schwarz inequality.
%
%

For the characterization of the equality case,
we consider $u$ satisfying \eqref{eq:problem torsion} and we set $w(x)= u(x) - |x|^2 /2$. Since $w$ is harmonic, direct computations show that
$$
0=|\na^2 u|^2 - \frac{(\De u)^2}{ N} = | \na^2 w|^2.
$$
Thus, $w$ is affine and $u$ is quadratic. Therefore, $u$ can be written in the form
$u(x)=\left( |x-z|^2 - a \right)/2$,
for some $z\in\RR^N$ and $a\in\RR$.
\par
Since $u=0$ on $\Ga_0$, then $|x-z|^2=a$ for $x\in \Ga_0$, that is, $a$ must be positive and $\Ga_0$ must be contained in $\Si \cap \pa B_{\sqrt a }(z)$.

%

Since the function $u$ (which is now explicit) is $C^1$ up to $\pa \Ga_0$, we have that for any $x \in \pa \Ga_0$, the normal to $\Ga_0$ is parallel to $Du(x)=(x-z)$ and it is orthogonal to the normal to the cone by the Neumann condition.
Moreover, being as $\na u = x-z$ we have that  $\langle x-z, \nu \rangle = u_\nu = 0$ on $ \Ga_1$ (that is, $z$ satisfies \eqref{eq:inner product z in cone}), and hence
$$
\sqrt{a}\,|\Ga_0|=\int_{\Ga_0} |x-z|\,dS_x
=\int_{\Ga_0} \langle x-z , \nu(x) \rangle \,dS_x
= \int_{\Ga_0 \cup \Ga_1} \langle x-z , \nu(x) \rangle \, dS_x
= N\,|\Si \cap \Om| .
$$

We mention that if $u$ is any smooth
%
%
function (not necessarily satisfying \eqref{eq:problem torsion}) for which the equality holds in \eqref{newton}, we can still
conclude that $u$ is a quadratic polynomial. A proof of this fact, which we do not need in the present paper, can be found in \cite{Pog2}.
%
\end{proof}


As immediate corollaries of the fundamental identities of Theorems \ref{thm:serrinidentity}, \ref{thm:identitySBT}, and \ref{thm:identityheintze-karcher} we have the following symmetry results.

The convexity of $\Si$ and the Neumann condition on $\Ga_1$ in \eqref{eq:problem torsion} guarantee that\footnote{Differentiating (along the tangential direction $\na u$) the Neumann condition $\langle \na u, \nu \rangle=0$ on $\Ga_1$ gives that $\langle \na^2 u \na u, \nu \rangle = - \langle D \nu \na u, \na u \rangle \text{ on } \Ga_1 $, where the right-hand side is non-positive since the second fundamental form of $\pa\Si$ at the points of $\Ga_1 \subset \pa\Si$ is non-negative definite, by convexity of $\Si$. Hence, \eqref{eq:disuguaglianza puntuale gratis per convex} immediately follows. This was also used in \cite[(3.8)]{PT}. Notice that, in the particular case where $\Ga_1$ is flat, \eqref{eq:disuguaglianza puntuale gratis per convex} holds with the equality sign.}
\begin{equation}
\label{eq:disuguaglianza puntuale gratis per convex}
\langle \na^2 u \na u, \nu \rangle \le 0 \quad \text{ on } \Ga_1 ,
\end{equation}
which in particular ensures that
\begin{equation}\label{eq:disuguaglianza senza u gratis per convex}
\int_{\Ga_1 } \langle \na^2 u \na u, \nu \rangle \, dS_x \le 0
\end{equation}
and, being as $u \le 0$ (by the comparison principle in Lemma \ref{lem:comparison in cones}), also
\begin{equation}\label{eq:disuguaglianza con u gratis per convex}
\int_{\Ga_1 } u \, \langle \na^2 u \na u, \nu \rangle \, dS_x \ge 0 .
\end{equation}

\begin{thm}[Symmetry for Serrin's problem relative to $\Si$] 
\label{thm:rigidity Serrin}
Let the assumptions of Theorem \ref{thm:serrinidentity} be satisfied.

If the right-hand side of \eqref{idwps} is non-positive, then we must have $\Si\cap\Om = \Si\cap B_R(z)$ for some $z$ satisfying \eqref{eq:inner product z in cone}. 
The same conclusion clearly holds if either $u_\nu$ is constant on $\Ga_0$ or $u_\nu - \langle x-z, \nu \rangle = 0$ on $\Ga_0$.
\end{thm}
\begin{proof}
If the right-hand side of \eqref{idwps} is non-positive, then the two integrands at the left-hand side must be zero, being non-negative by \eqref{newton}, the comparison principle in Lemma \ref{lem:comparison in cones} with $f:=-u$, and \eqref{eq:disuguaglianza con u gratis per convex}. Then \eqref{newton} must hold with the equality sign, since $u<0$ on $\Si\cap\Om$ by Lemma \ref{lem:relation u dist general}.
%
%
The conclusion follows from Lemma \ref{lem:sphericaldetector}.
\par
Finally, if $u_\nu\equiv c$ on $\Ga_0$ for some constant $c$, then 
$$
c\,|\Ga_0|=\int_{\Ga_0 \cup \Ga_1} u_\nu\,dS_x = N\,|\Si \cap \Om|,
$$
that is $c=R$, and hence we can apply the previous argument.

The same conclusion clearly holds if $u_\nu-q_\nu=0$ on $\Ga_0$. 
\end{proof}

\begin{thm}[Soap Bubble Theorem relative to $\Si$]\label{thm:SBT cone}
Let the assumptions of Theorem \ref{thm:identitySBT} be satisfied.

If the right-hand side of \eqref{H-fundamental} is non-positive, then we must have $\Si\cap\Om = \Si\cap B_R(z)$ for some $z$ satisfying \eqref{eq:inner product z in cone}.

The same conclusion clearly holds if the mean curvature $H$ of $\Ga_0$ satisfies $H\ge 1/R$ on $\Ga_0$.

In particular, the last inequality is satisfied if $H$ is constant on $\Ga_0$ and the gluing condition $\int_{\pa \Ga_0} \langle x -z ,\mathbf{n}_x\rangle d \cH^{N-2}_x \le 0$ is verified for at least one point $z$ satisfying \eqref{eq:inner product z in cone}.
\end{thm}
\begin{proof}
If the right-hand side in \eqref{H-fundamental} is non-positive (and this certainly happens if $H\ge 1/R$ on $\Ga_0$), all the summands at the left-hand side must be zero, being non-negative.
\par
The fact that the first summand is zero gives that the Cauchy-Schwarz deficit for the hessian matrix $\na^2 u$ must be identically zero and the conclusion follows from Lemma \ref{lem:sphericaldetector}.
\par
If $H$ equals some constant, using \eqref{eq:Minlowski relativa} we find that
$$H \equiv H_0(z)= \frac{1}{R} - \frac{\int_{\pa \Ga_0} \langle x -z ,\mathbf{n}_x\rangle d \cH^{N-2}_x }{(N-1)N|\Si\cap\Om|} ,$$
and hence, if the gluing condition $\int_{\pa \Ga_0} \langle x -z ,\mathbf{n}_x\rangle d \cH^{N-2}_x \le 0$ is in force, $H \ge 1/R$.
\end{proof}

\begin{rem}
{\rm
In passing, in the proof of Theorem \ref{thm:SBT cone} we also obtained that the third summand on the left-hand side of \eqref{H-fundamental} is zero and hence $u_\nu \equiv R$ on $\Ga_0$, that is, $u$ satisfies the assumptions of Theorem \ref{thm:rigidity Serrin}. This provides an alternative way to conclude the proof of Theorem \ref{thm:identitySBT}. More importantly, this piece of information will be used in the quantitative analysis performed in the next sections and will play a crucial role to obtain the optimal Lipschitz rate of stability.
}
\end{rem}

We now show that \eqref{heintze-karcher-identity} implies Heintze-Karcher's inequality relative to the cone $\Si$.

\begin{thm}[Heintze-Karcher's inequality relative to $\Si$ and a related overdetermined problem]
\label{thm:heintze-karcher}
Let the assumptions of Theorem \ref{thm:identityheintze-karcher} be satisfied.

\begin{enumerate}[(i)]
\item Heintze-Karcher's inequality (relative to $\Si$)
\begin{equation}
\label{heintze-karcher}
\int_{\Ga_0} \frac{dS_x}{H}\ge N |\Si\cap\Om|
\end{equation}
holds and the equality sign is attained if and only if $\Si\cap\Om = \Si\cap B_R(z)$ for some $z$ satisfying \eqref{eq:inner product z in cone};
\item the solution $u$ of \eqref{eq:problem torsion+inegrability assumption} satisfies 
\begin{equation}\label{equa:overdeteitemi}
u_\nu(x)=1/H(x) \quad \mbox{on} \quad \Ga_0 
\end{equation}
if and only if $\Si\cap\Om = \Si\cap B_R(z)$ for some $z$ satisfying \eqref{eq:inner product z in cone}.
\end{enumerate}
 \end{thm}

\begin{proof}
(i) All the summands at the left-hand side of \eqref{heintze-karcher-identity} are non-negative and hence \eqref{heintze-karcher} follows.
If the right-hand side is zero, those summands must be zero. The vanishing of the first summand gives the conclusion via Lemma \ref{lem:sphericaldetector}, as usual. Note in passing that the vanishing of the third summand gives that $u_\nu=1/H$ on $\Ga$, which also implies radial symmetry, by item (ii).

(ii) The right-hand side and the third summand of the left-hand side of \eqref{heintze-karcher-identity} are zero when \eqref{equa:overdeteitemi} occurs.
Thus, the conclusion follows from Lemma \ref{lem:sphericaldetector}, as usual.
\par
Notice that assumption \eqref{equa:overdeteitemi} forces $\Ga_0$ to be mean convex.
\end{proof}

\begin{rem}\label{rem:nota su convexity non necessaria}
{\rm 
A direct inspection of the proofs shows that, once that the fundamental identity in Theorem \ref{thm:serrinidentity} is established, symmetry via Lemma \ref{lem:sphericaldetector} can be obtained in Theorem \ref{thm:rigidity Serrin} even if the assumption of convexity of the cone $\Si$ is replaced by the weaker assumption \eqref{eq:disuguaglianza con u gratis per convex}. Similarly, once that the fundamental identities in Theorems \ref{thm:identitySBT} and \ref{thm:identityheintze-karcher} are established, symmetry via Lemma \ref{lem:sphericaldetector} can be obtained in Theorems \ref{thm:SBT cone} and \ref{thm:heintze-karcher} if the assumption of convexity of the cone $\Si$ is replaced by the weaker assumption \eqref{eq:disuguaglianza senza u gratis per convex}.
}
\end{rem}

As already mentioned in the Introduction, \eqref{eq:inner product z in cone} informs that all the components of $z$ in the directions spanned by $\nu(x)$ for $x\in \Ga_1$ must be zero. In particular, a sufficient condition forcing $z$ to be the origin is that $\mathrm{span}_{x\in\Ga_1} \nu(x) = \RR^N$, that is, \eqref{eq:intro k=N}. Such a sufficient condition generalizes the sufficient condition provided in \cite[Proposition 2.3]{PT isoperimetric}, as shown by the following.

%
%

\begin{prop}\label{prop:nondegenerate point}
Let $N\ge3$. If $\Ga_1$ has a transversally nondegenerate point, then $\mathrm{span}_{x\in \Ga_1} \nu(x) =\RR^N $.
\end{prop} 
For any
cone $\Si$ in $\RR^N$, denote with $\pa^{\ast 2} \Si$ the set of points $x\in \pa\Si$ around which $\pa\Si$ is of class $C^2$ (in Setting \ref{Setting} we have that $\Ga_1 \subset \pa^{\ast 2} \Si $), and, for $x \in \pa^{\ast 2} \Si$, denote with $\mathrm{II}_x$ the second fundamental form of $\pa\Si$ at $x$.
As in \cite{PT isoperimetric}, for $N \ge 3$, we adopt the following definitions:

we say that a point $x \in \pa^{\ast 2} \Si$ is {\it transversally nondegenerate} if the quadratic form $\mathrm{II}_x$ restricted to the tangent directions to $\pa\Si$ which are orthogonal to $x$ has all the eigenvalues different from zero. In other words, all principal curvatures of $\pa\Si$ at $x$ are non-zero except for that in the $x$-direction;

we say that a point $x \in \pa^{\ast 2} \Si$ is a point of {\it strict convexity} (resp. {\it strict concavity}) {\it for} $\pa\Si$ if the quadratic form $\mathrm{II}_x$ is strictly positive (resp. strictly negative) definite when it
is restricted to the tangent directions which are orthogonal to $x$.

\begin{proof}[Proof of Proposition \ref{prop:nondegenerate point}]
The following argument is related to \cite[Proof of Proposition 2.3]{PT isoperimetric}.
By continuity, there exists an open neighbourhood $V$ in $\pa^{\ast 2} \Si$ of transversally nondegenerate points. For $x \in V$ we can write an orthonormal frame for the tangent space $T_x \pa \Si$ to $\pa \Si$ at $x$, as $\left\lbrace \frac{x}{|x|} , \mathbf{e}_1 , \dots , \mathbf{e}_{N-2}  \right\rbrace$. We can always pick the $\mathbf{e}_j$'s such that they diagonalize the second
fundamental form $\mathrm{II}_x$. By the nondegeneracy property we have $\mathrm{II}_x(\mathbf{e}_j , \mathbf{e}_j)= \la_j \neq 0 $, for $j=1, \dots, N-2$.

We are going to prove that the orthogonal complement in $\RR^N$ of the vector subspace $\mathrm{span}_{x\in \Ga_1} \nu(x) \subseteq \RR^N$ is $\left\lbrace 0 \right\rbrace$, which easily entails the desired result.
Assume that $z\in \RR^N$ is such that $\langle z, \nu(x) \rangle = 0$ for any $x\in \Ga_1$.
Differentiating such a relation along $\mathbf{e}_j$, we get
$\la_j (x) \langle z, \mathbf{e}_j \rangle = 0$, for any $j=1,\dots,N-2$. That is, the vector $z$ is orthogonal to $\mathbf{e}_j$, for any $j=1,\dots,N-2$. Therefore, $z$ must be parallel to $\frac{x}{|x|}$. Since this holds true for any $x$ in an open $(N-1)$-dimensional neighbourhood in $\pa^{\ast 2} \Si$, then $z$ is forced to be the origin.
\end{proof}


\section{Poincar\'e-type inequalities}\label{sec:Poincare inequalities}
%
%


In what follows, for a set $G \subset \RR^N$ and a function $v: G \to \RR$, $v_G$ denotes the {\it mean value of $v$ in $G$}, that is
$$
v_G= \frac{1}{|G|} \, \int_G v \, dx.
$$
%
%
We also define 
$$
\nr  \na v \nr_{L^p (G)} = \left( \sum_{i=1}^N \nr  v_i \nr_{L^p (G)}^p \right)^\frac{1}{p} \quad \mbox{and} \quad
\nr  \na^2 v \nr_{L^p (G)} = \left( \sum_{i,j=1}^N \nr  v_{ij} \nr_{L^p (G)}^p \right)^\frac{1}{p},
$$
for $p \in [1, \infty)$.

We start by recalling versions of the classical Poincar\'e and Sobolev-Poincar\'e inequalities which provide explicit estimates of the constants for John domains: these can be deduced by \cite{HS1} and \cite{Bo}.

The class of John domain is huge: it contains Lipschitz domains, but also very irregular domains with fractal boundaries as, e.g., the Koch snowflake.
Roughly speaking, a domain is a $b_0$-John domain if it is possible to travel from one point of the domain to another without going too close to the boundary.
The formal definition is the following: a domain $G$ in $\RR^N$ is a {\it $b_0$-John domain}, $b_0 \ge 1$, if each pair of distinct points $x_1$ and $x_2$ in $G$ can be joined by a curve $\ga: \left[0,1 \right] \rightarrow G$ such that
%
%
\begin{equation*}
	\de_{ \pa G} (\ga(t)) \ge b_0^{-1} \min{ \left\lbrace |\ga(t) - x_1 |, |\ga(t) - x_2 | \right\rbrace  }.
\end{equation*}
The notion could be also defined through the so-called {\it $b_0$-cigar} property (see \cite{Va}).

\begin{lem}[\cite{HS1}]
	\label{lem:BoasStraube}
	Let $G \subset\RR^N$, $N\ge 2$, be a bounded $b_0$-John domain and consider $p \in \left[ 1, \infty \right)$. Then, there exists a positive constant, $\mu_{p} (G)$ such that
	\begin{equation}
		\label{eq:BoasStraube-poincare}
		\nr v - v_G \nr_{L^p(G)} \le \mu_{p} (G)^{-1} \nr  \na v  \nr_{L^p(G)},
	\end{equation}
	for every function $v \in L^p (G) \cap W^{1,p}_{loc} (G)$.
	\par
	The constant  $\mu_{p} (G)^{-1}$ can be explicitly estimated in terms of $N$, $p$, $b_0$, and the diameter $d_G$.
\end{lem}

\begin{rem}
	{\rm
	The arguments in \cite{BS} show that \eqref{eq:BoasStraube-poincare} remains true even without the John assumption: in fact, the sufficient assumption in \cite{BS} reduces to $\pa G$ being locally the graph of a continuous function. 
	}
\end{rem}

When $p<N$, inequality \eqref{eq:BoasStraube-poincare} can be strengthened as follows.

\begin{lem}[\cite{Bo}]\label{lem:Hurri}
Let $G \subset \RR^N$ be a bounded $b_0$-John domain, and consider two numbers $r, p$ such that 
\begin{equation}\label{eq:r p al in Hurri}
p<N \quad \text{and} \quad  1 \le p \le r \le \frac{Np}{N-p} .
\end{equation}
Then, there exists a positive constant $\mu_{r, p} (G)$ such that
\begin{equation}
\label{eq:John-Hurri-poincare}
\nr v - v_G \nr_{L^r (G)} \le \mu_{r, p} (G)^{-1} \nr  \na v  \nr_{L^p(G)},
\end{equation}
for every function $v\in L^1_{loc}(G)$ such that $ \na v \in L^p (G)$ .

The constant  $\mu_{r,p} (G)^{-1}$ can be explicitly estimated in terms of $N$, $p$, $b_0$, and $|G|$.
\end{lem}

%
\begin{rem}[Explicit estimates of the constants and geometric dependence]\label{rem:stime mu HS}
{\rm
The best constant is characterized by the (solvable) variational problem 
\begin{equation*}
\mu_{r, p} (G) = 
\min \left\{ \nr  \na v  \nr_{L^p(G)} : \nr v \nr_{L^r(G)} = 1 \text{ in } G,  v_{G} = 0 \right\}.
\end{equation*} 
The proofs in \cite{HS1} and \cite{Bo} have the benefit of giving explicit upper bounds.
In fact, by putting together \cite[Theorem 8.5]{HS1} and \cite[Theorem 8.5]{MarS} we find that
\begin{equation*}
	\mu_{p} (G)^{-1} \le k_{N, \, p} \, b_0^{3N(1 + \frac{N}{p})} \, d_G ,
\end{equation*}
and from \cite[Chapter 6]{Bo} we have that
\begin{equation*}
\mu_{r, p} (G)^{-1} \le k_{N,\, r, \, p} \, b_0^N |G|^{\frac{1}{N} +\frac{1}{r} +\frac{1}{p} } .
\end{equation*}
}
\end{rem}

As a consequence of Lemma \ref{lem:Hurri}, we have the following version of Sobolev inequality.

\begin{cor}\label{cor:Sobolev weighted}
Let $G \subset \RR^N$ be a bounded $b_0$-John domain, and consider two numbers $r, p$ satisfying \eqref{eq:r p al in Hurri}.
Then, for any function $v\in L^1_{loc}(G)$ such that $ \na v \in L^p (G)$ we have that
\begin{equation}\label{eq:Sobolev weighted}
\nr v  \nr_{L^r (G)} \le 
|G|^{ \frac{1}{r} - \frac{1}{p} } \, \nr v \nr_{L^p(G)} +  \mu_{r, p} (G)^{-1} \, \nr  \na v  \nr_{L^p(G)}
\end{equation}
where $\mu_{r, p} (G)^{-1}$ is that appearing in Lemma \ref{lem:Hurri}.
\end{cor}
\begin{proof}
For any function $v\in L^1_{loc}(G)$ such that $\na v \in L^p (G)$, we compute that
\begin{equation*}
\begin{split}
\nr v \nr_{L^r(G)} 
& \le 
\nr v_G \nr_{L^r(G)} + \nr v - v_G \nr_{L^r(G)}
\\
& = 
|G|^{ \frac{1}{r} - 1 } \,  \left| \int_G v \, dx \right| + \nr v - v_G \nr_{L^r(G)}
\\
& \le 
|G|^{ \frac{1}{r} - \frac{1}{p} } \, \nr v \nr_{L^p(G)} +  \mu_{r, p} (G)^{-1} \, \nr \na v  \nr_{L^p(G)} ,
\end{split}
\end{equation*}
where in the last inequality we used H\"older's inequality for the first summand and Lemma \ref{lem:Hurri} for the second summand.
\end{proof}


We now establish a new Poincar\'{e}-type inequality for vector fields.

\begin{thm}\label{thm:Poincare new RN}
Given $1 \le p < +\infty$, let $G \subset \RR^N$ be a bounded Lipschitz
domain. 
Let $A \subseteq \pa G$ be a relatively open subset of $\pa G$ with positive $N-1$-dimensional measure and such that $\nu(x)$ is continuous on $A$ and
\begin{equation}\label{eq:condition Poincare span RN}
\mathrm{span}_{x \in A} \, \nu (x) = \RR^N .
\end{equation}

Then, there exists a positive constant $ \eta_{p}(A , G)$ (depending on $N$, $p$, $A$ and $G$) such that
\begin{equation}\label{eq:Poincare new RN}
\nr \vV \nr_{L^p(G)} \le \eta_{p}(A , G)^{-1} \, \nr D \vV \nr_{L^p(G)} ,
\end{equation}
for any vector function $\vV: G \to \RR^N$ belonging to $W^{1,p} (G)$
and
such that $\langle \vV , \nu \rangle = 0$ a.e. on $A$.
\end{thm}
\begin{proof}
We use a standard compactness argument that goes back at least to Morrey \cite{Morrey}.
Suppose \eqref{eq:Poincare new RN} to be false. Then there would exist a sequence $\left\lbrace \vV_k \right\rbrace$ such that
\begin{equation}\label{eq:2_PoincareBS}
\nr \vV_k \nr_{ L^p (G)} = 1
\end{equation}
and
\begin{equation}\label{eq:3_PoincareBS}
\nr  D \vV_k \nr_{L^p(G)} < 1/k \quad \text{for each } k .
\end{equation}
In particular, the $\vV_k$ form a bounded sequence in $W^{1,p}(G)$, which embeds compactly in $L^p(G)$ by the Rellich-Kondrachov theorem (see, e.g., \cite[Theorem 6.2]{Adams}).
By passing to a subsequence we may assume that the $\vV_k$ converge in $L^p (G)$ to some limit $\vV_0$, and
in view of \eqref{eq:2_PoincareBS} and \eqref{eq:3_PoincareBS} the convergence even takes place in $W^{1,p} (G)$.
But $\nr \vV_0 \nr_{W^{1,p} (G)} = 1$ by \eqref{eq:2_PoincareBS}, and $D \vV_0$ vanishes identically (a.e. in $G$) by \eqref{eq:3_PoincareBS}. Hence, $\vV_0$ is a nonzero constant vector in $G$. Therefore, by trace theorem (see, e.g., \cite{Gagliardo}, \cite{Morrey}, \cite{Leoni}) we have that 
%
%
\begin{equation}\label{eq:4_PoincareBS}
\vV_k \to \vV_0  \quad \text{in } L^p ( \pa G) \supseteq L^p ( A) , \quad \text{where } \vV_0 \text{ is a nonzero constant vector}.
\end{equation}

We now use the assumption $\langle \vV_k, \nu \rangle = 0$ a.e. on $A$, to find a contradiction.
In fact, by using such assumption 
%
%
we have that
$$ | \langle \vV_0, \nu(x) \rangle | = |\langle \vV_k(x) - \vV_0, \nu(x) \rangle| \le  |\vV_k(x) - \vV_0| \quad \text{ for a.e. } x \text{ in } \,  A ,$$
and hence $\langle \vV_0, \nu \rangle = 0$ a.e. on $A$.
By continuity of $\nu$ on $A$, we actually have that
\begin{equation}\label{eq:provamethodPoincare}
\langle \vV_0, \nu \rangle = 0 \quad \text{everywhere in } A .
\end{equation}
Since $\vV_0 \in \mathrm{span}_{x \in A} \nu (x)$,
\eqref{eq:provamethodPoincare} gives that $\vV_0$ must be the zero vector. This contradicts \eqref{eq:4_PoincareBS} and proves the theorem.
\end{proof}
\begin{rem}
{\rm
(i) A sufficient condition for the validity of \eqref{eq:condition Poincare span RN} is the existence of a point $\Tilde{x} \in A\subseteq \pa G$ around which $\pa G$ is of class $C^2$ and such that all the principal curvatures of $\pa G$ at $\Tilde{x}$ are non-zero. In particular, a point on $A$ of strict convexity is sufficient to guarantee the validity of \eqref{eq:condition Poincare span RN}.

(ii) As mentioned in the Introduction,
the previous theorem will be exploited with $A:=\Ga_1$ and $G:=\Si\cap\Om$ in the notation of Setting \ref{Setting}. As shown in Proposition \ref{prop:nondegenerate point}, in this case the existence of a transversally nondegenerate point on $\Ga_1$ is sufficient for the validity of \eqref{eq:condition Poincare span RN}. In particular, this is always the case if $\Si$ is a strictly convex cone (and $\Ga_1 \neq \varnothing$).
}
\end{rem}

More in general, a direct inspection of the proof of Theorem \ref{thm:Poincare new RN} shows that the following more general statement holds true.

\begin{thm}\label{thm:Poincare new in general}
Removing the assumption \eqref{eq:condition Poincare span RN}, inequality \eqref{eq:Poincare new RN} remains true for any
$$\vV: G \to \mathrm{span}_{x \in A} \, \nu (x) \subseteq \RR^N$$
belonging to $W^{1,p} (G)$ and such that $\langle \vV ,\nu  \rangle =0 $ a.e. on $A$.
\end{thm}

In the case $p < N$, we have the following strengthened version of \eqref{eq:Poincare new RN}, which will be useful in the sequel.
\begin{thm}\label{thm:Strengthened Poincare new RN}
Let $G \subset \RR^N$ be a bounded Lipschitz domain and let $A \subseteq \pa G$ be a relatively open subset of $\pa G$ with positive $N-1$-dimensional measure on which $\nu(x)$ is continuous.
Let $r, p$ be two numbers satisfying \eqref{eq:r p al in Hurri}.

Then, there exists a positive constant $ \eta_{r,p}(A,G)$ (depending on $N$, $r$, $p$, $A$ and $G$) such that
\begin{equation}\label{eq:Strengthened Poincare new RN}
\nr \vV \nr_{L^{r}(G)} \le \eta_{r,p}(A,G)^{-1} \, \nr  D \vV \nr_{L^p(G)} ,
\end{equation}
for any vector function
$\vV: G \to \mathrm{span}_{x \in A} \, \nu (x) \subseteq \RR^N$ belonging to $W^{1,p} (G)$
and such that $\langle \vV , \nu \rangle = 0$ a.e. in $A$.
Moreover, we have that
\begin{equation*}
\eta_{r,p}(A,G)^{-1} 
\le 
\max\left\lbrace 
|G|^{ \frac{1}{r} - \frac{1}{p} } \, \, \eta_{p}(A,G)^{-1} 
, \,
\mu_{r, p} (G)^{-1} 
\right\rbrace ,
\end{equation*}
where $\mu_{r, p} (G)^{-1}$ and $\eta_{p}(A,G)^{-1}$ are those appearing in Lemma \ref{lem:Hurri} and Theorem \ref{thm:Poincare new in general}.
\end{thm}
\begin{proof}
%
%
By using \eqref{eq:Sobolev weighted} on each component of $\vV=(v_1, \dots, v_N)$ we have that
$$
\nr v_i  \nr_{L^r (G)} 
\le 
2
\max\left\lbrace
|G|^{ \frac{1}{r} - \frac{1}{p} } \, \nr v_i \nr_{L^p(G)} 
, \,
\mu_{r, p} (G)^{-1} \, \left( \sum_{j=1}^N \nr  (v_i)_j \nr^p_{L^p(G)} \right)^{\frac{1}{p}}
\right\rbrace .
$$
Raising to the power of $r$, summing up for $i=1, \dots ,N$, and then raising to the power of $1/r$,
we obtain that
$$
\nr \vV  \nr_{L^r (G)} 
\le 
2
\max\left\lbrace
|G|^{ \frac{1}{r} - \frac{1}{p} } \, \nr \vV \nr_{L^p(G)} 
, \,
\mu_{r, p} (G)^{-1} \, \nr  D \vV \nr^p_{L^p(G)} 
\right\rbrace ,
$$
where we used the inequality
\begin{equation}\label{eq:ineq for Sobolev norm equivalence}
\sum_{i=1}^N x_i^{\frac{r}{p}} \le \left( \sum_{i=1}^N x_i \right)^{\frac{r}{p}} 
\end{equation}
which holds for every $(x_1, \dots, x_N) \in \RR^N$ with $x_i \ge 0$ for $i=1, \dots, N$, since $r/p \ge 1$.

Thus, \eqref{eq:Strengthened Poincare new RN} follows by \eqref{eq:Poincare new RN} in Theorem \ref{thm:Poincare new in general}.
\end{proof}

When the dimension of $\mathrm{span}_{x \in A} \, \nu (x)$ equals $1$, Theorem \ref{thm:Poincare new in general} reduces to a standard Poincar\'e-type inequality for scalar valued functions with zero trace on a subset of the boundary with positive $N-1$-dimensional measure\footnote{The scalar function would be the component of $\vV$ in the direction spanned by $\nu(x)$ for $x \in A$.}; therefore, the constant may be
%
%
estimated by exploiting the following result.

\begin{lem}[Poincar\'e inequality for functions with zero trace]\label{lem:Poincare zero trace}
Let $G \subset \RR^N$ be a bounded Lipschitz domain, and let $A \subseteq \pa G$ be a subset of $\pa G$ with positive $N-1$-dimensional measure. Consider two numbers $r, p$ satisfying either
\begin{equation}\label{eq:case r=p}
1 \le r=p <\infty
\end{equation}
or \eqref{eq:r p al in Hurri}.

Then, there exists a positive constant $c$ such that
\begin{equation}
\label{eq:Poincare per zero trace}
\nr v \nr_{L^r (G)} \le c \, \nr  \na v  \nr_{L^p(G)},
\end{equation}
for every function $v: G \to \RR$ belonging to $W^{1,p} (G)$ and having zero trace on $A$.
The constant $c$ can be explicitly estimated in terms of $N$, $r$, $p$, $|G|/|A|$, the constant $\mu_{ p}(G)^{-1}$
%
%
(and $\mu_{r, p}(G)^{-1}$ and $|G|$ in case \eqref{eq:r p al in Hurri}) and the constant $\la_p(A)$ of the trace inequality of the embedding $W^{1,p}(G) \hookrightarrow L^p(A)$. 
\end{lem}
\begin{proof}
We start proving the inequality in the case $1 \le r=p < \infty$.

An application of the triangle inequality gives that
\begin{equation*}
\nr v - v_A \nr_{L^p(G)} \le \nr v - v_G \nr_{L^p(G)} + \nr v_G - v_A \nr_{L^p(G)} ,
\end{equation*}
and hence, by \eqref{eq:BoasStraube-poincare},
\begin{equation}\label{eq:step Poincare trace}
\nr v - v_A \nr_{L^p(G)} \le \mu_{p}(G)^{-1} \nr  \na v \nr_{L^p (G)}   + \nr v_G - v_A \nr_{L^p(G)} ,
\end{equation}
so that we are left to estimate the last summand in \eqref{eq:step Poincare trace}.

For any $\la \in \RR$ we compute that
$$
\nr v_A - \la\nr_{L^p(G)}=|G|^{1/p} |v_A - \la| 
\le \frac{|G|^{1/p}}{|A|} \int_A |v-\la|
\le \left( \frac{|G|}{|A|} \right)^{\frac{1}{p}} \, \nr v- \la \nr_{L^p (A)} ,
$$
where the equality follows from the fact that $|v_A - \la|$ is constant, while in the last inequality we used H\"older's inequality.
Using this with $\la=v_{G}$ gives that
$$
\nr v_A - v_G \nr_{L^p(G)} \le \left( \frac{|G|}{|A|} \right)^{\frac{1}{p}} \, \nr v- v_G \nr_{L^p (A)}
$$

Since the trace embedding $W^{1,p}(G) \hookrightarrow L^p(A)$ and \eqref{eq:BoasStraube-poincare} give that
\begin{equation}\label{eq:step altro embedding trace poincare}
\nr v - v_G \nr_{L^p(A)} 
\le \la_p(A) \nr v - v_G \nr_{W^{1,p} (G)} 
\le \la_p(A) \left( 1 + \mu_{p}(G)^{-p} \right)^{1/p} \nr  \na v \nr_{L^p (G)} ,
\end{equation}
the last inequality leads to 
$$
\nr v_A - v_G \nr_{L^p(G)} \le \left( \frac{|G|}{|A|} \right)^{\frac{1}{p}} \, \la_p(A) \left( 1 + \mu_{p}(G)^{-p} \right)^{1/p} \nr  \na v \nr_{L^p (G)} ,
$$
which plugged into \eqref{eq:step Poincare trace} gives that \eqref{eq:Poincare per zero trace} holds true with
\begin{equation}\label{eq:stima costante Poincare zero trace r=p}
c \le \mu_{p}(G)^{-1} + \left( \frac{|G|}{|A|} \right)^{\frac{1}{p}} \, \la_p(A) \left( 1 + \mu_{p}(G)^{-p} \right)^{1/p} ,
\end{equation}
whenever $1 \le r=p < \infty$.

If $r$ and $p$ are as in \eqref{eq:r p al in Hurri}, we can put together \eqref{eq:Sobolev weighted} and the version of \eqref{eq:Poincare per zero trace} that we just proved to obtain that \eqref{eq:Poincare per zero trace} holds true with
$$
c \le |G|^{\frac{1}{r} - \frac{1}{p}} \left( \mu_{p}(G)^{-1} + \left( \frac{|G|}{|A|} \right)^{\frac{1}{p}} \, \la_p(A) \left( 1 + \mu_{p}(G)^{-p} \right)^{1/p} \right) + \mu_{r,p}(G)^{-1} .
$$
\end{proof}

\section{Estimates for the torsional rigidity relative to a cone}\label{sec: pointwise estimates for torsion}
We now return to consider Setting \ref{Setting} and introduce a useful comparison principle for mixed boundary value problems in cones.

\begin{lem}\label{lem:comparison in cones}
Let $\Si \cap \Om$ as in Setting \ref{Setting}. Let $f \in C^2(\Si \cap \Om)\cap C^1((\Si\cap \Om)\cup\Ga_0\cup\Ga_1)$ satisfy
\begin{equation*}
\begin{cases}
\De f \le 0 \quad & \text{ in } \Si\cap\Om
\\
f \ge 0     \quad & \text{ on } \Ga_0
\\
f_\nu \ge 0 \quad & \text{ on } \Ga_1 .
\end{cases}
\end{equation*}
If $f \in L^\infty (\Si\cap\Om) \cap W^{1,2} (\Si\cap\Om)$, then $f \ge 0$.
\end{lem}
\begin{proof}
Set $f^-$ to be the negative part of $f$. By using the divergence theorem and that $\De f \le 0$ in $\Si\cap\Om$, we compute
$$
- \int_{\Si\cap\Om} \left| \na f^- \right|^2 \, dx \ge \int_{\Si\cap\Om} \dv \left( f^- \, \na f \right) \, dx = \int_{\Ga_1} f^- \, f_\nu \, dS_x \ge 0 ,
$$
where in the last equality we used that $f^-=0$ on $\Ga_0$ and in the last inequality that $f_\nu \ge 0$ on $\Ga_1$. Hence, we have that $\int_{\Si\cap\Om} \left| \na f^- \right|^2 \, dx \le 0$. Thus, Poincar\'e's inequality (\eqref{eq:Poincare per zero trace} with $v:=f^-$, $G:=\Si\cap\Om$, $A:=\Ga_0$, $r=p:=2$) gives that $f^- \equiv 0$ in $\Si\cap\Om$.
\end{proof}

\begin{lem}
	\label{lem:relation u dist general}
	Let $\Si$ be a cone and let $u$ be the solution of \eqref{eq:problem torsion}.
	We have that
	\begin{equation}\label{eq:relation u dist general}
		-u(x)\ge\frac12\,\de_{\pa(\Si\cap\Om)}(x)^2 \ \mbox{ for every } \ x\in\ol{\Si\cap\Om} ,
	\end{equation}
where $\de_{\pa(\Si\cap\Om)} (x)$ denotes the distance of $x$ to $\pa(\Si\cap\Om)$.

If $\Si$ is a convex cone, then we have that
\begin{equation}\label{eq:non serve ma serve inproof reldist finer}
	-u(x) \ge \frac{1}{2}\,\de_{\Ga_0} (x)^2  \ \mbox{ for every } \ x\in\ol{\Si\cap\Om} ,
\end{equation}	
where $\de_{\Ga_0} (x)$ denotes the distance of $x$ to $\Ga_0$.
\end{lem}
\begin{proof}
By Lemma \ref{lem:comparison in cones} with $f:=-u$, we know that $u \le 0$.
Now, for $x\in \Si\cap\Om$, let $\de:=\de_{\pa(\Si\cap\Om)} (x)$ and consider the ball $B_\de(x)$ with radius $\de$ centered at $x$. 
Consider $w(y)=(|y-x|^2- \de^2)/2$, that is the solution of the classical torsion problem in $B_\de (x)$
	\begin{equation}\label{eq:torsionball}
		\De w =N \text{  in } B_\de (x), \, w =0 \text{  on } \pa B_\de (x).
	\end{equation}
By (classical) comparison we have that $w\ge u$ on $\ol{B}_\de (x)$ and hence, in particular, $w(x) \ge u(x)$. Inequality \eqref{eq:relation u dist general} easily follows.

We now prove \eqref{eq:non serve ma serve inproof reldist finer}. 
For $x\in \ol{\Si}\cap\Om$, we set $\de:=\de_{\Ga_0} (x)$ and consider the ball $B_\de(x)$ with radius $\de$ centered at $x$. 
Setting $w(y)=(|y-x|^2- \de^2)/2$, we find that
\begin{equation*}
	\begin{cases}
		\De (w-u) = 0 \quad & \text{ in } \Si \cap B_{\de} (x)
		\\
		w-u \ge 0 \quad & \text{ on } \Si \cap \pa B_{\de} (x)
		\\
		w_\nu - u_\nu \ge 0 \quad & \text{ on } \pa \Si \cap B_{\de} (x) ,
	\end{cases}
\end{equation*}
where the last boundary condition follows noting that $\langle y - x, \nu (y) \rangle \ge 0$ on $\pa \Si \cap B_{\de} (x)$, being as $\Si \cap B_{\de} (x)$ (convex and hence) star-shaped with respect to $x \in \ol{\Si} \cap B_{\de} (x)$.
By comparison (Lemma \ref{lem:comparison in cones} with $f:=w-u$) we have that $w \ge u$ in $\Si\cap B_{\de}(x)$, and hence, in particular $w(x) \ge u(x)$, that is \eqref{eq:non serve ma serve inproof reldist finer}.
\end{proof}


\subsection{Hopf Lemma relative to a convex cone}\label{subsec:Hopf}
We introduce the notion of {\it uniform interior sphere condition relative to the cone} $\Sigma$.
\begin{definit}\label{def:interior sphere relative to cone}
We say that $\Si\cap\Om$ satisfies the $\ul{r}_i$-uniform interior sphere condition relative to the cone $\Si$, if for each $x \in \ol{\Ga}_0$ there exists a touching ball
%
%
%
of radius $\ul{r}_i$ such that

(i) its center $x_0$ is contained in $\ol{\Si \cap \Om}$

and

(ii) its closure intersects $\ol{\Ga}_0$ only at $x$.
\end{definit}

Since in our setting $\Ga_0$ is smooth, we must have that 
\begin{equation}\label{eq:centro interior sphere}
x_0=x- \ul{r}_i \, \nu(x)  \quad \text{ for } x \in \Ga_0 ; 
\end{equation}
however, notice that this may not be the only possibility for the points on $\pa \Ga_0$.

When $\Si=\RR^N$ Definition \ref{def:interior sphere relative to cone} returns the classical uniform interior sphere condition.
When $\Si \subsetneq \RR^N$, the condition in Definition \ref{def:interior sphere relative to cone} is related to how $\ol{\Ga}_0$ and $\pa\Si$ intersect. In fact, a necessary condition for the validity of (i) and (ii) is that $\langle \nu_{\Ga_0} (x) , \nu_{\pa\Si}(x) \rangle \ge 0$ for $\cH^{N-2}$-a.e. $x \in \pa \Ga_0$.
%

\begin{lem}[Hopf lemma relative to $\Si$]
\label{lem:Hopf in cones}
Let $u$ be the solution of \eqref{eq:problem torsion}.
If $\Si$ is a convex cone and $\Si\cap\Om$ satisfies the $\ul{r}_i$-uniform interior sphere condition with radius $\ul{r}_i$, then we have that
	\begin{equation}
		\label{eq:Hopf relative to cone}
		| \na u| \ge \ul{r}_i  \ \mbox{ for every } x \in \Ga_0 .
	\end{equation}
\end{lem}
\begin{proof}
For any $x \in \Ga_0$, let $B_{\ul{r}_i} (x_0)$ be the touching ball of radius $\ul{r}_i$ centered at the point $x_0=x- \ul{r}_i \, \nu(x)$. Setting $w(y)=\left(|y - x_0|^2 - \ul{r}_i^2 \right)/2$, we have that
\begin{equation}\label{eq:system for comparison Hopf}
\begin{cases}
\De (w-u) = 0 \quad & \text{ in } \Si \cap B_{\ul{r}_i} (x_0)
\\
w-u \ge 0 \quad & \text{ on } \Si \cap \pa B_{\ul{r}_i} (x_0)
\\
w_\nu - u_\nu \ge 0 \quad & \text{ on } \pa \Si \cap B_{\ul{r}_i} (x_0) .
\end{cases}
\end{equation}
The last boundary condition holds true since $\langle y - x_0, \nu (y) \rangle \ge 0$ on $\pa \Si \cap B_{\ul{r}_i} (x_0)$, being as $\Si \cap B_{\ul{r}_i} (x_0)$ (convex and hence) star-shaped with respect to $x_0 \in \ol{\Si} \cap B_{\ul{r}_i} (x_0)$.
By comparison (Lemma \ref{lem:comparison in cones} with $f:=w-u$) we have that $w \ge u$ in $\Si\cap B_{\ul{r}_i}(x_0)$, and hence, since $u(x)=w(x)=0$, we obtain
	$$
	u_\nu(x) \ge w_\nu (x).
	$$
Since $w_\nu (x)=\ul{r}_i$, the conclusion follows.
\end{proof}

\begin{rem}
{\rm We recall that, when $\Si=\RR^N$, Hopf-type results remain true in the more general case where $\Om$ satisfies an interior pseudoball condition (see \cite{ABMMZ}). Corresponding conditions may be introduced when $\Si \subsetneq \RR^N$. 

We recall that, when $\Si=\RR^N$, interior sphere and pseudoball conditions also lead to linear growth to the boundary (see, e.g., \cite[Lemma 3.1]{MP2} and \cite[Step 2 in the proof of Theorem I]{CPY}): when $\Si \subsetneq \RR^N $ additional conditions may be requested to get linear growth from $\Ga_0$.  
}
\end{rem}

\subsection{Upper bound for the gradient in a convex cone}\label{subsec:upper bound gradient}

In this section we introduce a sufficient geometrical condition that  guarantees the boundedness of $\nr \na u \nr_{L^\infty (\Ga_0)}$, and allow to obtain an explicit upper bound.
In fact, in analogy to the interior sphere condition relative to the cone introduced above, we now introduce the notion of {\it uniform exterior sphere condition relative to the cone}.
\begin{definit}\label{def:exterior sphere relative to cone}
We say that $\Si\cap\Om$ satisfies the $\ul{r}_e$-uniform exterior sphere condition relative to the cone $\Si$, if for each $x \in \ol{\Ga}_0$ there exists a touching ball
%
%
%
of radius $\ul{r}_e$ such that

(i) its center $x_0$ is contained in $\ol{\Si} \setminus \ol{\Si\cap\Om}$

and

(ii) its closure intersects $\ol{\Ga}_0$ only at $x$.
\end{definit}

Since in our setting $\Ga_0$ is smooth, we must have that 
\begin{equation}\label{eq:centro exterior sphere}
x_0 = x + \ul{r}_e \, \nu(x)  \quad \text{ for } x \in \Ga_0 ; 
\end{equation}

When $\Si=\RR^N$ Definition \ref{def:exterior sphere relative to cone} returns the classical uniform exterior sphere condition.
When $\Si \subsetneq \RR^N$, the condition in Definition \ref{def:exterior sphere relative to cone} is related to how $\ol{\Ga}_0$ and $\pa\Si$ intersect. In fact, a necessary condition for the validity of (i) and (ii) is that $\langle \nu_{\Ga_0}(x) , \nu_{\pa\Si} (x) \rangle \le 0$ for $\cH^{N-2}$-a.e. $x \in \pa \Ga_0$.

\begin{lem}\label{lem:upper bound gradient on Ga0}
Let $u$ be the solution of \eqref{eq:problem torsion}.	
If $\Si$ is a convex cone and $\Si\cap\Om$ satisfies the $\ul{r}_e$-uniform exterior sphere condition with radius $\ul{r}_e$, then we have that
	\begin{equation}
		\label{eq:upper bound gradient Ga0 relative to cone}
		\nr \na u \nr_{L^\infty (\Ga_0)} \le 
\begin{cases}
\displaystyle
6 \, \ul{r}_e \, \left( 1+ \frac{d_{\Si\cap\Om}}{ \ul{r}_e } \right)^4 \quad & \text{for } N=2 ,
\\
\displaystyle
\frac{3N}{2} \, \ul{r}_e \, \left( 1+ \frac{d_{\Si\cap\Om}}{ \ul{r}_e } \right)^N
\quad & \text{for } N\ge3 .
\end{cases} 
\end{equation}
\end{lem}
\begin{proof}
For any $x \in \Ga_0$, let $B_{\ul{r}_e} (x_0)$ be the exterior touching ball of radius $\ul{r}_e$ centered at the point $x_0=x + \ul{r}_e \, \nu(x) \in \ol{\Si} \setminus \ol{\Si\cap\Om}$. For $\mathcal{R}>\ul{r}_e$ (to be chosen later), we define
\begin{equation*}
w(y)=
\begin{cases}
\displaystyle\frac12\, |y-x_0|^2 +\frac{ \mathcal{R}^2}{2}\,(1-\ka^2)\,\frac{\log(|y-x_0|/\ul{r}_e)}{\log\ka} -\frac{ \ul{r}_e^2 }{2} 
\ &\mbox{ for } \ N=2, 
\vspace{5pt} \\
\displaystyle\frac12\,|y-x_0|^2 +\frac12\,\frac{\mathcal{R}^2}{1-\ka^{N-2}}\,\left\{ (1-\ka^2)\,(|y-x_0|/\ul{r}_e)^{2-N}+\ka^N-1\right\} \ &\mbox{ for } \ N \ge 3 ,
\end{cases}
\end{equation*} 
where $\ka=\ul{r}_e / \mathcal{R}$. This is the torsion function in the annulus $A_{\ul{r}_e, \mathcal{R}}(x_0)$ centered at $x_0$ of radii $\ul{r}_e$ and $\mathcal{R}$, i.e.,
$$\Delta w = N \ \textrm{ in } \ A, \quad w = 0 \ \textrm{ on } \ \pa A.$$
Notice that $w$ is radial and its gradient satisfies
\begin{equation}\label{eq:step upper bound gradient}
\na w = 
\begin{cases}
|\na w|\frac{x_0-y}{|y-x_0|} \quad & \text{ for } \ul{r}_e \le |y-x_0| \le \mathcal{Z} 
\\
|\na w|\frac{y-x_0}{|y-x_0|} \quad & \text{ for }  \mathcal{Z} \le |y-x_0| \le \mathcal{R} ,
\end{cases}
\end{equation}
where
\begin{equation*}
\mathcal{Z}= 
\begin{cases}
\left[ \frac{ \mathcal{R}^2}{2} \frac{1-\ka^2}{\log(\frac{1}{\ka})} \right]^{1/2}
\quad & \text{for } N=2 ,
\\
\left[ \frac{N-2}{2} \ul{r}_e^{N-2} \mathcal{R}^2 \frac{1-\ka^2}{1-\ka^{N-2}} \right]^{1/N} 
\quad & \text{for } N\ge 3;
\end{cases}
\end{equation*}
the sphere centered at $x_0$ of radius $\mathcal{Z}$ is made of critical points of $w$.
Setting
\begin{equation}\label{eq:choice mathcalR in upperbound}
\mathcal{R}=
\begin{cases}
\displaystyle
2 \, \ul{r}_e \left( 1+ \frac{d_{\Si\cap\Om}}{ \ul{r}_e } \right)^2
\quad & \text{for } N = 2 ,
\\
\displaystyle
\sqrt{3} \, \, \ul{r}_e \, \left( 1+ \frac{ d_{\Si\cap\Om} }{ \ul{r}_e} \right)^{N/2} \quad & \text{for } N \ge 3,
\end{cases}
\end{equation} 
ensures that\footnote{ 
In fact, being as $0<\kappa= \ul{r}_e/ \mathcal{R}$ one can easily check that, with the choice in \eqref{eq:choice mathcalR in upperbound}, we have that
\begin{equation*}
\mathcal{Z} = \left[ \frac{\ul{r}_e \mathcal{R}}{2} \, \, \frac{1-\ka^2}{\ka \log(\frac{1}{\ka})} \right]^{1/2} 
\ge
\left[ \frac{\ul{r}_e \mathcal{R}}{2} (1+\ka) \right]^{1/2}
\ge \left[ \frac{\ul{r}_e \mathcal{R}}{2} \right]^{1/2} 
= d_{\Si\cap\Om} + \ul{r}_e \quad \text{for } N=2,
\end{equation*}
and
\begin{equation*}
\mathcal{Z} = \left[ \frac{N-2}{2} \ul{r}_e^{N-2} \mathcal{R}^2 \frac{1-\ka^2}{1-\ka^{N-2}} \right]^{1/N} \ge
\left[ \frac{1}{2} \, \ul{r}_e^{N-2} \mathcal{R}^2 (1-\ka^2)\right]^{1/N}
\ge
\left[ \frac{1}{3} \, \ul{r}_e^{N-2} \mathcal{R}^2 \right]^{1/N} 
= d_{\Si\cap\Om} + \ul{r}_e \quad \text{for } N \ge 3;
\end{equation*}
in the last inequality for $N\ge3$ we also used that $1-\ka^2 \ge 3/2$, which easily follows from the definition of $\ka$, \eqref{eq:choice mathcalR in upperbound}, and the trivial inequality $\left( 1+ \frac{ d_{\Si\cap\Om} }{ \ul{r}_e} \right)^{N}\ge 1$.
}
$\mathcal{Z} \ge d_{\Si\cap\Om} + \ul{r}_e$, and hence $A_{\ul{r}_e, \mathcal{R}} (x_0) \supset A_{\ul{r}_e,\mathcal{Z}} (x_0) \supset \Si\cap\Om$.

Thus, by recalling \eqref{eq:step upper bound gradient}, for any $y\in \Ga_1 \subset \ol{\Si\cap\Om}$ we have that 
$$
\langle \na w,\nu \rangle = - |\na w| \langle y-x_0 ,\nu (y) \rangle \le 0 ,$$
being as $B_{\mathcal{R}(x_0)} \cap \Si$ (convex and hence) star-shaped with respect to $x_0 \in B_{\mathcal{R}(x_0)} \cap \ol{\Si}$.
Thus,
we have that
\begin{equation}\label{eq:system for comparison upper bound}
\begin{cases}
\De (u-w) = 0 \quad & \text{ in } \Si \cap \Om
\\
u-w \ge 0 \quad & \text{ on } \Ga_0
\\
u_\nu - w_\nu  \ge 0 \quad & \text{ on } \Ga_1 .
\end{cases}
\end{equation}

By comparison (Lemma \ref{lem:comparison in cones} with $f:=u-w$), we have that $u \ge w$ in
$\Si\cap \Om$,
and hence, since $u(x)=w(x)=0$, we obtain
	$$
	u_\nu(x) \le w_\nu (x).
	$$
Since direct computations
(see \cite[proof of Theorem 3.10]{MP})
show that
$$
w_\nu (x) \le \frac{\max\left\lbrace 3,N \right\rbrace }{2} \, \frac{\mathcal{R}( \mathcal{R} -\ul{r}_e)}{\ul{r}_e} ,$$
and, obviously,
$$
\frac{\max\left\lbrace 3,N \right\rbrace }{2} \, \frac{\mathcal{R}( \mathcal{R} -\ul{r}_e)}{\ul{r}_e} 
\le 
\frac{\max\left\lbrace 3,N \right\rbrace }{2} \, \frac{\mathcal{R}^2}{\ul{r}_e}
= 
\begin{cases}
6 \, \ul{r}_e \, \left( 1+ \frac{d_{\Si\cap\Om}}{ \ul{r}_e } \right)^4 \quad & \text{for } N=2
\\
\frac{3N}{2} \, \ul{r}_e \, \left( 1+ \frac{d_{\Si\cap\Om}}{ \ul{r}_e } \right)^N
\quad & \text{for } N\ge3
\end{cases} ,
$$
the conclusion follows.
\end{proof}

The next Lemma shows that the explicit bound for $\nr \na u \nr_{L^\infty (\Ga_0)}$ obtained in Lemma \ref{lem:upper bound gradient on Ga0} is also an explicit bound for $\nr \na u \nr_{L^\infty (\Si\cap\Om)}$, whenever $u \in W^{1,\infty} (\Si \cap \Om) \cap W^{2,2} (\Si\cap\Om)$. We recall that global $W^{1,\infty}$ regularity for the Neumann problem in convex domains can be found in \cite{Mazya}.

\begin{lem}\label{lem:upper bound gradient on Ga0 = whole domain}
Let $u$ be the solution of \eqref{eq:problem torsion}.	
If $u \in W^{1,\infty} (\Si \cap \Om) \cap W^{2,2} (\Si\cap\Om)$, then we have that
\begin{equation*}
\nr \na u \nr_{L^\infty (\Si\cap\Om)} = \nr \na u \nr_{L^\infty (\Ga_0)} .
\end{equation*}
\end{lem}
\begin{proof}
The conclusion
follows by applying Lemma \ref{lem:comparison in cones} to the function $f:= \frac{1}{2} \left( \nr | \na u|^2  \nr_{L^\infty(\Ga_0)} - | \na u|^2 \right)$ which satisfies
\begin{equation*}
\begin{cases}
\De f = - | \na ^2 u|^2 \le 0 \quad & \text{ in } \Si\cap\Om
\\
f = \frac{1}{2} \left( \nr | \na u|^2  \nr_{L^\infty(\Ga_0)} - | \na u|^2 \right)  \ge 0     \quad & \text{ on } \Ga_0
\\
f_\nu = - \langle \na^2 u \na u, \nu \rangle  \ge 0 \quad & \text{ on } \Ga_1 ,
\end{cases}
\end{equation*}
where the last inequality follows from \eqref{eq:disuguaglianza puntuale gratis per convex}.
\end{proof}

	
	\subsection{A bound for \texorpdfstring{$\max(-u)$}{max(-u)}}\label{subsec:bound for max(-u)}
	It is clear that putting together Lemma \ref{lem:upper bound gradient on Ga0} and Lemma \ref{lem:upper bound gradient on Ga0 = whole domain} would give an explicit estimate for $\max(-u)$. In fact, being as $u=0$ on $\Ga_0$, we have that
	$$\max\limits_{\ol{\Si\cap\Om}} (-u) \le \nr \na u \nr_{L^\infty(\Si\cap\Om)} \, d_{\Si\cap \Om} .$$
	However, notice that Lemma \ref{lem:upper bound gradient on Ga0} 
	requires an exterior sphere condition, which is not necessary to give an estimate for $\max(-u)$. 
	
	A general explicit estimate of $\max(-u)$ in terms of the dimension $N$ and the diameter $d_{\Si\cap\Om}$ only is provided in the following lemma.
	
	\begin{lem}
		\label{lem:max stima con diametro}
		Let $\Si$ be a cone and $u$ be the solution of \eqref{eq:problem torsion}. 
		
		(i) If $\Si$ is convex, then we have that
		\begin{equation}\label{eq:max stima generale con N e diametro}
			\max\limits_{\ol{\Si\cap\Om}} (-u ) \le C(N, d_{\Si\cap\Om}) ,
		\end{equation}
	where $C(N, d_{\Si\cap\Om})$ is an explicit constant only depending on $N$ and $d_{\Si\cap\Om}$.
		
		(ii) 
		If there exists a point $x_0 \in \ol{\Si\cap\Om}$ satisfying $\langle x- x_0 , \nu(x)\rangle = 0$ on $\Ga_1$, 
		then we have that
		\begin{equation}\label{eq:max stima con diametro}
			\max\limits_{\ol{\Si\cap\Om}} (-u ) \le \frac{d_{\Si\cap\Om}^2}{2}.
		\end{equation}

	\end{lem}
	\begin{proof}
		(i) Without loss of generality, we can assume that the direction $e_N$ is internal to $\ol{\Si}$. Set 
		$\ol{x}=(\ol{x}_1, \dots, \ol{x}_N)$ to be a point in $\ol{\Si\cap\Om}$ such that
		$$
		\ol{x}_N=\max \left\lbrace y_N \, : \, y=(y_1,\dots,y_N) \in \ol{\Si\cap\Om} \right\rbrace .
		$$
		Hence, define $x_0=\ol{x}+e_N$ and notice that, by construction, $x_0\in \ol{\Si}\setminus \ol{\Om}$ and the unit ball $B(x_0)$ centered at $x_0$ is contained in $\RR^N \setminus \ol{\Si\cap\Om}$ and $\ol{B}(x_0)$ touches $\ol{\Ga}_0$ at $\ol{x}$. Thus, we define $w$ to be the torsion function in the annulus $A_{1, \mathcal{R}} (x_0)$. Following the computations in the proof of Lemma \ref{lem:upper bound gradient on Ga0} with the only difference that now $r_e$ is replaced by $1$, we find that the choice 
		\begin{equation*}
			\mathcal{R}=
			\begin{cases}
				\displaystyle
				2  \left( 1+ d_{\Si\cap\Om} \right)^2
				\quad & \text{for } N = 2 ,
				\\
				\displaystyle
				\sqrt{3} \,  \, \left( 1+  d_{\Si\cap\Om} \right)^{N/2} \quad & \text{for } N \ge 3,
			\end{cases}
		\end{equation*} 
		ensures that
		$\langle \na w , \nu \rangle = - |\na w| \langle y-x_0 , \nu (y) \rangle \le 0$ on $\Ga_1$, and hence \eqref{eq:system for comparison upper bound}. By comparison (Lemma \ref{lem:comparison in cones} with $f:=u-w$) we have that $u-w \ge 0$ in $\Si \cap \Om$, and hence,
		$$
		\max\limits_{\ol{\Si\cap\Om}} (-u) \le \max\limits_{\ol{\Si\cap\Om}} (-w) = - w_{|\pa B_{\mathcal{Z}}(x_0)} = C(N,d_{\Si\cap\Om}) , 
		$$
		where
		\begin{equation*}
			\mathcal{Z}= 
			\begin{cases}
				\left[   \frac{\mathcal{R}^{2} -1}{2 \, \log(\mathcal{R})} \right]^{1/2}
				\quad & \text{for } N=2 ,
				\\
				\left[ \frac{N-2}{2}  \frac{ \mathcal{R}^{2} - 1 }{1-\mathcal{R}^{2-N}} \right]^{1/N} 
				\quad & \text{for } N\ge 3 ,
			\end{cases}
		\end{equation*}
		and
		$\pa B_{\mathcal{Z}}(x_0)$ is the sphere of critical points of $w$.

		(ii)
		Let $x_0$ be a point in $\ol{\Si\cap\Om}$ satisfying $\langle x- x_0 , \nu(x)\rangle = 0$ on $\Ga_1$ and let $B_{d_{\Si\cap\Om}} (x_0)$ be the ball of radius $d_{\Si\cap\Om}$ centered at the point $x_0$. It is clear that $\Si\cap\Om \subset B_{d_{\Si\cap\Om}} (x_0)$. Setting $w(y)=\left(|y - x_0|^2 - d_{\Si\cap\Om}^2 \right)/2$, we have that
		\begin{equation*}
			\begin{cases}
				\De (u-w) = 0 \quad & \text{ in } \Si \cap \Om
				\\
				u-w \ge 0 \quad & \text{ on }  \Ga_0
				\\
				u_\nu - w_\nu = 0 \quad & \text{ on } \Ga_1 .
			\end{cases}
		\end{equation*}
		The last boundary condition holds true since $w_\nu(x) = \langle x - x_0, \nu (x) \rangle = 0$ on $\Ga_1$.
		By comparison (Lemma \ref{lem:comparison in cones} with $f:=u-w$) we have that $u-w \ge 0$ in $\Si \cap \Om$, and hence,
		$$
		\max\limits_{\ol{\Si\cap\Om}} (-u) \le \max\limits_{\ol{\Si\cap\Om}} (-w)= \frac{ d_{\Si\cap\Om}^2}{ 2},
		$$
		which is the desired conclusion.
	\end{proof}
	\begin{rem}
		{\rm
			(i) It is clear from the proof, that the conclusion in (i) of Lemma \ref{lem:max stima con diametro} remains true even if the assumption of convexity of $\Si$ is replaced with the weaker assumption that there exists a point $x_0 \in \ol{\Si}\setminus \ol{\Om}$ such that $B(x_0)\subset \RR^N \setminus \ol{\Si\cap\Om}$, $\ol{B}(x_0)$ touches $\ol{\Ga}_0$, and $\langle x- x_0 ,\nu(x) \rangle \ge 0$ on $\Ga_1$.
			
			(ii) Convexity of $\Si$ is not requested in (ii) of Lemma \ref{lem:max stima con diametro}. Notice that, in particular, the assumption in (ii) is surely satisfied if, e.g., either $\Ga_1=\emptyset$ or $\Ga_1\neq \emptyset$ and there exist $x_0\in\Ga_1$ and a cone $\Si_{x_0}$ with vertex in $x_0$ such that $\Si\cap\Om=\Si_{x_0}\cap\Om$.
		}
	\end{rem}

\section{Sharp quantitative rigidity in terms of an \texorpdfstring{$L^2$}{L2}-pseudo-distance}\label{sec: Sharp stability in L2}

\subsection{Setting}
In addition to Setting \ref{Setting}, from now on we assume
that $\Si$ and $\Om$ intersect in a Lipschitz way so that $\Si\cap\Om$ is a Lipschitz domain.


\begin{rem}\label{rem:nota su convexity non necassaria per Stability}
{\rm
As already mentioned, the convexity of the cone ensures that \eqref{eq:disuguaglianza puntuale gratis per convex} and hence \eqref{eq:disuguaglianza senza u gratis per convex} and \eqref{eq:disuguaglianza con u gratis per convex} hold true.
We stress that, as already noticed for the rigidity results, many of the stability results remain true (once that the fundamental identities in Theorems \ref{thm:serrinidentity}, \ref{thm:identitySBT}, \ref{thm:identityheintze-karcher} are established) even if the convexity of $\Si$ is relaxed with one of the weaker conditions in \eqref{eq:disuguaglianza puntuale gratis per convex}, \eqref{eq:disuguaglianza senza u gratis per convex}, \eqref{eq:disuguaglianza con u gratis per convex}, depending on the relevant result.
To simplify the presentation, we preferred to always assume the cone $\Si$ to be convex. This also allows to exploit the
bounds obtained in Sections \ref{subsec:Hopf}, \ref{subsec:upper bound gradient}, and \ref{subsec:bound for max(-u)}.
}
\end{rem}

Setting
$$
k:= \text{ dimension of } \mathrm{span}_{x \in \Ga_1 } \, \nu(x) ,
$$
we have that $0 \le k \le N$. Notice that the case $\Ga_1 = \emptyset $ is included in our treatment (in that case, we have $k=0$).
Up to changing orthogonal coordinates, we can assume that $\mathrm{span}_{x \in \Ga_1} \, \nu(x)$ is the space generated by the first $k$ axes $\mathbf{e}_1, \dots, \mathbf{e}_k$.
Notice that, in this way, if we set $z \in \RR^N$ of the form
\begin{equation}\label{eq:choice of z 1of2}
z=(0, \dots , 0 , z_{k+1}, \dots, z_k) \in \RR^N ,
\end{equation}
it surely satisfies \eqref{eq:inner product z in cone}.
We also fix
\begin{equation}\label{eq:choice of z 2of2}
z_i = \frac{1}{|\Si\cap\Om|} \int_{\Si\cap\Om} ( x_i - u_i(x) ) \, dx \quad \text{for } i=k+1, \dots, N ,
\end{equation}
where $u_i$ denotes the $i$-th partial derivative of $u$ and $x_i$ the $i$-th component of the vector $x \in \RR^N$.
With this choice of $z$, if we consider the harmonic function 
\begin{equation}\label{def: h armonica con q}
h:= q - u , \quad \text{ where } q \text{ is the quadratic function defined in \eqref{quadratic}},
\end{equation}
we have that
\begin{equation}\label{eq:gradienteproiettatosuspan nu}
(h_1, \dots, h_k , 0 , \dots, 0) \in \mathrm{span}_{x \in \Ga_1} \, \nu(x) \subseteq \RR^N ,
\quad \quad
\langle (h_1, \dots, h_k , 0 , \dots, 0), \nu \rangle = \langle \na h , \nu \rangle  = 0 \, \text{ on } \Ga_1 
\end{equation}
and
\begin{equation}\label{eq:gradienteparte ortogonale a span nu}
\int_{\Si \cap \Om} h_i \, dx = 0 \quad \text{for } i = k+1, \dots, N .
\end{equation}
The last identity in \eqref{eq:gradienteproiettatosuspan nu} follows by \eqref{eq:inner product z in cone} and the Neumann condition in \eqref{eq:problem torsion}.

By direct computation, it is easy to check that $|\na^2 h|^2$ equals the Cauchy-Schwarz deficit for $\na^2 u$, that is,
\begin{equation}\label{eq:hessiana h}
|\na^2 h|^2 = | \na^2 u|^2- \frac{(\De u)^2}{N}	\quad \text{ in } \, \Si\cap\Om .
\end{equation}

We now establish the following lemma.
\begin{lem}\label{lem:Mixed pp Poincareaigradienti}
For $1 \le p < \infty$, we have that
\begin{equation*}
\nr \na h \nr_{L^p (\Si\cap\Om)} \le C \, \nr  \na^2 h  \nr_{L^p(\Si\cap\Om)} ,
\end{equation*}
for some positive constant $C$ satisfying $C \le \La_{p}(k)$, where we have set
\begin{equation}\label{def:La p al k}
\La_{p}(k) :=
\begin{cases}
\mu_{p}(\Si\cap\Om)^{-1} \quad & \text{if } k=0
\\
\eta_{p} (\Ga_1 , \Si \cap \Om )^{-1} \quad & \text{if } k=N
\\
\max\left[ \mu_{p}(\Si\cap\Om)^{-1} , \, \eta_{p} ( \Ga_1 ,  \Si \cap \Om )^{-1}  \right] \quad & \text{if } 1 \le k \le N-1 ,
\end{cases}
\end{equation}
where $\mu_{p}(\Si\cap\Om)$ and $\eta_{p} ( \Ga_1 , \Si \cap \Om )$ are those in \eqref{eq:BoasStraube-poincare} and Theorem \ref{thm:Poincare new in general}.
\end{lem}
\begin{proof}
%
%
In light of \eqref{eq:gradienteproiettatosuspan nu}, we can apply \eqref{eq:Poincare new RN} in Theorem \ref{thm:Poincare new in general} with $G:=\Si\cap\Om$, $A:=\Ga_1$, and $\vV:= (h_1, \dots, h_k , 0 , \dots, 0) $ to get that
\begin{equation}\label{eq:dimprovaPoincaregradmixed}
\left( \sum_{i=1}^k \nr h_i \nr_{L^p (\Si\cap\Om)}^p \right)^{1/p}\le \eta_{p}(\Ga_1 , \Si \cap \Om )^{-1} \, \left(  \sum_{i=1}^k \sum_{j=1}^N \nr h_{ij} \nr_{L^p (\Si\cap\Om)}^p \right)^{1/p}.
\end{equation}
In light of \eqref{eq:gradienteparte ortogonale a span nu}, we can apply \eqref{eq:BoasStraube-poincare} to each first partial derivative $h_i$ of $h$, $i=k+1, \dots, N$. Raising to the power of $p$ those inequalities and \eqref{eq:dimprovaPoincaregradmixed}, and then summing up, the conclusion easily follows.
\end{proof}

\begin{rem}
\label{rem:casi k=0 e k=1 evitando eta}
{\rm

(i) As mentioned in Remark \ref{rem:stime mu HS}, $\mu_{p}(\Si\cap\Om)$ can be explicitly estimated in terms of the diameter $d_{\Si\cap\Om}$ and the so-called $b_0$-John parameter, which in turn can be explicitly estimated in terms of the Lipschitz regularity parameters (see Remark \ref{rem:stima b_0} and \cite[Lemma A.2]{MP7}). 

Uniformity of $\eta_{p}(\Ga_1 , \Si \cap \Om )$ in the class of uniformly bounded and uniformly Lipschitz domains $G$ such that $\Ga_1 \subset \ol{G}$ may be obtained following the argument in \cite[Theorem 2]{BouCha}. 

(ii) When $k=0$, we have $\La_{p}(0)=\mu_{p}(\Si\cap\Om)^{-1}$ and hence $\eta_{p}(\Ga_1 , \Si \cap \Om )$ does not come into play.

(iii) In the case $k=1$, an estimate for $\eta_{p}(\Ga_1,\Si\cap\Om)$ is provided by Lemma \ref{lem:Poincare zero trace} (with $v:=h_1$, $A:=\Ga_1$, $G:=\Si\cap\Om$).
}
\end{rem}

\subsection{On trace-type inequalities}
\begin{lem}[Trace-type Poincar\'e inequality for $h-h_{\Si\cap\Om}$ and $\na h$]
\label{lem:traceinequality h e nah con lambda}
We have that
\begin{equation}\label{eq:unweighted trace for h-hOm}
\nr h - h_{\Si\cap\Om} \nr_{L^2(\Ga_0)} \le C \, \nr \na h \nr_{L^2(\Si\cap\Om)} ,
\end{equation}
where $C$ is a positive constant that can be explicitly estimated by means of
\begin{equation*}
C \le \la_2(\Ga_0) \sqrt{1 + \mu_{2}(\Si\cap\Om)^{-2}}
\end{equation*}
where $\la_2(\Ga_0)$ is the constant of the trace embedding $W^{1,2}(\Si\cap\Om) \hookrightarrow L^2(\Ga_0)$ (inequality \eqref{eq:trace inequality intro}) and $\mu_{2}(\Si\cap\Om)^{-1}$ is the constant of the Poincar\'e inequality \eqref{eq:BoasStraube-poincare} (with $p:=2$).
Moreover, we have that
\begin{equation}\label{eq:unweighted trace for nah}
\nr \na h \nr_{L^2(\Ga_0)} \le C \, \nr \na^2 h \nr_{L^2(\Si\cap\Om)} ,
\end{equation}
where $C$ is a positive constant that can be explicitly estimated by means of
\begin{equation*}
C \le \la_2(\Ga_0) \sqrt{1 + \La_{2}(k)^{2}} ,
\end{equation*}
where the constant $\La_{2}(k)$ is that defined in \eqref{def:La p al k} 
(with $p:=2$).
\end{lem}
\begin{proof}
Inequality \eqref{eq:unweighted trace for h-hOm} follows by using \eqref{eq:trace inequality intro} with $v:= h - h_{\Si\cap\Om}$ and \eqref{eq:BoasStraube-poincare} with $p:=2$.

Inequality \eqref{eq:unweighted trace for nah} follows by using \eqref{eq:trace inequality intro} with $v:=h_i$ (for $i=1,\dots, N$) and Lemma \ref{lem:Mixed pp Poincareaigradienti} with $p:=2$.
\end{proof}
In what follows we show that we can obtain explicit ad hoc trace-type inequalities for $h$ and $\na h$ whenever we have at our disposal
a positive lower bound $\lbunu$ for $|\na u|$ on $\Ga_0$, i.e.,
\begin{equation}\label{def:lower bound unu}
u_\nu \ge \lbunu > 0 \quad \text{on } \Ga_0 .
\end{equation}
This approach allows to avoid the use in our quantitative estimates of the general trace constant $\la_2(\Ga_0)$ appearing in \eqref{eq:trace inequality intro}. 
Furthermore, it allows to obtain refined weighted trace inequality for $h$ and $\na h$, which will be useful to obtain analogous stability results for Serrin's problem: this will be addressed in a forthcoming paper. 

\begin{rem}\label{rem:sfera interna relativa allora lbunu=ulri}
{\rm The uniform interior sphere condition relative to $\Si$ introduced in Definition \ref{def:interior sphere relative to cone} is a geometrical condition that guarantees the validity of \eqref{def:lower bound unu}. In fact, if $\Si\cap\Om$ satisfies the $\ul{r}_i$-uniform interior sphere condition relative to $\Si$, then by Lemma \ref{lem:Hopf in cones} we have that \eqref{def:lower bound unu} holds true with $\ul{m}:=\ul{r}_i$.
}
\end{rem}

\begin{lem}[Weighted trace inequality for $h-h_{\Si\cap\Om}$ and $\na h$]
\label{lem:weighted trace inequality ad hoc}
For any $z \in \RR^N$ satisfying \eqref{eq:inner product z in cone}, consider $h=q-u$, where $q$ is defined as in \eqref{quadratic}.
Let $\lbunu$ be the lower bound defined in \eqref{def:lower bound unu}.
We have that
\begin{equation}\label{eq:weighted trace for h - h_SicapOm ma valida per traslazioni}
\nr h -h_{\Si\cap\Om}  \nr^2_{L^2(\Ga_0)} \le \frac{1}{\lbunu} \, \left( N \, \nr h - h_{\Si\cap\Om} \nr^2_{L^2(\Si\cap\Om)} + 2 \, \nr (-u)^{ \frac{1}{2} } \na h \nr^2_{L^2(\Si\cap\Om)} \right) .
\end{equation}

Moreover, we have that
\begin{equation}\label{eq:weighted trace for nah senza applicazione Poincare}
\nr \na h \nr^2_{L^2(\Ga_0)} \le \frac{1}{\lbunu} \left\lbrace  N \int_{\Si\cap\Om} |\na h|^2 dx + 2 \int_{\Si\cap\Om} (-u) |\na^2 h|^2 dx + 2 \int_{\Ga_1} u \langle  \na^2 u \na u , \nu \rangle dS_x  \right\rbrace .
\end{equation}
\end{lem}
\begin{rem}
{\rm 
As noticed in \eqref{eq:disuguaglianza con u gratis per convex}, the convexity of the cone gives that the remainder term $\int_{\Ga_1} u \langle  \na^2 u \na u , \nu \rangle dS_x$ in \eqref{eq:weighted trace for nah senza applicazione Poincare} is non-negative. We also notice that if $k=0$ or $k=1$, the reminder term vanishes.
}
\end{rem}
\begin{proof}[Proof of Lemma \ref{lem:weighted trace inequality ad hoc}]
Applying \eqref{eq:nuovaidentityforidentityfortraceinequality} with $v:=h-h_{\Si\cap\Om}$ and using that $h_\nu=0$ on $\Ga_1$ (which holds true in light of \eqref{eq:inner product z in cone} and the Neumann condition in \eqref{eq:problem torsion})
give that
\begin{equation*}
	\int_{\Ga_0} (h-h_{\Si\cap\Om})^2 u_{\nu} \, dS_x = N \int_{\Si\cap\Om} (h - h_{\Si\cap\Om})^2 dx + 2 \int_{\Si\cap\Om} (-u) |\na h |^2 dx ,
\end{equation*}
which coupled with \eqref{def:lower bound unu} gives \eqref{eq:weighted trace for h - h_SicapOm ma valida per traslazioni}.

Let us now prove \eqref{eq:weighted trace for nah senza applicazione Poincare}.
Applying \eqref{eq:identityfortraceinequality} with $v:=h$ gives that
\begin{equation}\label{eq:0 intermedia per weighted trace}
\int_{\Ga_0 } |\na h|^2 u_{\nu} dS_x  
=
2 \int_{\Ga_1} u \langle  \na^2 h \na h , \nu \rangle dS_x + N \int_{\Si\cap\Om} |\na h|^2 dx + 2 \int_{\Si\cap\Om} (-u) |\na^2 h|^2 dx.
\end{equation}
Notice that $\langle  \na^2 h \na h , \nu \rangle = \langle  \left( I -\na^2 u \right)\left( \left( x-z \right) - \na u \right) , \nu \rangle , $
and hence we have that
\begin{equation}\label{eq:1 intermedia per weighted trace}
\langle  \na^2 h \na h , \nu \rangle = \langle  \na^2 u \na u , \nu \rangle \quad \text{on } \Ga_1 .
\end{equation}
In the last identity, we used that
\begin{equation}\label{eq:2 intermedia per weighted trace}
\langle  \na^2 u (x-z) , \nu \rangle = 0 \quad \text{ on } \Ga_1
\end{equation}
and, as before, that $h_\nu  = \langle \left( x-z \right) - \na u , \nu \rangle = 0 $ on $\Ga_1$.
Identity \eqref{eq:2 intermedia per weighted trace} follows by using \eqref{eq:inner product z in cone} and the Neumann condition in \eqref{eq:problem torsion} (see, e.g., \cite[formula after (3.3)]{PT}).
Using \eqref{eq:1 intermedia per weighted trace} into \eqref{eq:0 intermedia per weighted trace} gives that
\begin{equation*}
\int_{\Ga_0 } |\na h|^2 u_{\nu} dS_x  = 2 \int_{\Ga_1} u \langle  \na^2 u \na u , \nu \rangle dS_x + N \int_{\Si\cap\Om} |\na h|^2 dx + 2 \int_{\Si\cap\Om} (-u) |\na^2 h|^2 dx ,
\end{equation*}
which coupled with \eqref{def:lower bound unu} gives \eqref{eq:weighted trace for nah senza applicazione Poincare}.
\end{proof}
\begin{rem}\label{rem:h-la ok per ogni la}
{\rm 
It is clear from the proof that \eqref{eq:weighted trace for h - h_SicapOm ma valida per traslazioni} remains true if we replace $h$ with $h - \la$ for any $\la\in\RR$.
}
\end{rem}

The next Corollary shows that, whenever \eqref{def:lower bound unu} is available, Lemma \ref{lem:weighted trace inequality ad hoc} allows to remove the dependence on $\la_2(\Ga_0)$ in the constants of \eqref{eq:trace inequality intro} (with $v:=h$) and \eqref{eq:unweighted trace for h-hOm}. Moreover, being as $\langle \na^2 u \na u, \nu \rangle \le 0$ on $\Ga_1$ (by convexity of $\Si$), it also provides a suitable version of \eqref{eq:unweighted trace for nah} that removes the dependence on $\la_2(\Ga_0)$ at the cost of introducing a reminder term  ($\int_{\Ga_1} \left(- \langle  \na^2 u \na u , \nu \rangle \right) dS_x$), which will be re-absorbed later as it also appears in the left-hand side of the integral identities \eqref{H-fundamental} and \eqref{heintze-karcher-identity}.

\begin{cor}\label{cor:trace type unweigthed con m}
Let $\lbunu$ be the lower bound defined in \eqref{def:lower bound unu}. Then, the trace inequality
$$
\nr h \nr_{L^2(\Ga_0)} \le C \, \nr  h \nr_{W^{1,2}(\Si \cap \Om)},
$$
holds true for some positive constant $C$ that can be explicitly estimated by means of
$$
C \le \sqrt{ \frac{ \max\left\lbrace N, 2 \max\limits_{\ol{\Si\cap\Om}}(-u) \right\rbrace }{  \lbunu } } .
$$
Also, the constant $C$ appearing in \eqref{eq:unweighted trace for h-hOm} can be explicitly estimated by means of
\begin{equation}\label{eq:unweighted trace for h-hOm con lbunu}
C \le \sqrt{\frac{1}{\lbunu} \left( \frac{N}{ \mu_{2}(\Si\cap\Om)^2}  + 2 \max\limits_{\ol{\Si\cap\Om}} (-u) \right)},
\end{equation}
where $\mu_{2}(\Si\cap\Om)$ is the constant appearing in \eqref{eq:BoasStraube-poincare} (with $p:=2$).

Finally,
we have that
\begin{equation}\label{eq:trace na h unweighted usando lbunu}
\nr \na h \nr^2_{L^2(\Ga_0)} \le \Tilde{C} \, \left( \nr \na^2 h \nr^2_{L^2(\Si\cap\Om)} + \int_{\Ga_1} \left(- \langle  \na^2 u \na u , \nu \rangle \right) dS_x \right) ,
\end{equation}
with
$$
\Tilde{C} \le  \frac{1}{\lbunu} \left( N \, \La_{2}(k)^2  + 2 \max_{\ol{\Si\cap\Om}} (-u) \right) ,
$$
where $\La_{2}(k)$ is the constant defined in \eqref{def:La p al k} (with $p:=2$).
\end{cor}
\begin{proof}
The first inequality in the statement easily follows from \eqref{eq:weighted trace for h - h_SicapOm ma valida per traslazioni} using the trivial inequality
\begin{equation}\label{eq:prova trivial - u max -u}
-u(x) \le \max\limits_{\ol{\Si\cap\Om}}(-u)  \quad \text{for any } x \in \ol{\Si\cap\Om}.
\end{equation}

Putting together \eqref{eq:weighted trace for h - h_SicapOm ma valida per traslazioni}, \eqref{eq:BoasStraube-poincare} (with $v:=h$, $G:=\Si\cap\Om$, $p:=2$), and \eqref{eq:prova trivial - u max -u} we easily get that \eqref{eq:unweighted trace for h-hOm} holds true with $C$ as in \eqref{eq:unweighted trace for h-hOm con lbunu}.

Finally, \eqref{eq:trace na h unweighted usando lbunu} easily follows by putting together \eqref{eq:weighted trace for nah senza applicazione Poincare}, Lemma \ref{lem:Mixed pp Poincareaigradienti} (with $p:=2$), and \eqref{eq:prova trivial - u max -u}.
\end{proof}

\begin{rem}\label{rem:stima max -u}
{\rm
The parameter $\max\limits_{\ol{\Si\cap\Om}} (-u)$ can be explicitly estimated in terms of $N$ and the diameter $d_{\Si\cap\Om}$ only, as shown in Lemma \ref{lem:max stima con diametro}.
}
\end{rem}

\subsection{Lipschitz stability for Alexandrov's Soap Bubble Theorem in cones}

Now that the point $z$ has been chosen in \eqref{eq:choice of z 1of2} and \eqref{eq:choice of z 2of2}, we simply write $H_0$ to denote the value $H_0(z)$ defined in \eqref{def-H0}, that is, we set 
\begin{equation}\label{def:H0 IN SECTION LIPSCHITZ STABILITY}
H_0:=H_0(z)= \frac{1}{R} - \frac{\int_{\pa \Ga_0} \langle x-z,\mathbf{n}_x\rangle d \cH^{N-2}_x }{(N-1)N|\Si\cap\Om|} .
\end{equation}

\begin{lem}\label{lem:RIGHTHANDSIDE SBT}
Assume that the gluing condition $\int_{\pa \Ga_0} \langle x -z ,\mathbf{n}_x\rangle d \cH^{N-2}_x \le 0$ is verified for $z$ as in \eqref{eq:choice of z 1of2} and \eqref{eq:choice of z 2of2}.
Then, we have that
\begin{equation}\label{eq:RIGHTHANDSIDE SBT}
\max\left\lbrace 
\nr \na^2 h \nr_{L^2(\Si\cap\Om)} , 
\,\nr u_\nu - R \nr_{L^2 (\Ga_0 )} 
\right\rbrace
\le 
\ol{C}
\,
\nr H_0 - H \nr_{L^2(\Ga_0)} ,
\end{equation}
where
\begin{equation}\label{eq:C RIGHTHANDSIDE SBT}
\ol{C}:= \max \left\lbrace N-1 , \nr u_\nu \nr_{L^\infty (\Ga_0)} \right\rbrace^2(C+3) 
\end{equation}
and $C$ is the constant in \eqref{eq:unweighted trace for nah}.

If $\lbunu>0$ is  the lower bound defined \eqref{def:lower bound unu}, then we have that
\begin{equation}\label{eq:RIGHTHANDSIDE SBT withlbunu}
\max\left\lbrace 
\left[
\nr \na^2 h \nr^2_{L^2(\Si\cap\Om)} 
+  \int_{\Ga_1} \left(- \langle  \na^2 u \na u , \nu \rangle \right) dS_x
\right]^{1/2} , 
\,\nr u_\nu - R \nr_{L^2 (\Ga_0 )} 
\right\rbrace
\le 
\ol{C}
\,
\nr H_0 - H \nr_{L^2(\Ga_0)} ,
\end{equation}
with
\begin{equation}\label{eq:C RIGHTHANDSIDE SBT withlbunu}
\ol{C}:= \max \left\lbrace N-1 , \nr u_\nu \nr_{L^\infty (\Ga_0)} \right\rbrace^2 ( \Tilde{C}+3) ,
\end{equation}
where $\Tilde{C}$ is that appearing in \eqref{eq:trace na h unweighted usando lbunu}.
\end{lem}
\begin{rem}
{\rm
By Remark \ref{rem:casi H_0 non dipendente da z}, if we assume that $\ol{\Ga}_0$ and $\pa\Si$ intersect orthogonally then the gluing condition is always verified being as $\int_{\pa \Ga_0} \langle x -z ,\mathbf{n}_x\rangle d \cH^{N-2}_x = 0$, and we have $H_0:=H_0(z)=1/R$.
}
\end{rem}
\begin{proof}[Proof of Lemma \ref{lem:RIGHTHANDSIDE SBT}]
By using that all the summands in the left-hand side of \eqref{eq:NEWVERSION identity-SBT2} are non-negative, from \eqref{eq:NEWVERSION identity-SBT2} and \eqref{eq:hessiana h} we obtain that
\begin{equation*}
\begin{split}
\frac{1}{\max\left\lbrace N-1 , R \right\rbrace}
\max
&
\left\lbrace 
\nr \na^2 h \nr_{L^2(\Si\cap\Om)} , 
\,\nr u_\nu - R \nr_{L^2 (\Ga_0 )} 
\right\rbrace^2
\\
& \le
\int_{\Ga_0}(H_0-H)\,( u_\nu^2 - R^2 )\,dS_x+
R \, \int_{\Ga_0}(H_0-H)\, (R - q_\nu ) \, dS_x 
\\
& \le
\nr H_0 - H \nr_{L^2(\Ga_0)}
\left[ \nr u_\nu^2 - R^2 \nr_{L^2(\Ga_0)}
+
R \, \nr R - q_\nu \nr_{L^2(\Ga_0)} 
\right]
\\
& \le
\nr H_0 - H \nr_{L^2(\Ga_0)}
\left[ \left( \nr u_\nu \nr_{L^{\infty}(\Ga_0)} + 2 R \right) \nr u_\nu - R \nr_{L^2(\Ga_0)} + R \, \nr h_\nu \nr_{L^2(\Ga_0)} \right] 
,
\end{split}
\end{equation*}
where in the second inequality we used H\"older's inequality, and in the third inequality we used that
$\nr u_\nu^2 - R^2 \nr_{L^2(\Ga_0)} \le \left( \nr u_\nu \nr_{L^{\infty}(\Ga_0)} + R \right) \nr u_\nu - R \nr_{L^2(\Ga_0)}$ and $\nr R - q_\nu \nr_{L^2(\Ga_0)} = \nr u_\nu - R + h_\nu \nr_{L^2(\Ga_0)}  \le \nr u_\nu -R \nr_{L^2(\Ga_0)} + \nr h_\nu \nr_{L^2(\Ga_0)}  .$

Thus, by using $\nr h_\nu  \nr_{L^2(\Ga_0)} \le \nr \na h \nr_{L^2(\Ga_0)}$ together with \eqref{eq:unweighted trace for nah}, we deduce that
$$
\max\left\lbrace 
\nr \na^2 h \nr_{L^2(\Si\cap\Om)} , 
\,\nr u_\nu - R \nr_{L^2 (\Ga_0 )} 
\right\rbrace^2
\le 
\ol{C}
\,
\nr H_0 - H \nr_{L^2(\Ga_0)}
\,
\max\left\lbrace 
\nr \na^2 h \nr_{L^2(\Si\cap\Om)} , 
\,\nr u_\nu - R \nr_{L^2 (\Ga_0 )} 
\right\rbrace ,
$$
with 
$$ \ol{C} = \max\left\lbrace N-1 , R \right\rbrace 
\left[ \nr u_\nu \nr_{L^\infty (\Ga_0) } + (C+2)R \right],$$
where $C$ is the constant in \eqref{eq:unweighted trace for nah}.
Being as $R = (u_\nu)_{\Ga_0} \le \nr u_\nu \nr_{L^\infty(\Ga_0)} ,$
we can choose $\ol{C}$ as in \eqref{eq:C RIGHTHANDSIDE SBT}.
Thus, \eqref{eq:RIGHTHANDSIDE SBT} easily follows.

If $\lbunu$ is as in \eqref{def:lower bound unu}, we can repeat the argument above by replacing
$$
\max\left\lbrace 
\nr \na^2 h \nr_{L^2(\Si\cap\Om)} , 
\,\nr u_\nu - R \nr_{L^2 (\Ga_0 )} 
\right\rbrace
$$
with
$$
\max\left\lbrace 
\left[
\nr \na^2 h \nr^2_{L^2(\Si\cap\Om)} 
+  \int_{\Ga_1} \left(- \langle  \na^2 u \na u , \nu \rangle \right) dS_x
\right]^{1/2}
, 
\,\nr u_\nu - R \nr_{L^2 (\Ga_0 )} 
\right\rbrace ,
$$
and using \eqref{eq:trace na h unweighted usando lbunu} in place of \eqref{eq:unweighted trace for nah}. In this way we obtain \eqref{eq:RIGHTHANDSIDE SBT withlbunu} with $\ol{C}$ as in \eqref{eq:C RIGHTHANDSIDE SBT withlbunu}.
\end{proof}

\begin{thm}[Lipschitz stability for Alexandrov's Soap Bubble Theorem in terms of an $L^2$-psudodistance]
\label{thm:SBT stab Lipschitz costanti non geometriche}
Given the point $z$ defined in \eqref{eq:choice of z 1of2} and \eqref{eq:choice of z 2of2}, assume that $\int_{\pa \Ga_0} \langle x -z ,\mathbf{n}_x\rangle d \cH^{N-2}_x \le 0$. Then, we have that
\begin{equation}\label{eq:Stab Lipschitz SBT IN SECTION}
\nr |x-z| - R \nr_{L^2(\Ga_0)} \le \widehat{C} \, \nr H_0 - H \nr_{L^2(\Ga_0)} ,
\end{equation}
where
$$
\widehat{C}:= \max\left\lbrace C , 1 \right\rbrace \ol{C}
$$
with $\ol{C}$ as in \eqref{eq:C RIGHTHANDSIDE SBT} and $C$ as in \eqref{eq:unweighted trace for nah}; 
that is, $\widehat{C}$ is an explicit constant only depending on $N$, $\la_2(\Ga_0)$, $\nr u_\nu \nr_{L^\infty(\Ga_0)}$, and $\La_{2}(k)$.

If \eqref{def:lower bound unu} is in force, then the dependence on $\la_2(\Ga_0)$ can be dropped, and \eqref{eq:Stab Lipschitz SBT IN SECTION} holds true with
$$
\widehat{C}:= \max\left\lbrace \Tilde{C} , 1 \right\rbrace \ol{C} ,
$$
with $\ol{C}$ as in \eqref{eq:C RIGHTHANDSIDE SBT withlbunu} and $\Tilde{C}$ as in \eqref{eq:trace na h unweighted usando lbunu};
that is, $\widehat{C}$ is an explicit constant only depending on $N$, $\lbunu$, $\max\limits_{\ol{\Si\cap\Om}}(-u)$, $\nr u_\nu \nr_{L^\infty(\Ga_0)}$, and $\La_{2}(k)$.
\end{thm}
\begin{proof}
By using the triangle inequality we compute that
\begin{equation*}
 	\begin{split}
 		\nr |x-z| - R \nr_{L^2(\Ga_0)}
 		& \le \nr |x-z| - | \na u | \nr_{L^2(\Ga_0)} + \nr | \na u| - R  \nr_{L^2(\Ga_0)}
 		\\
 		& \le \nr  (x-z) -  \na u  \nr_{L^2(\Ga_0)} + \nr | \na u| - R  \nr_{L^2(\Ga_0)}
 		\\
 		& = \nr  \na h \nr_{L^2(\Ga_0)} + \nr u_\nu - R  \nr_{L^2(\Ga_0)} ,
 	\end{split}
\end{equation*}
from which by using \eqref{eq:unweighted trace for nah} (resp. \eqref{eq:trace na h unweighted usando lbunu}) we obtain that
$$
\nr |x-z| - R \nr_{L^2(\Ga_0)}
\le \max\left\lbrace C , 1 \right\rbrace
\max\left\lbrace 
\nr \na^2 h \nr^2_{L^2(\Si\cap\Om)} , 
\,\nr u_\nu - R \nr_{L^2 (\Ga_0 )} 
\right\rbrace ,
$$
with $C$ as in \eqref{eq:unweighted trace for nah} (resp.
$$ 
\nr |x-z| - R \nr_{L^2(\Ga_0)}
\le \max\left\lbrace \Tilde{C} , 1 \right\rbrace
\max\left\lbrace 
\left[
\nr \na^2 h \nr^2_{L^2(\Si\cap\Om)} 
+  \int_{\Ga_1} \left(- \langle  \na^2 u \na u , \nu \rangle \right) dS_x
\right]^{1/2} , 
\,\nr u_\nu - R \nr_{L^2 (\Ga_0 )} 
\right\rbrace
$$
with $\Tilde{C}$ as in \eqref{eq:trace na h unweighted usando lbunu}). The conclusion easily follows by \eqref{eq:RIGHTHANDSIDE SBT} (resp. \eqref{eq:RIGHTHANDSIDE SBT withlbunu}).
\end{proof}
\begin{rem}\label{rem:stima unu su Ga0}
{\rm
Using Lemma \ref{lem:upper bound gradient on Ga0}, the quantity $\nr u_\nu \nr_{L^\infty (\Ga_0)}$ appearing in the constants can be explicitly estimated in terms of $N$, the radius $\ul{r}_e$ of the uniform exterior sphere condition relative to $\Si$, and the diameter $d_{\Si\cap\Om}$ of $\Si\cap\Om$, which are pure geometrical parameters associated to $\Si\cap\Om$.
}
\end{rem}

We conclude this section with the	
\begin{proof}[Proof of Theorem \ref{thm:INTRO_SBT stab Lipschitz costanti non geometriche}]
The result immediately follows by Theorem \ref{thm:SBT stab Lipschitz costanti non geometriche}, recalling Remarks \ref{rem:stima max -u}, \ref{rem:stima unu su Ga0}, and \ref{rem:sfera interna relativa allora lbunu=ulri}. In particular, whenever $\Si\cap\Om$ satisfies the $\ul{r}_i$-uniform interior and $\ul{r}_e$-uniform exterior sphere conditions relative to $\Si$, putting together the second part of the statement of Theorem \ref{thm:SBT stab Lipschitz costanti non geometriche} and Remarks \ref{rem:stima max -u}, \ref{rem:stima unu su Ga0}, and \ref{rem:sfera interna relativa allora lbunu=ulri}, we find that \eqref{eq:Stab Lipschitz SBT IN SECTION} holds true with an explicit constant $\widehat{C}=\widehat{C}(N, \ul{r}_i, \ul{r_e}, d_{\Si\cap\Om} , \La_{2}(k))$.
\end{proof}

\subsection{Optimal stability for Heintze-Karcher's inequality in cones}
\begin{thm}[Optimal stability for Heintze-Karcher's inequality in terms of an $L^2$-psudodistance]
\label{thm:HK stability pseudodistance}
Given the point $z$ defined in \eqref{eq:choice of z 1of2} and \eqref{eq:choice of z 2of2}, there exists a radius $\rho >0$ such that
\begin{equation}\label{eq:STIMA finale HK}
\nr \frac{|x-z|^2 - \rho^2}{2} \nr_{L^2(\Ga_0)} \le \widehat{C} \, \left( \int_{\Ga_0} \frac{dS_x}{H} - N | \Si\cap\Om | \right)^{1/2} ,
\end{equation}
with
$$
\widehat{C}:= \sqrt{N-1} \,  \la_2(\Ga_0)^2  \left( 1 + \La_{2}(k)^{2} \right),
$$
where $\La_{2}(k)$ is the constant defined in \eqref{def:La p al k} 
(with $p:=2$) and $\la_2(\Ga_0)$ is the constant of the trace embedding $W^{1,2}(\Si\cap\Om) \hookrightarrow L^2(\Ga_0)$ (inequality \eqref{eq:trace inequality intro}).

If \eqref{def:lower bound unu} is in force, then the dependence on $\la_2(\Ga_0)$ can be dropped, and \eqref{eq:STIMA finale HK} remains true with
$$
\widehat{C} := \frac{\sqrt{N-1}}{\lbunu}   \left( N \, \La_{2}(k)^{2} + 2 \max\limits_{\ol{\Si\cap\Om}} (-u) \right) .
$$

\end{thm}
\begin{proof}
%
Since $-u>0$ in $\Si\cap\Om$ (e.g., by Lemma \ref{lem:relation u dist general}), we have that
$$h_{\Si\cap\Om}=\frac{1}{|\Si\cap\Om|}  \int_{\Si\cap\Om} \left( \frac{|x-z|^2}{2} - u \right) \, dx > 0  .$$
Set $\rho:=\sqrt{2 \, h_{\Si\cap\Om}}$
and notice that
$$
h - h_{\Si\cap\Om} = \frac{|x-z|^2 - \rho^2}{2} \quad \text{on } \Ga_0 .
$$

By putting together \eqref{eq:unweighted trace for h-hOm} and \eqref{eq:unweighted trace for nah}, we find that
\begin{equation}\label{eq:1 intermedia per HK stab}
\nr h - h_{\Si\cap\Om} \nr_{L^2( \Ga_0)} \le C \, \nr \na^2 h \nr_{L^2(\Si \cap \Om)}, 
\end{equation}
where $C$ is the product of the two constants in \eqref{eq:unweighted trace for h-hOm} and \eqref{eq:unweighted trace for nah}. Since all the summands at the left-hand side of \eqref{heintze-karcher-identity} are non-negative, from \eqref{heintze-karcher-identity} and \eqref{eq:hessiana h} we easily obtain that
\begin{equation}\label{eq:RIGHTHANDSIDE HK}
\nr \na^2 h \nr_{L^2(\Si \cap \Om)} \le \sqrt{N-1} \, \left( \int_{\Ga_0} \frac{dS_x}{H} - N | \Si\cap\Om | \right)^{1/2} ,
\end{equation} 
and the conclusion follows.

If in place of \eqref{eq:unweighted trace for h-hOm} and \eqref{eq:unweighted trace for nah} we use \eqref{eq:unweighted trace for h-hOm con lbunu} and \eqref{eq:trace na h unweighted usando lbunu}, in place of \eqref{eq:1 intermedia per HK stab} we get
\begin{equation*}
\nr h - h_{\Si\cap\Om} \nr_{L^2( \Ga_0)} \le C \, \left[
\nr \na^2 h \nr^2_{L^2(\Si\cap\Om)} 
+  \int_{\Ga_1} \left(- \langle  \na^2 u \na u , \nu \rangle \right) dS_x
\right]^{1/2} , 
\end{equation*}
where
the dependence in $C$ from the general trace embedding constant $\la_2(\Ga_0)$ has been replaced with the dependence on $\lbunu$
and $\max\limits_{\ol{\Si\cap\Om}}(-u)$.

The conclusion now follows from
$$
\left[
\nr \na^2 h \nr^2_{L^2(\Si\cap\Om)} 
+  \int_{\Ga_1} \left(- \langle  \na^2 u \na u , \nu \rangle \right) dS_x
\right]^{1/2}
\le 
\sqrt{N-1} \, \left( \int_{\Ga_0} \frac{dS_x}{H} - N | \Si\cap\Om | \right)^{1/2} ,
$$
which can be easily deduced by \eqref{heintze-karcher-identity} and \eqref{eq:hessiana h}.

\smallskip 

To compute the explicit value of the constant(s) $\widehat{C}$ in the statement we also used that $\mu_{2}(\Si\cap\Om)^{-1} \le \La_{2}(k)$, which immediately follows by the definition of $\La_{2}(k)$ in \eqref{def:La p al k}.
\end{proof}

As an immediate corollary, we provide the	
\begin{proof}[Proof of Theorem \ref{thm:INTRO_HK stability pseudodistance}]
The desired result immediately follows by Theorem \ref{thm:HK stability pseudodistance}, recalling Remarks \ref{rem:stima max -u} and \ref{rem:sfera interna relativa allora lbunu=ulri}.
\end{proof}

\begin{rem}\label{rem:HK stability scelta rho}
{\rm
We point out that, in the proof of Theorem \ref{thm:HK stability pseudodistance}, an alternative choice for $\rho$ could be $\rho:= \sqrt{2 \, h_{\Ga_0}}$, where
$$h_{\Ga_0} = \frac{1}{| \Ga_0 |} \int_{\Ga_0} \frac{|x-z|^2}{2} \, dS_x .$$
With this choice, the proof remains the same with the only difference that $h-h_{\Si\cap\Om}$ must be replaced by $h-h_{\Ga_0}$ and the application of \eqref{eq:unweighted trace for h-hOm} (resp. \eqref{eq:unweighted trace for h-hOm con lbunu}) must be replaced by the trace-type Poincar\'e inequality obtained by coupling \eqref{eq:trace inequality intro} (resp. the inequality
\begin{equation*}
	\nr h -h_{\Ga_0}  \nr^2_{L^2(\Ga_0)} \le \frac{1}{\lbunu} \, \left( N \, \nr h - h_{\Ga_0} \nr^2_{L^2(\Si\cap\Om)} + 2 \, \nr (-u)^{ \frac{1}{2} } \na h \nr^2_{L^2(\Si\cap\Om)} \right) ,
\end{equation*} 
which holds true in light of \eqref{eq:weighted trace for h - h_SicapOm ma valida per traslazioni} and Remark \ref{rem:h-la ok per ogni la}) with \eqref{eq:nuova prova Poincare zero trace HK}-\eqref{eq:nuovaprova costante la_2 traczero HK poincare} (resp. \eqref{eq:nuova prova Poincare zero trace HK}-\eqref{eq:nuovaprova costante SENZA la_2 traczero HK poincare}) provided in the next Lemma.
}
\end{rem}

\begin{lem}
We have that
\begin{equation}\label{eq:nuova prova Poincare zero trace HK}
\nr h - h_{\Ga_0} \nr_{L^2 (\Si\cap\Om)} \le c \, \nr  \na h  \nr_{L^2(\Si\cap\Om)} ,
\end{equation}
for some positive constant $c$ that can be explicitly estimated as follows:
\begin{equation}\label{eq:nuovaprova costante la_2 traczero HK poincare}
c \le \mu_{2}(\Si\cap\Om)^{-1} + \left( \frac{|\Si\cap\Om|}{|\Ga_0|} \right)^{\frac{1}{2}} \, \la_2 (\Ga_0) \left( 1 + \mu_{2}(\Si\cap\Om)^{-2} \right)^{1/2} ,
\end{equation}

If \eqref{def:lower bound unu} is in force, then
\begin{equation}\label{eq:nuovaprova costante SENZA la_2 traczero HK poincare}
c \le \mu_{2}(\Si\cap\Om)^{-1} + \left( \frac{|\Si\cap\Om|}{|\Ga_0|} \right)^{\frac{1}{2}} \,  \left( \frac{ N \, \mu_{2}(\Si\cap\Om)^{-1} + 2 \, \max\limits_{\ol{\Si\cap\Om}} (-u) }{\ul{m}} \right)^{1/2} .
\end{equation}
\end{lem}
\begin{proof}
The conclusion with \eqref{eq:nuovaprova costante la_2 traczero HK poincare} immediately follows from \eqref{eq:Poincare per zero trace} and \eqref{eq:stima costante Poincare zero trace r=p} (with $v:=h$, $r:=p:=2$, $G:=\Si\cap\Om$, $A:=\Ga_0$).

The conclusion with \eqref{eq:nuovaprova costante SENZA la_2 traczero HK poincare} can be proved by following the proof of Lemma \ref{lem:Poincare zero trace} (with $v:=h$, $r:=p:=2$, $G:=\Si\cap\Om$, $A:=\Ga_0$) with the only difference that now we can replace \eqref{eq:step altro embedding trace poincare} with (\eqref{eq:unweighted trace for h-hOm} with $C$ as in) \eqref{eq:unweighted trace for h-hOm con lbunu}, removing in this way the dependence on $\la_p(A)$.
\end{proof}

\section{Stability estimates in terms of \texorpdfstring{$\rho_e - \rho_i$}{rhoe-rhoi}}
\label{sec:Stronger stability rhoe - rhoi}

Given the point $z \in \RR^N$ defined in \eqref{eq:choice of z 1of2} and \eqref{eq:choice of z 2of2}, we set
\begin{equation}\label{def:rhoe and rhoi}
	\rho_e= \max_{x \in \ol{\Ga}_0}{|x-z|} \quad \rho_i=\min_{x \in \ol{\Ga}_0}{|x-z|} ,
\end{equation}
so that we have 
$$\Ga_0 \subseteq \left( \ol{B}_{\rho_e}(z) \setminus B_{\rho_i}(z) \right) \cap \Si .$$

Given $\theta \in \left( 0, \pi /2 \right]$ and $\ca >0$, we say that a set $E$ satisfies the \textit{$(\theta, \ca )$-uniform interior cone condition}, if for every $x \in \pa E$ there is a unit vector $\om=\om_x$ such that the cone with vertex at the origin, axis $\om$, opening width $\te$, and height $\ca$ defined by
$$
\cC_{\om}=\left\{ y \, : \, \langle y , \om \rangle > |y| \cos(\theta) , \, |y|< \ca \right\}
$$
is such that
\begin{equation*}
	w + \cC_{\om}  \subset E \ \text{ for every } \ w \in B_{\ca} (x) \cap \ol{E} .
\end{equation*}
Such a condition is equivalent to Lipschitz-regularity of the domain; more precisely, it is equivalent to the strong local Lipschitz property of Adams \cite[Pag 66]{Adams} and to the uniform Lipschitz regularity in \cite[Section III]{Ch} and \cite[Definition 2.1]{Ru}.

We stress that such assumption allows to avoid unnecessary technicalities. Nevertheless, the results presented in what follows could be obtained under more general assumptions (see, e.g., \cite[Remark 3.5]{MP7} and \cite{Pog2, MP4}).

\begin{rem}\label{rem:stima b_0}
	{\rm 
		If $E$ satisfies the $(\te, \ca)$-uniform interior cone condition, then it is a $b_0$-John domain and $b_0$ can be explicitly estimated in terms of $\te$, $\ca$, and $d_E$ (see \cite[Lemma A.2]{MP7}).
	}
\end{rem}


\begin{lem}\label{lem:relation_osc_rhoe-rhoi e Hess h}	
If $\Si\cap\Om$ satisfies the $(\te,\ca)$-uniform interior cone condition, then,
we have that
\begin{equation}\label{eq:rhoe-rhoi stima cone}
	\rho_e - \rho_i \le	C \, \left( \max_{ \ol{\Ga}_0} h - \min_{ \ol{\Ga}_0} h \right) ,
\end{equation}
where the positive constant $C$ can be explicitly estimated in terms of $N, \, \te, \, \ca$ only.
\end{lem}
\begin{proof}	
Since $B_{\rho_e}(z) \supseteq \Si\cap B_{\rho_e}(z) \supseteq \Si\cap\Om $ and $\Si\cap\Om$ contains at least a cone $\cC$ of height $\ca$ and width $\te$, then 
$$
\rho_e \ge \left( \frac{|\Si\cap\Om|}{|B|} \right)^{1/N} \ge \left( \frac{ |\cC|}{|B|} \right)^{1/N} ,
$$
and hence
$$
\max_{ \ol{\Ga}_0} h - \min_{ \ol{\Ga}_0} h = \frac{1}{2} \left( \left(\rho_e \right)^2 - \left(\rho_i\right)^2 \right) \ge
\frac{1}{2} \left( \frac{ |\cC|}{|B|} \right)^{1/N}  ( \rho_e - \rho_i) ,
$$
that gives \eqref{eq:rhoe-rhoi stima cone}.
\end{proof}

We now recall \cite[Lemma 3.4]{MP7}, which is a reformulation of results proved in \cite{MP6}. 

\begin{lem}[{\cite[Lemma 3.4]{MP7}}]
\label{lem:p>N + p<N and p=N} 
Let $1 \le p \le \infty$.
Let $E \subset \RR^N$ be a bounded domain satisfying the $(\te, \ca )$-uniform interior cone condition. 
	
\begin{enumerate}[(i)]
\item If $p>N$, then for any $f\in W^{1,p}(E)$, it holds that
\begin{equation*}
\max\limits_{\ol{ E}} f - \min\limits_{ \ol{ E}} f \le c \, \nr \na f \nr_{L^p(E)} ,
\end{equation*}
for some explicit constant $c=c (N, p, \te, \ca , d_E)$.
	
\item If $1\le p\le N $, for any $f\in W^{1,\hat{q} }(E)$ with $N< \hat{q} \le\infty$, it holds that
\begin{equation*}
\max\limits_{\ol{ E}} f - \min\limits_{ \ol{ E}} f \le 
c \,
\begin{cases}
\displaystyle \nr\na f\nr_{L^p(E)}^{\al_{p, \hat{q} }}\nr\na f\nr_{L^{\hat{q}}(E)}^{1-\al_{p, \hat{q} }} \quad & \text{ if } 1 \le p<N,
\vspace{3pt}
\\
\displaystyle \nr \na f \nr_{L^N(E) } \log\left(  e \, |E|^{\frac{1}{N} - \frac{1}{\hat{q}}} \frac{  \nr \na f \nr_{ L^{\hat{q}}(E)} }{\nr \na f \nr_{L^N(E)} }\right) \quad & \text{ if } p=N ,
\end{cases}
\end{equation*}
where
$$\al_{p, \hat{q} }=\frac{p\, ( \hat{q} -N)}{N\,( \hat{q} -p)}$$
and $ c $ is some explicit constant only depending on $N, p, \hat{q}, \te, \ca,$ and the diameter $d_E$ of $E$.
\end{enumerate}
\end{lem}
Applying the above result to $h$, we easily obtain the following.
\begin{lem}
\label{lem:intermediate h}
Let $\Si\cap\Om$ be a bounded domain satisfying the $(\te,\ca)$-uniform interior cone condition. Consider the function $h$ defined in \eqref{def: h armonica con q}. 
\par
There exists an explicit positive constant $C=C(N, p, \ca, \te, d_{\Si\cap\Om})$ such that
\begin{equation*}
\label{eq:stima-grad}
\rho_e - \rho_i  \le 
C \,\begin{cases}
\nr \na h \nr_{L^p(\Si\cap\Om)}  \ &\mbox{if $p>N$}; 
\\
\displaystyle \nr \na h \nr_{L^N(\Si\cap \Om)} \log \left(  e \, |\Si\cap\Om|^{\frac{1}{N}} \frac{ \nr \na h \nr_{L^\infty(\Si\cap\Om)}  }{ \nr \na h \nr_{L^N(\Si\cap\Om)} } \right)   \ &\mbox{if $p=N$}; 
\\
\nr \na h \nr_{L^\infty(\Si\cap\Om)}^{ (N-p)/N } \nr \na h \nr_{L^p (\Si\cap\Om)}^{p/N}   \ &\mbox{if $1\le p<N$.}
\end{cases}
\end{equation*}
\end{lem}
\begin{proof}
We apply Lemma \ref{lem:p>N + p<N and p=N} with $f:=h$ and $\hat{q}:=\infty$. By taking into account \eqref{eq:rhoe-rhoi stima cone}, the desired estimates easily follow.
\end{proof}

We will use the following strengthened version of Lemma \ref{lem:Mixed pp Poincareaigradienti} in the case where $p<N$.

\begin{lem}\label{lem:Mixed Strengthened Poincareaigradienti}
Let $z \in \RR^N$ be the point chosen as in \eqref{eq:choice of z 1of2}, \eqref{eq:choice of z 2of2}.
%
%
If $1 \le p<N$,
then we have that
\begin{equation*}
\nr \na h \nr_{L^{r}( \Si\cap\Om )} \le C \, \nr  \na^2 h  \nr_{L^p ( \Si\cap\Om)} ,
\end{equation*}
for some positive constant $C$ satisfying $C\le \La_{r,p}(k)$, where we have set
\begin{equation}\label{eq:NEWCONSTANT La_rp (k)}
\La_{r,p}(k):=
\begin{cases}
\mu_{r ,p}(\Si\cap\Om)^{-1} \quad & \text{if } k=0
\\
\eta_{r , p} (\Ga_1,\Si \cap \Om )^{-1} \quad & \text{if } k=N
\\
\max\left[ \mu_{r , p}(\Si\cap\Om)^{-1} , \, \eta_{ r , p} (\Ga_1,\Si \cap \Om )^{-1}  \right] \quad & \text{if } 1 \le k \le N-1 ,
\end{cases}
\end{equation}
where $\mu_{r, p}(\Si\cap\Om)$ and $\eta_{r, p} (\Ga_1,\Si \cap \Om )$ are those in \eqref{eq:John-Hurri-poincare} and Theorem \ref{thm:Strengthened Poincare new RN}.
\end{lem}
\begin{proof}
In light of \eqref{eq:gradienteproiettatosuspan nu}, we can apply \eqref{eq:Strengthened Poincare new RN} in Theorem \ref{thm:Strengthened Poincare new RN} with $G:=\Si\cap\Om$, $A:=\Ga_1$, and $\vV:= (h_1, \dots, h_k , 0 , \dots, 0) $ to get that
\begin{equation}\label{eq:dim prova Strengthened Poincaregradmixed}
\left( \sum_{i=1}^k \nr h_i \nr_{L^{r} (\Si\cap\Om)} \right)^{1/r}\le \eta_{r, p}(\Ga_1,\Si\cap\Om)^{-1} \, \left( \sum_{i=1}^k \sum_{j=1}^N \nr  h_{ij} \nr^p_{L^p (\Si\cap\Om)} \right)^{1/p}.
\end{equation}
In light of \eqref{eq:gradienteparte ortogonale a span nu}, we can apply \eqref{eq:John-Hurri-poincare} (with $G:=\Si\cap\Om$) to each first partial derivative $h_i$ of $h$, $i=k+1, \dots, N$.
Raising to the power of $r$ those inequalities and \eqref{eq:dim prova Strengthened Poincaregradmixed}, and then summing up, the conclusion easily follows by using \eqref{eq:ineq for Sobolev norm equivalence}.
\end{proof}
%
%

In the next theorem, we associate the difference $\rho_e - \rho_i$ with the $L^2$-norm of the hessian matrix $\na^2 h$.
\begin{thm}
\label{thm:SBT-W22-stability in cones}
Let $\Si\cap\Om$ be a bounded domain satisfying the $(\te,\ca)$-uniform interior cone condition. Let $z \in \RR^N$ be the point chosen as in \eqref{eq:choice of z 1of2}, \eqref{eq:choice of z 2of2}.
Then, there exists an explicit positive constant $C$ such that
\begin{equation}
\label{ineq:diff-radii-hessian}
\rho_e - \rho_i \le  
C \,\begin{cases}
\nr \na^2 h \nr_{ L^2 \Si\cap\Om)}  \ &\mbox{for $N=2, 3$}; \vspace{3pt} \\
\displaystyle \nr \na^2 h \nr_{L^2 (\Si\cap\Om)}  \max \left[ \log \left(  \frac{  e \, \nr \na h \nr_{L^\infty (\Si\cap\Om)} }{  \nr \na^2 h \nr_{L^2 (\Si\cap\Om)} } \right) , 1\right], \ &\mbox{for $N=4$}; \vspace{3pt}\\
\nr \na h \nr_{L^\infty (\Si\cap\Om)}^{ \frac{N-4}{N-2} }\,  \nr \na^2 h \nr_{L^2 ( \Si \cap \Om)}^{ \frac{2}{N-2}} ,   \ &\mbox{for $N\ge 5$.}
\end{cases}
\end{equation}
The constant $C$ can be explicitly estimated only in terms of $N, \ca , \te$, the constant $\eta_{p} (\Ga_1,\Si\cap\Om)$ (with $p=(\min\left\lbrace 3 ,N\right\rbrace +1)/2$),
and the diameter $d_{\Si\cap\Om}$.
\end{thm}

\begin{proof}
In what follows, we use the letter $C$ to denote a constant whose value may change line by line. All the constants $C$ can be explicitly computed (by following the steps of the proof) and estimated in terms of the parameters declared in the statement only.

In particular, in the following proof we are going to apply Lemma \ref{lem:Mixed Strengthened Poincareaigradienti} (with $r:=Np/(N-p)$, $p:=(\min\left\lbrace 3 ,N\right\rbrace +1)/2$), which introduces the constant $\La_{Np/(N-p), p}(k)$ defined in \eqref{eq:NEWCONSTANT La_rp (k)}. Notice that $\La_{Np/(N-p), p}(k)$ can be estimated in terms of $N, p , \ca, \te, d_{\Si\cap\Om}$, and, if $1 \le k \le N$, $\eta_{p}(\Ga_1 , \Si \cap \Om )$.
In fact, $\mu_{Np/(N-p), p}(\Si\cap\Om)$ (which appears in \eqref{eq:NEWCONSTANT La_rp (k)} if $0 \le k \le N-1$) can be estimated in terms of $N, p , \ca, \te, d_{\Si\cap\Om}$ by recalling Remarks \ref{rem:stime mu HS} and \ref{rem:stima b_0}.
Moreover, from the statement of Theorem \ref{thm:Strengthened Poincare new RN} (and recalling Remark \ref{rem:stime mu HS} and Remark \ref{rem:stima b_0}) we have that $\eta_{Np/(N-p), p}(\Ga_1,\Si\cap\Om)$ (which appears in \eqref{eq:NEWCONSTANT La_rp (k)} if $1 \le k \le N$) can be estimated in terms of $N, p , \ca, \te$ and $\eta_{p}(\Ga_1 , \Si \cap \Om )$.

(i) Lemma \ref{lem:intermediate h}  with $p := 6$ gives that
\begin{equation*}\label{eq:dimSBTW22N23-step 1}
\rho_e - \rho_i  \le 
C
\, \nr \na h \nr_{L^6(\Si\cap \Om)} ,
\end{equation*}
and the conclusion for $N=2,3$ follows by using that
$$
\nr \na h \nr_{L^6(\Si\cap \Om)}\le C \, \nr \na^2 h  \nr_{L^2 (\Si\cap\Om)} .
$$
The last inequality follows from Lemma \ref{lem:Mixed Strengthened Poincareaigradienti} (with $r:=6$, $p:=3/2$) and H\"older's inequality, for $N=2$, and directly from Lemma \ref{lem:Mixed Strengthened Poincareaigradienti} (with $r:=6$, $p:=2$), for $N=3$.
\par
(ii) Let $N=4$.
We use Lemma \ref{lem:intermediate h} with $p:=N=4$ and get:
$$
\rho_e - \rho_i  \le 
C \,\max \left[ \nr \na h \nr_{L^4 (\Si\cap\Om)}  \log \left(  \frac{  e \, \nr \na h \nr_{L^\infty (\Si\cap \Om)} }{  \nr \na h \nr_{L^4 (\Si\cap\Om)} } \right) , \nr \na h \nr_{L^4 (\Si\cap \Om)} \right].   
$$
Next, Lemma \ref{lem:Mixed Strengthened Poincareaigradienti} (with $r:=4$, $p:=2$) gives:
$$
\nr \na h \nr_{L^4 (\Si\cap\Om)} \le C  \, \nr \na^2 h \nr_{L^2 (\Si\cap\Om)}.
$$
Thus, the desired conclusion ensues by invoking the monotonicity of the function 
$t\mapsto t \max \{\log( A/t), 1\}$ for every $A>0$.

(iii) When $N\ge 5$, we can use Lemma \ref{lem:intermediate h} with $p:=2N/(N-2)$ and put it together with
Lemma \ref{lem:Mixed Strengthened Poincareaigradienti} (with $r:=2N/(N-2)$, $p:=2$). 
\end{proof}

By coupling the previous theorem with a suitable upper bound for $\nr \na h \nr_{L^{\infty}(\Si\cap\Om)}$, we easily obtain the following corollary.

\begin{cor}
\label{cor:SBT-W22-stability in cones with gradient bound}
Let $\Si\cap\Om$ be a bounded domain satisfying the $(\te,\ca)$-uniform interior cone condition. Let $z \in \RR^N$ be the point chosen as in \eqref{eq:choice of z 1of2}, \eqref{eq:choice of z 2of2}.
Then, there exists an explicit positive constant $C$ such that
\begin{equation}
\label{ineq:gradientbound+diff-radii-hessian}
\rho_e - \rho_i \le  
C \,\begin{cases}
\nr \na^2 h \nr_{ L^2 \Si\cap\Om)}  \ &\mbox{for $N=2, 3$}; \vspace{3pt} 
\\
\displaystyle \nr \na^2 h \nr_{L^2 (\Si\cap\Om)}  \max \left[ \log \left(  \frac{  e  }{  \nr \na^2 h \nr_{L^2 (\Si\cap\Om)} } \right) , 1\right], \ &\mbox{for $N=4$}; \vspace{3pt}
\\
\nr \na^2 h \nr_{L^2 ( \Si \cap \Om)}^{ \frac{2}{N-2}} ,   \ &\mbox{for $N\ge 5$.}
\end{cases}
\end{equation}
The constant $C$ can be explicitly estimated only in terms of $N, \ca , \te$, the constant $\eta_{p}(\Ga_1,\Si\cap\Om)$ (with $p=(\min\left\lbrace 3 ,N\right\rbrace +1)/2$),
and the diameter $d_{\Si\cap\Om}$; for $N=4$ and $N\ge 5$ $C$ also depends $\nr \na u \nr_{L^{\infty}(\Si\cap\Om)}$.
\end{cor}
\begin{proof}
The cases $N=2,3$ are the same as in Theorem \ref{thm:SBT-W22-stability in cones} and there is nothing to prove.

For $N =4$ and $N\ge 5$, we need to estimate $ \nr \na h \nr_{L^\infty(\Si\cap\Om)}$ and we proceed as follows.
Denoting with 
$$
\left( \na h \right)_{\Si\cap\Om} := \frac{1}{|\Si\cap\Om|} \int_{\Si\cap\Om} \na h(y) \, dy
$$
the mean value of $\na h$ over $\Si\cap\Om$, we compute
\begin{equation}\label{eq:step1gradientbound}
\nr \na h \nr_{L^{\infty}(\Si\cap\Om)} 
\le 
\nr \na h - \left( \na h \right)_{\Si\cap\Om}  \nr_{L^{\infty}(\Si\cap\Om)} + \left| \left( \na h \right)_{\Si\cap\Om} \right| 
\end{equation}
and note that
\begin{equation*}
\nr \na h - \left( \na h \right)_{\Si\cap\Om} \nr_{L^{\infty}(\Si\cap\Om)} 
\le 
\nr x - (x)_{\Si\cap\Om} \nr_{L^{\infty}(\Si\cap\Om)}
+
\nr \na u - \left( \na u \right)_{\Si\cap\Om} \nr_{L^{\infty}(\Si\cap\Om)}
\le 
d_{\Si\cap\Om} + 2 \nr \na u \nr_{L^{\infty}(\Si\cap\Om)} ,
\end{equation*}
so that we just need to estimate the second summand at the right-hand side of \eqref{eq:step1gradientbound}.
Recalling \eqref{eq:gradienteparte ortogonale a span nu}, an application of H\"older's inequality gives that
$$
\left|\left( \na h \right)_{\Si\cap\Om}  \right|  
\le
\frac{1}{\sqrt{|\Si\cap\Om|}} \left( \sum_{i=1}^k \nr h_i \nr_{L^2 (\Si\cap\Om) }^2 \right)^{1/2},
$$
and hence, by using \eqref{eq:dimprovaPoincaregradmixed} with $p:=2$,
$$
\left| \left( \na h \right)_{\Si\cap\Om}  \right| \le
\frac{\eta_{2}(\Ga_1,\Si\cap\Om)^{-1}}{\sqrt{|\Si\cap\Om|}} \, \nr \na^2 h \nr_{L^2 (\Si\cap\Om) } .
$$
Thus, by putting all together, \eqref{eq:step1gradientbound} gives that
\begin{equation}\label{eq:gradientboundh suitable for SBT}
\nr \na h \nr_{L^{\infty}(\Si\cap\Om)} 
\le 
d_{\Si\cap\Om} + 2 \, \nr \na u \nr_{L^{\infty}(\Si\cap\Om)}
+
\frac{\eta_{2}(\Ga_1,\Si\cap\Om)^{-1}}{\sqrt{|\Si\cap\Om|}} \, \nr \na^2 h \nr_{L^2 (\Si\cap\Om) }.
\end{equation}

If 
\begin{equation*}
\frac{\nr \na^2 h \nr_{L^2(\Si\cap\Om)}}{\sqrt{|\Si\cap\Om|}} \le 1 , 
\end{equation*}
putting together \eqref{ineq:diff-radii-hessian} and \eqref{eq:gradientboundh suitable for SBT} we find an explicit constant
$$
C^{*}= C^*(N, \ca , \te, d_{\Si\cap\Om},  \eta_{2}(\Ga_1,\Si\cap\Om), \nr \na u \nr_{L^{\infty}(\Si\cap\Om)} )
$$
such that \eqref{ineq:gradientbound+diff-radii-hessian} holds true with $C=C^*$.
On the other hand, if 
$$ \frac{\nr \na^2 h \nr_{L^2(\Si\cap\Om)}}{\sqrt{|\Si\cap\Om|}} \ge 1 $$
\eqref{ineq:gradientbound+diff-radii-hessian} trivially holds true with $C= \frac{d_{\Si\cap\Om}}{|\Si\cap\Om|^{\frac{1}{N-2}}}$ being as
$$
\rho_e - \rho_i \le d_{\Si\cap\Om} \le \frac{d_{\Si\cap\Om}}{|\Si\cap\Om|^{\frac{1}{N-2}}} \, \nr \na^2 h \nr_{L^2 ( \Si \cap \Om)}^{ \frac{2}{N-2}} .
$$

In any case, \eqref{ineq:gradientbound+diff-radii-hessian} always holds true with $C=\max \left\lbrace C^*, \, \frac{d_{\Si\cap\Om}}{|\Si\cap\Om|^{1/(N-2)}}  \right\rbrace$. The volume $|\Si\cap\Om|$ appearing here can be easily estimated in terms of $N, \te, \ca$; in fact, being as a cone $\cC$ of height $\ca$ and opening width $\te$ surely contained in $\Si\cap\Om$, we have that $| \Si\cap\Om | \ge |\cC|$.
\end{proof}

\begin{thm}[Stability in terms of $\rho_e - \rho_i$ for Alexandrov's Soap Bubble Theorem in cones] 
\label{thm:SBT stability rhoe rhoi}
Let $\Si\cap\Om$ be as in Setting \ref{Setting} and assume that $\Si$ is a convex cone and $\Si\cap\Om$ satisfies the $(\te,\ca)$-uniform interior cone condition.
Given the point $z$ defined in \eqref{eq:choice of z 1of2} and \eqref{eq:choice of z 2of2}, assume that $\int_{\pa \Ga_0} \langle x -z ,\mathbf{n}_x\rangle d \cH^{N-2}_x \le 0$. 
Denote by $H$ the mean curvature of $\Ga_0$ and let
$H_0$ be that defined in \eqref{def:H0 IN SECTION LIPSCHITZ STABILITY}.
Let $\rho_i$ and $\rho_e$ be defined by \eqref{def:rhoe and rhoi}.
Then, we have that
\begin{equation}
	\label{eq:stability SBT rhoei}
	\rho_e - \rho_i  \le 
	C \,
\begin{cases}
\nr H_0 - H \nr_{L^2(\Ga_0)},  \ &\mbox{if } N=2, \, 3, 
\\
\nr H_0 - H \nr_{L^2(\Ga_0)} \max \left[ \log \left( \frac{1  }{ \nr H_0 - H \nr_{L^2(\Ga_0)} } \right) , 1 \right],   \ &\mbox{if } N=4 ,
\\
\nr H_0 - H \nr^{\frac{2}{N-2}}_{L^2(\Ga_0)},  \ &\mbox{if } N \ge 5 .
\end{cases}
\end{equation}
The constant $C$ can be explicitly estimated only in terms of $N, \ca , \te$, the constant $\eta_{p}(\Ga_1,\Si \cap \Om )$ (with $p=(\min\left\lbrace 3 ,N\right\rbrace +1)/2$),
and the diameter $d_{\Si\cap\Om}$; for $N=4$ and $N\ge 5$ $C$ also depends on $\nr \na u \nr_{L^{\infty}(\Si\cap\Om)}$.
\end{thm}
\begin{proof}
The conclusion follows by putting together Lemma \ref{lem:RIGHTHANDSIDE SBT} and Corollary \ref{cor:SBT-W22-stability in cones with gradient bound}.
\end{proof}

\begin{thm}[Stability in terms of $\rho_e - \rho_i$ for Heintze-Karcher's inequality in cones] 
\label{thm:HK stability rhoe rhoi}
Let $\Si\cap\Om$ be as in Setting \ref{Setting} and assume that $\Si$ is a convex cone and $\Si\cap\Om$ satisfies
the $(\te,\ca)$-uniform interior cone condition. Denote by $H$ the mean curvature of $\Ga_0$.
Let $z \in \RR^N$ be the point chosen in \eqref{eq:choice of z 1of2}-\eqref{eq:choice of z 2of2} and let $\rho_i$ and $\rho_e$ be defined by \eqref{def:rhoe and rhoi}.
Then, we have that
\begin{equation}
	\label{eq:stability HK rhoei}
	\rho_e - \rho_i  \le 
	C \,
\begin{cases}
\left( \int_{\Ga_0} \frac{dS_x}{H} - N| \Si\cap\Om|\right)^{1/2},  \ &\mbox{if } N=2, \, 3, 
\\
\left( \int_{\Ga_0} \frac{dS_x}{H} - N| \Si\cap\Om|\right)^{1/2} \max \left[ \log \left( \frac{1  }{ \left( \int_{\Ga_0} \frac{dS_x}{H} - N| \Si\cap\Om|\right)^{1/2} } \right) , 1 \right],   \ &\mbox{if } N=4 ,
\\
\left( \int_{\Ga_0} \frac{dS_x}{H} - N| \Si\cap\Om|\right)^{\frac{1}{N-2}},  \ &\mbox{if } N \ge 5 .
\end{cases}
\end{equation}
The constant $C$ can be explicitly estimated only in terms of $N, \ca , \te$, the constant $\eta_{p}(\Ga_1,\Si\cap\Om)$ (with $p=(\min\left\lbrace 3 ,N\right\rbrace +1)/2$),
and the diameter $d_{\Si\cap\Om}$; for $N=4$ and $N\ge 5$ $C$ also depends on $\nr \na u \nr_{L^{\infty}(\Si\cap\Om)}$.
\end{thm}
\begin{proof}
The conclusion follows by putting together \eqref{eq:RIGHTHANDSIDE HK} and Corollary \ref{cor:SBT-W22-stability in cones with gradient bound}.
\end{proof}

\begin{rem}\label{rem:stima norma na u infinito SicapOm con N r_e diam}
{\rm
Whenever $\Si\cap\Om$ satisfies the $\ul{r}_e$-uniform exterior sphere condition relative to $\Si$, $\nr \na u \nr_{L^{\infty}(\Si\cap\Om)}$ can be explicitly estimated in terms of $N$, $d_{\Si\cap\Om}$ and $\ul{r}_e$ by putting together Lemma \ref{lem:upper bound gradient on Ga0} and Lemma \ref{lem:upper bound gradient on Ga0 = whole domain}.
}
\end{rem}

\begin{rem}
	{\rm
		An analogous of Remark \ref{rem:casi k=0 e k=1 evitando eta} holds for the results of the present section.
		
		(i) 
		As already mentioned, uniformity of $\eta_{p}(\Ga_1 , \Si \cap \Om )$ in the class of uniformly bounded and uniformly Lipschitz domains $G$ such that $\Ga_1 \subset \ol{G}$ may be obtained following the argument in \cite[Theorem 2]{BouCha}. 
		
		(ii) When $k=0$, the dependence on $\eta_{p}(\Ga_1,\Si\cap\Om)$ (with $p=(\min\left\lbrace 3 ,N\right\rbrace +1)/2$) should be removed in Theorem \ref{thm:SBT-W22-stability in cones}, Corollary \ref{cor:SBT-W22-stability in cones with gradient bound}, and Theorems \ref{thm:SBT stability rhoe rhoi}, \ref{thm:HK stability rhoe rhoi}. In fact,
		it does not come into play when $k=0$.

		(iii) In the case $k=1$, an estimate for $\eta_{p}(\Ga_1,\Si\cap\Om)$ and/or $\eta_{r, p}(\Ga_1,\Si\cap\Om)$ is provided by Lemma \ref{lem:Poincare zero trace} (with $v:=h_1$, $A:=\Ga_1$, $G:=\Si\cap\Om$).
}
\end{rem}

\section{Additional remarks}\label{sec: Additional remarks}

\subsection{Alternative choices for the point \texorpdfstring{$z$}{z}}
As already mentioned in the Introduction, different choices of the point $z$ lead to alternative stability results. For instance, we can avoid using \eqref{eq:Poincare new RN} (and \eqref{eq:Strengthened Poincare new RN}) and hence completely remove the dependence on $\eta_{p}( \Ga_1 ,  \Si \cap \Om )^{-1}$
for any $0 \le k \le N$, at the cost of leaving the point $z$ free to have non-zero components also in the directions spanned by $\nu$ on $\Ga_1$. A suitable choice to do this may be \eqref{eq:Alternative choice z}, that is,
\begin{equation*}
	z= \frac{1}{|\Si\cap\Om|} \int_{\Si\cap\Om} (x - \na u) \, dx,
\end{equation*}
as stated in Theorem \ref{thm:intro alternative z}. 
This is a somewhat modified center of mass of $\Si\cap\Om$ and may be equivalently re-written as
$$
z = \frac{1}{|\Si\cap\Om|} \left\lbrace \int_{\Si\cap\Om} x  \, dx - \int_{\Ga_1} u \, \nu  \, dS_x  \right\rbrace ,
$$
in light of the homogeneous Dirichlet condition on $\Ga_0$ in \eqref{eq:problem torsion}.
The next result contains Theorem \ref{thm:intro alternative z} and, in addition, the analogous modification of the bounds for $\rho_e -\rho_i$ with this new choice of $z$.

\begin{thm}\label{thm:final alternative z}
	Setting
	\begin{equation*}
		z= \frac{1}{|\Si\cap\Om|} \int_{\Si\cap\Om} (x -\na u) \, dx = \frac{1}{|\Si\cap\Om|} \left\lbrace \int_{\Si\cap\Om} x  \, dx - \int_{\Ga_1} u \, \nu  \, dS_x  \right\rbrace  ,
	\end{equation*}
	we have that:
	
	(i) Theorems \ref{thm:INTRO_SBT stab Lipschitz costanti non geometriche} and \ref{thm:INTRO_HK stability pseudodistance} remain true with $\La_{2}(k)$ replaced simply by $\mu_{2}(\Si\cap\Om)^{-1}$.
	In turn, $\mu_{2}(\Si\cap\Om)^{-1}$ may be explicitly estimated as in Remark \ref{rem:stime mu HS} (and Remark \ref{rem:stima b_0}).
	
	(ii) Theorems \ref{thm:SBT stability rhoe rhoi} and \ref{thm:HK stability rhoe rhoi} hold true with an explicit constant $C$ only depending on $N$, $\ca$, $\te$, $d_{\Si\cap\Om}$, and, if $N\ge 4$, $\nr \na u \nr_{L^\infty(\Si\cap\Om)}$.
	\end{thm}
\begin{rem}
	{\rm The dependence on  $\nr \na u \nr_{L^\infty(\Si\cap\Om)}$ in (ii) of Theorem \ref{thm:final alternative z} can be replaced by that on $\ul{r}_e$, whenever $\Si\cap\Om$ satisfies the $\ul{r}_e$-uniform exterior sphere condition relative to $\Si$. In fact, as noticed in Remark \ref{rem:stima norma na u infinito SicapOm con N r_e diam}, $\nr \na u \nr_{L^\infty(\Si\cap\Om)}$ can be explicitly estimated in terms of $N$, $d_{\Si\cap\Om}$, and $\ul{r}_e$.
	}
\end{rem}
 \begin{proof}[Proof of Theorem \ref{thm:final alternative z}]
 With the choice of $z$ in \eqref{eq:Alternative choice z}, the classical Poincar\'e inequality \eqref{eq:BoasStraube-poincare} and \eqref{eq:John-Hurri-poincare}
 %
 %
 can be applied to all the partial derivatives $h_i$ for $i=1,\dots,N$ (which with this choice of $z$ have zero mean on $\Si\cap\Om$ for any $i=1,\dots,N$). Thus, we obtain versions of Lemma \ref{lem:Mixed pp Poincareaigradienti} and \ref{lem:Mixed Strengthened Poincareaigradienti} where $\La_{p}(k)$ and $\La_{r,p}(k)$ are replaced simply by $\mu_{p}(\Si\cap\Om)^{-1}$ and $\mu_{r,p}(\Si\cap\Om)^{-1}$ (independently of $k$). 
 Finally, as already noticed, $\mu_{p}(\Si\cap\Om)^{-1}$ and $\mu_{r,p}(\Si\cap\Om)^{-1}$ can be estimated in terms of $\ca$, $\te$, $d_{\Si\cap\Om}$.
 The rest of the proofs remains unchanged.
 \end{proof}

\subsection{The classical case \texorpdfstring{$\Si=\RR^N$}{Sigma=R^N}}
\label{subsec:classical case Si=RN}
When $\Si=\RR^N$ the choice of $z$ in \eqref{eq:INTRO_choice of z 1of2}-\eqref{eq:INTRO_choice of z 2of2} agrees with that in \eqref{eq:Alternative choice z}, and reduces to the center of mass of $\Om$,
being as
$$
z= \frac{1}{|\Om|} \int_{\Om} (x -\na u) \, dx  = \frac{1}{|\Om|} \left[ \int_{\Om} x  \, dx - \int_{\Ga_0} u \nu \, dx \right] =\frac{1}{|\Om|}  \int_{\Om} x  \, dx .
$$
In the last identity, we used the homogeneous Dirichlet condition on $\Ga_0=\pa\Om$ in \eqref{eq:problem torsion}.
\begin{proof}[Proof of Theorem \ref{thm:classical INTRO in R^N}]
As already noticed in the Introduction the result immediately follows from Theorems \ref{thm:INTRO_SBT stab Lipschitz costanti non geometriche}, and \ref{thm:INTRO_HK stability pseudodistance}. Being as $\Si=\RR^N$ and hence $\Ga_1=\varnothing$, we have that $k=0$ and hence $\La_{2}(0)= \mu_{2}(\Si\cap\Om)^{-1}$. In turn, being as $\Om$ a $C^2$ domain, $\mu_{2}(\Si\cap\Om)^{-1}$ can be explicitly estimated in terms of $\ul{r}_i$ and $d_{\Si\cap\Om}$ only (see \cite[(iii) of Remark 2.4]{MP3}).

Finally, notice that the dependence on $\ul{r}_e$ in (i) appeared only to estimate (from above) $\nr u_\nu \nr_{L^{\infty}(\Ga_0)} = \nr \na u \nr_{L^{\infty}(\Si\cap \Om)} = \nr \na u \nr_{L^{\infty}(\Om)} $ by means of Lemma \ref{lem:upper bound gradient on Ga0}. When $\Si=\RR^N$ and $\Ga_0=\pa\Om$ is mean convex, $\nr \na u \nr_{L^{\infty}(\Om)}$ can be estimated in terms of $N$ and $\max_{\ol{\Om}}(-u)$ only (see, e.g, \cite[Lemma 2.2]{MP5}). In turn, $\max_{\ol{\Om}}(-u)$ can be easily estimated by $d_{\Om}$. We mention that, in this case ($\Si=\RR^N$), a finer bound for $\max_{\ol{\Om}}(-u)$ in terms of $N$ and the volume $|\Om|$ holds true thanks to a classical result on radially decreasing rearrangements due to Talenti (\cite{Talenti}).
\end{proof}

In the special case $\Si=\RR^N$, the bounds for $\rho_e -\rho_i$ in Theorems \ref{thm:SBT stability rhoe rhoi} and \ref{thm:HK stability rhoe rhoi} 
%
%
return those obtained in \cite{MP3}.
We point out that other choices for the point $z$ are admissible (see
\cite{MP, MP2, MP3}).

\section*{Acknowledgements}
The author is supported by the Discovery Early Career Researcher Award (DECRA) DE230100954 ``Partial Differential Equations: geometric aspects and applications'' and the Australian Laureate Fellowship FL190100081 ``Minimal surfaces, free boundaries and partial differential equations''.
The author is member of the Australian Mathematical Society (AustMS) and the Gruppo Nazionale Analisi Matematica
Probabilit\`a e Applicazioni (GNAMPA) of the Istituto Nazionale di Alta Matematica
(INdAM).


\end{document}